\documentclass[11pt,reqno]{amsart}
\usepackage{amsmath, amssymb, amsthm,amsfonts}
\usepackage{enumitem}
\usepackage{graphicx,color}
\usepackage{graphics}
\usepackage{comment}
\usepackage[OT2,T1]{fontenc}
\usepackage{tikz}
\DeclareSymbolFont{cyrletters}{OT2}{wncyr}{m}{n}
\DeclareMathSymbol{\Sha}{\mathalpha}{cyrletters}{"58}

\DeclareMathOperator{\SL}{SL}
\DeclareMathOperator{\DSL}{\widetilde{SL}}
\DeclareMathOperator{\GL}{GL}
\DeclareMathOperator{\Sh}{Sh}

\def\ov{\overline}

\newcommand{\Hup}{\mathbb{H}}

\newcommand{\ord}{{\rm ord}}

\newcommand{\A}{\mathbb{A}}
\newcommand{\As}{\mathbb{A}_\Q^{\times}}

\newcommand{\Qs}{\mathbb{Q}^{\times}}

\newcommand{\Q}{{\mathbb Q}}
\newcommand{\Z}{{\mathbb Z}}
\newcommand{\N}{{\mathbb N}}
\newcommand{\C}{{\mathbb C}}
\newcommand{\R}{{\mathbb R}}

\newcommand{\kro}[2]{\left( \frac{#1}{#2} \right) }

\newcommand{\hs}[1]{\left( #1 \right)_p}
\newcommand{\hst}[1]{\left( #1 \right)_2}
\newcommand{\hsq}[1]{\left( #1 \right)_q}
\newcommand{\hsb}[1]{\Big( #1 \Big)_p}
\newcommand{\mat}[4]{\left(\begin{matrix} #1 & #2 \\ #3 & #4 \end{matrix}\right)}

            \DeclareFontFamily{U}{wncy}{} 
            \DeclareFontShape{U}{wncy}{m}{n}{% 
               <5>wncyr5% 
               <6>wncyr6% 
               <7>wncyr7% 
               <8>wncyr8% 
               <9>wncyr9% 
               <10>wncyr10% 
               <11>wncyr10% 
               <12>wncyr6% 
               <14>wncyr7% 
               <17>wncyr8% 
               <20>wncyr10% 
               <25>wncyr10}{} 
\DeclareMathAlphabet{\cyr}{U}{wncy}{m}{n}

\begin{document}

\newtheorem{thm}{Theorem}
\newtheorem*{thms}{Theorem}
\newtheorem{lem}{Lemma}[section]
\newtheorem{prop}[lem]{Proposition}
\newtheorem{alg}[lem]{Algorithm}
\newtheorem{cor}[lem]{Corollary}
\newtheorem{conj}[lem]{Conjecture}
\newtheorem{remark}{Remark}

\theoremstyle{definition}

\newtheorem{ex}{Example}

\theoremstyle{remark}

\title[Minus space of half-integral weight]{Newforms of half-integral weight: the minus space counterpart 
%of $S_{k+1/2}^{+,\mathrm{new}}(\Gamma_0(4M))$
}

\author{Ehud Moshe Baruch}
\address{Department of Mathematics\\
         Technion\\
         Haifa , 32000\\
         Israel}
\email{embaruch@math.technion.ac.il}
\thanks{}
\author{Soma Purkait}
\address{Faculty of Mathematics\\
         Kyushu University\\
         Japan}
\email{somapurkait@gmail.com}
\date{August 2016}
\begin{abstract}
We define a subspace of the space of holomorphic modular forms of 
weight $k+1/2$ and level $4M$ where $M$ is odd and square-free. 
We show that this subspace is isomorphic under the 
Shimura-Niwa correspondence to the space of newforms of weight 
$2k$ and level $2M$ and that this is a Hecke isomorphism. 
The space we define is a proper subspace of the orthogonal 
complement of the Kohnen plus space if the Kohnen plus space is nonzero.
%Let $M$ be odd and square-free. There is a Hecke isomorphism
%between $S_{k+1/2}(\Gamma_0(4M))$ and $S_{2k}(\Gamma_0(2M))$ under which 
%Kohnen's new space $S_{k+1/2}^{+,\mathrm{new}}(4M)$ maps 
%Hecke isomorphically to $S_{2k}^{\mathrm{new}}(\Gamma_0(M))$. It is 
%natural to look for a subspace of 
%$S_{k+1/2}(\Gamma_0(4M))$ that maps Hecke isomorphically to 
%$S_{2k}^{\mathrm{new}}(\Gamma_0(2M))$. 
%We construct such a subspace and call it the ``minus space'' at 
%level $4M$. In an analogy with our work in the integral weight 
%setting we show that this minus space is characterized as a
%common eigenspace of certain finite set of operators coming from 
%the local Hecke algebra.
\end{abstract}
\maketitle

\section{Introduction}
Let $M$ be odd and square-free and $k$ be a positive integer. 
In a remarkable work, Niwa~\cite{Niwa} comparing the traces of 
Hecke operators proved existence of Hecke isomorphism 
between the space of holomorphic cusp forms of weight 
$k+1/2$ on the congruence subgroup $\Gamma_0(4M)$ and  
the space of weight $2k$ cusp forms on 
$\Gamma_0(2M)$.
In \cite{Kohnen1, Kohnen2} Kohnen defines plus space $S_{k+1/2}^{+}(4M)$, 
a subspace of the space  $S_{k+1/2}(\Gamma_0(4M))$, which is given by a certain 
Fourier coefficient condition. Kohnen considers a 
new space $S_{k+1/2}^{+, \mathrm{new}}(4M)$ inside his plus space and 
proves that this new subspace is Hecke isomorphic to 
$S_{2k}^{\mathrm{new}}(\Gamma_0(M))$, the space of newforms 
of weight $2k$ and level $M$, giving a newform theory of 
half-integral weight.

In the case $M=1$, the Kohnen plus space is given as an eigenspace of 
a certain Hecke operator considered by Niwa~\cite{Niwa}. 
Niwa's operator has two eigenvalues, one positive and 
one negative and the Kohnen plus space is the eigenspace of the positive 
eigenvalue. In this case the Kohnen plus space itself is  
the newspace $S_{k+1/2}^{+, \mathrm{new}}(4)$ and is 
isomorphic as a Hecke module to $S_{2k}(\Gamma_0(1))$. 
In the case $M>1$, the Kohnen plus space can be again given as an 
eigenspace of a certain operator \cite{Kohnen2} and we shall see that 
this operator at level $4M$ is indeed an analogue of Niwa's operator at 
level $4$. 

The Kohnen isomorphism uses the above mentioned Niwa isomorphism 
between the full spaces of weight $k+1/2$, level $4M$ and 
weight $2k$, level $2M$. From Kohnen's results, 
it is clear that the Niwa map sends the Kohnen 
plus space to a subspace of old forms inside 
$S_{2k}(\Gamma_0(2M))$. Kohnen's map can be viewed as a 
composition of the Niwa map with another map which sends an 
old form of level $2M$ from this image to its newform 
originator at level $M$.  
Our objective in this paper is to identify the subspace of 
half-integer weight forms which is sent by the Niwa map 
isomorphically onto the space of newforms for $\Gamma_0(2M)$. 
It is clear that this space will complement the Kohnen plus space. 
However, their sum will not give the whole space if the Kohnen plus space 
is nonzero. 

In fact, our results are motivated by the results of 
Loke and Savin~\cite{L-S} who interpreted the Kohnen 
plus space in representation theory language.
For the case $M=1$, Loke and Savin defined another space of half-integer 
weight forms which they showed is "conjugate" to the Kohnen plus space. 
This means that it is an image of the Kohnen plus space by an invertible 
Hecke operator and is isomorphic to the Kohnen plus space as a Hecke module. 
We show that the Kohnen plus space and the space considered by Loke and 
Savin do not intersect and that their sum maps isomorphically to the 
space of old forms $S_{2k}^{\mathrm{old}}(\Gamma_0(2))$ under the Niwa map. 
Hence it is natural to consider the orthogonal complement of the direct 
sum under the Petersson inner product, which we call the minus space. 
We show that the minus space is mapped isomorphically under the Niwa map 
to $S_{2k}^{\mathrm{new}}(\Gamma_0(2))$, the space of newforms on 
$\Gamma_0(2)$. We also characterize this space as a common eigenspace of 
two Hecke operators: the Niwa operator used by Kohnen to define the Kohnen plus 
space and a conjugate of the Niwa operator which was considered by Loke 
and Savin. The minus space is the intersection of the negative eigenspace 
of both operators. We normalize the negative eigenvalue to be $-1$ as in 
\cite{B-P}. 
Our description of the minus space at level $4$ is completely analogous to our 
description of the new space $S_{2k}^{\mathrm{new}}(\Gamma_0(2))$ 
in \cite{B-P} where we showed that 
$S_{2k}^{\mathrm{new}}(\Gamma_0(2))$ is the 
common $-1$ eigenspace of two Hecke operators. 
To summarize the case of $M=1$: We show that the space 
$S_{k+1/2}(\Gamma_0(4))$ decomposes into a direct sum of three 
spaces: The Kohnen plus space, a "conjugate" of the Kohnen plus space given 
by Loke and Savin and the minus space. The Kohnen plus space and its 
conjugate are indistinguishable as Hecke modules which is the same as 
saying that they are mapped under the Niwa map to "conjugate" spaces of 
old forms. The minus space is different as a Hecke module from both spaces.

In order to generalize this result for $M$ odd and square-free we 
consider a certain $p$-adic Hecke algebra for every prime $p$ dividing $M$. 
Our work follows that of Loke and Savin who studies a certain $2$-adic 
Hecke algebra which allowed them to give a representation theoretic 
interpretation of the Kohnen plus space and to introduce the operator which is 
conjugate of Niwa's operator and the space which is a ``conjugate'' to 
Kohnen's plus space. In our $p$-adic Hecke algebra, we consider two
$p$-adic operators that give rise to conjugate classical Hecke operators 
which when we use along with Niwa's operator and it's conjugate 
allow us to define our minus space at level $4M$. 
We give two descriptions of the minus space: 
One description as an orthogonal complement of a 
certain sum of subspaces and another description as a common $-1$ 
eigenspace of the Niwa operator, its conjugate and a pair of conjugate operators 
for each prime dividing $M$. This again is completely analogous to our 
description of the space of newforms of weight $2k$ for $\Gamma_0(2M)$ 
given in \cite[Theorem 1]{B-P}.
Our main result is that the minus space of weight $k+1/2$ at level $4M$ 
is isomorphic as a Hecke module to the space of newforms of weight $2k$ 
at level $2M$.

Our paper is divided as follows. We set up notation following 
Shimura's work on half-integral weight forms and recall Gelbart's theory 
of the double cover of $\SL_2(\Q_p)$. In Section $3$ 
we define a genuine Hecke algebra of the double cover of $\SL_2(\Q_p)$ 
modulo certain subgroups and a genuine central character and give its 
presentation using generators and relations. In particular we recall 
the work of Loke and Savin when $p=2$. 
In Section $4$ we translate 
certain elements in our $p$-adic Hecke algebra to classical Hecke operators 
on $S_{k+1/2}(\Gamma_0(4M))$. We obtain two classical operators, 
$\widetilde{Q}_p$ with eigenvalues $p$ and $-1$ and an involution
$\widetilde{W}_{p^2}$. We further consider $\widetilde{Q}'_p$ 
which is conjugate of $\widetilde{Q}_p$ by $\widetilde{W}_{p^2}$. 
We check that these operators are self-adjoint 
with respect to the Petersson inner product. We recall 
Kohnen's classical operator $Q$ on $S_{k+1/2}(\Gamma_0(4M))$ 
which he uses to describe his plus space. We show that his operator 
$Q$ comes from the $2$-adic Hecke algebra considered by Loke and Savin, 
and is an analogue of Niwa's operator at level $4$. Let 
$\widetilde{Q}'_2:= \kro{2}{2k+1} Q/\sqrt{2}$ and 
$\widetilde{Q}_2$ be conjugate of $\widetilde{Q}'_2$ by an 
involution $\widetilde{W}_{4}$. In Section $5$ we 
define our minus space $S_{k+1/2}^{-}(4M)$ and prove our main result:
\begin{thms}
Let $S^{-}_{k+1/2}(4M) \subseteq S_{k+1/2}(\Gamma_0(4M))$ 
be the common $-1$ eigenspace of operators 
$\widetilde{Q}_p$ and $\widetilde{Q}'_p$ for all primes $p$ dividing $2M$.
Then $S^{-}_{k+1/2}(4M)$ has a basis of eigenforms for 
all the operators $T_{q^2}$ where $q$ is a prime coprime to $2M$ and  
all the operators $U_{p^2}$ where $p$ is a prime dividing $2M$, and 
maps isomorphically under the Niwa map onto the 
space $S^{\mathrm{new}}_{2k}(\Gamma_0(2M))$. 
\end{thms}

\section{Preliminaries and Notation}
Let $k,\ N$ denote positive integers. We denote by 
$S_k(\Gamma_0(N))$ the space of holomorphic cusp forms of weight $k$ 
on the group $\Gamma_0(N)$. For each prime $p$ not dividing $N$ 
we have Hecke operator $T_p$ on $S_k(\Gamma_0(N))$ whose action 
on $q$-expansion can be given as follows: if 
$f=\sum_{n=1}^\infty a_nq^n \in S_k(\Gamma_0(N))$ then 
$T_p(f) = \sum_{n=1}^\infty (a_{pn} + p^{k-1}a_{n/p})q^n$.

For $m\in \N$, let $U_m$, $V_m$ be given by following action 
on any formal $q$-series: 
\[U_m(\sum_{n=1}^\infty a_nq^n) = \sum_{n=1}^\infty a_{mn}q^n, 
\quad V_m(\sum_{n=1}^\infty a_nq^n) = \sum_{n=1}^\infty a_{n}q^{mn}.\]
It is well-known that $V_m$ maps $S_k(\Gamma_0(N))$ to $S_k(\Gamma_0(mN))$
and if $m \mid N$ then $U_m$ is an operator on $S_k(\Gamma_0(N))$. 

We briefly recall the theory of half-integral weight modular forms \cite{Shimura}. 
Let $\mathcal{G}$ be the set of all ordered pairs $(\alpha, \phi(z))$ where 
$\alpha = \mat{a}{b}{c}{d} \in \GL_2(\R)^{+}$ and $\phi(z)$ is a holomorphic 
function on the upper half plane $\Hup$ such that 
$\phi(z)^2 = t\det(\alpha)^{-1/2}(cz+d)$ with $t$ in the unit circle 
$S^1:= \{z \in \C : |z| =1\}$. Then $\mathcal{G}$ is a 
group under the following operation:
\[ (\alpha, \phi(z)).(\beta, \psi(z))= (\alpha\beta, \phi(\beta z)\psi(z)).\]
Let $P:\mathcal{G}\rightarrow\GL_2^+(\R)$ be the 
homomorphism given by the projection map onto the first coordinate. 

Let $\zeta = (\alpha, \phi(z)) \in \mathcal{G}$. 
% Let $k \ge 1$ and $r=k+1/2$.
Define the slash operator $|[\zeta]_{k+1/2}$ on functions $f$ on $\Hup$ by 
$f|[\zeta]_{k+1/2}(z) = 
%f(\alpha z)(\phi(z))^{-r} = 
f(\alpha z)(\phi(z))^{-2k-1}$. 

Let $N$ be divisible by $4$ and $\alpha =\mat{a}{b}{c}{d} \in \Gamma_0(N)$. 
Define the automorphy factor
\[j(\alpha,z) = \varepsilon_d^{-1}\kro{c}{d}(cz+d)^{1/2},\]
where $\varepsilon_d = 1$ or $(-1)^{1/2}$ according as 
$d \equiv 1\text{ or }3 \pmod{4}$ and $\kro{c}{d}$ is as 
in Shimura's notation. 
%Note that this is precisely Gelbart's definition for $\chi(\alpha)$ (See page 24, 27).
Let
\[\Delta_0(N):=\{\alpha^* = (\alpha, j(\alpha,z)) \in 
\mathcal{G}\ |\ \alpha \in \Gamma_0(N)\} \leq \mathcal{G}.\] 
The map $L:\Gamma_0(N)\rightarrow \mathcal{G}$ given by 
$\alpha\mapsto\alpha^*$ defines an isomorphism onto $\Delta_0(N)$. Thus 
$P|_{\Delta_0(N)}$ and $L$ are inverse of each other. 
Denote by $\Delta_0(N)$ and $\Delta_1(N)$ respectively the images
of $\Gamma_0(N)$ and $\Gamma_1(N)$. 

Let $\chi$ be an even Dirichlet character modulo $N$. 
Let $S_{k +1/2}(\Gamma_0(N), \chi)$ be the space of cusp forms of weight 
$k+1/2$, level $N$ 
and character $\chi$ consisting of $f \in S_{k +1/2}(\Delta_1(N))$
such that $f|[\alpha^*]_{k +1/2}(z) =\chi(d)f(z)$ for 
all $\alpha \in \Gamma_0(N)$. In particular when $\chi$ is trivial 
$S_{k +1/2}(\Gamma_0(N),\chi) = S_{k +1/2}(\Delta_0(N))$.

Let $\xi$ be an element of $\mathcal{G}$ such that $\Delta_0(N)$ and 
$\xi^{-1}\Delta_0(N)\xi$ are commensurable. Then we have an 
operator $|[\Delta_0(N)\xi\Delta_0(N)]_{k/2}$ on $S_{k+1/2}(\Delta_0(N))$ defined by
\[f|[\Delta_0(N)\xi\Delta_0(N)]_{k+1/2} = \text{det}(\xi)^{(2k-3)/4}
\sum_v f|[\xi_v]_{k+1/2} \]
where $\Delta_0(N)\xi\Delta_0(N)=\bigcup_v\Delta_0(N)\xi_v$.

Let $\xi = (\mat{1}{0}{0}{p^2},p^{1/2})$. If
$p$ is a prime dividing $N$, then by 
\cite[Proposition 1.5]{Shimura},
\[f|[\Delta_0(N)\xi\Delta_0(N)]_{k+1/2} = 
p^{(2k-3)/2}\sum_{s=0}^{p^2-1}f|[(\mat{1}{s}{0}{p^2},p^{1/2})]_{k+1/2}(z),\]
thus if $f = \sum_{n=1}^\infty a_nq^n$ then 
$f|[\Delta_0(N)\xi\Delta_0(N)]_{k+1/2} = \sum_{n=1}^{\infty}a_{p^2n}q^n =
U_{p^2}(f)$. If $p$ is a prime such that $(p,N)=1$ then the Hecke 
operator $T_{p^2}$ is defined by 
\[T_{p^2}(f)=f|[\Delta_0(N)\xi\Delta_0(N)]_{k+1/2}.\]

%We shall need connection between automorphic forms and modular forms of 
%half-integral weight.
\vskip5mm

We shall be studying local Hecke algebra of the double cover of $\SL_2$. We 
next recall Gelbart's~\cite{Gelbart} description of the double cover. 
Let $p$ be any prime (including infinite place). The group $\SL_2(\Q_p)$ has 
a non-trivial central extension by $\mu_2 = \{\pm1\}$:
\begin{equation*}
\begin{split}
1\ \longrightarrow \ &\mu_2 \ \longrightarrow \ \DSL_2(\Q_p)\ 
\longrightarrow\ \SL_2(\Q_p)\ 
\longrightarrow\ 1\\
\{(I,&\pm1)\}\ \quad (g, \pm1) \quad \longmapsto \quad g \qquad \qquad
\end{split}
\end{equation*}
We use the $2$-cocycle defined below to determine the double cover 
$\DSL_2(\Q_p)$. 
For $g =\mat{a}{b}{c}{d} \in \SL_2(\Q_p)$, define
\[\tau(g) = \begin{cases}
         c & \text{if $c \ne 0$}\\
         d & \text{if $c = 0$}
        \end{cases}; \]
if $p=\infty$, set $s_p(g)=1$ while for a finite prime $p$
\[s_p(g) = \begin{cases}
         \hs{c,d} & \text{if $cd \ne 0$ and $\ord_p(c)$ is odd}\\
         1 & \text{else}.
        \end{cases}
\]
Define the $2$-cocycle $\sigma_p$ on $\SL_2(\Q_p)$ as follows:
\[\sigma_p(g, h) = \hs{\tau(gh)\tau(g),\tau(gh)\tau(h)}s_p(g)s_p(h)s_p(gh).\]
Then the double cover 
$\DSL_2(\Q_p)$ is the set $\SL_2(\Q_p) \times \mu_2$ with the group law:
\[(g,\ \epsilon_1)(h,\ \epsilon_2) = (gh,\ \epsilon_1\epsilon_2\sigma_p(g, h)).\]
For any subgroup $H$ of $\SL_2(\Q_p)$, we shall denote by $\ov{H}$ the complete 
inverse image of $H$ in $\DSL_2(\Q_p)$. 

We consider the following subgroups of $\SL_2(\Z_p)$:
\[K_0^p(p^n) =\left\{ \mat{a}{b}{c}{d} \in \SL_2(\Z_p)\ :\ c \in p^n\Z_p \right\},\]
\[K_1^p(p^n) =\left\{ \mat{a}{b}{c}{d} \in \SL_2(\Z_p)\ :\ c \in p^n\Z_p,\  
a \equiv 1 \pmod{p^n\Z_p} \right\}.\]
 
By \cite[Proposition 2.8]{Gelbart} for odd primes $p$, 
$\DSL_2(\Q_p)$ splits over $\SL_2(\Z_p)$. Thus 
$\ov{\SL_2(\Z_p)}$ is isomorphic to the direct product $\SL_2(\Z_p) \times \mu_2$ 
and $\ov{K_0^p(p)}$ is isomorphic to $K_0^p(p) \times \mu_2$. It 
follows from \cite[Corollary 2.13]{Gelbart} that the center $M_p$ of 
$\DSL_2(\Q_p)$ is simply the direct product $\{\pm I\} \times \mu_2$. 
Thus any genuine central character is given by a 
non-trivial character of $\mu_2 \times \mu_2$.  

However $\DSL_2(\Q_2)$ does not split over $\SL_2(\Z_2)$ but instead splits 
over the subgroup $K_1^2(4)$.
%\[K_1^2(4)=\left\{ \mat{a}{b}{c}{d} \in \SL_2(\Z_2)\ :\ 
%c \equiv 0,\ a \equiv 1 \pmod{4\Z_2}\right\}.\]
In this case the center $M_2$ of $\DSL_2(\Q_2)$ is a cyclic group of order $4$ 
generated by $(-I,1)$ and so a genuine central character is given by a sending 
$(-I,1)$ to a primitive fourth root of unity. 

We set up a few more notation.
For $s \in \Q_p$, $t \in \Qs_p$ let us define the following elements of 
$\SL_2(\Q_p)$:
\[x(s) =\mat{1}{s}{0}{1},\ y(s)=\mat{1}{0}{s}{1},\  
w(t) = \mat{0}{t}{-t^{-1}}{0},\  h(t)=\mat{t}{0}{0}{t^{-1}}.\]
Let $N=\{(x(s), \epsilon): s \in \Q_p,\ \epsilon = \pm 1\}$, 
$\bar{N}=\{(y(s), \epsilon): s \in \Q_p,\ \epsilon = \pm 1\}$ 
and $T=\{(h(t), \epsilon): t \in \Qs_p,\ \epsilon = \pm 1\}$
be the subgroups of $\DSL_2(\Q_p)$. Then the normalizer
$N_{\DSL_2(\Q_p)}(T)$ of $T$ in $\DSL_2(\Q_p)$ consists of 
elements $(h(t), \epsilon)$, $(w(t), \epsilon)$ for $t \in \Qs_p$. 
We note the following useful relations:
\begin{equation}\label{eq:eq1}
\begin{split}
(h(s),1)(h(t),1)=(h(st),\hs{s,t}),& \quad  (w(s),1)(w(t),1)=(h(-st^{-1}),\hs{s,t})\\
(h(s),1)(w(t),1)=(w(st),\hs{s,-t}),& \quad  (w(s),1)(h(t),1)=(w(st^{-1}),\hs{-s,t})
\end{split}
\end{equation}
For any subgroup $S$ of $\DSL_2(\Q_p)$, let 
$N^{S} = N \cap S$, $T^{S}= T \cap S$ and $\bar{N}^{S} = \bar{N} \cap S$.

\section{A Local Hecke Algebra of $\DSL_2(\Q_p)$}
Loke and Savin \cite{L-S} studied a genuine local Hecke algebra of 
$\DSL(\Q_2)$ corresponding to $\overline{K_0^2(4)}$ and a genuine central 
character, and gave 
an interpretation of Kohnen's plus space at level $4$ in terms 
of certain elements in this $2$-adic Hecke algebra. 
In this section we shall recall their work on the $2$-adic Hecke algebra. 
We shall then study genuine Iwahori Hecke algebra 
for $\DSL_2(\Q_p)$ corresponding to $\ov{K_0^p(p)}$ and a genuine 
character of $M_p$ for general odd prime $p$. 

Let $p$ be any finite prime and $C_c^{\infty}(\DSL_2(\Q_p))$ be 
the space of locally constant, compactly supported complex-valued functions 
on $\DSL_2(\Q_p)$. For an open compact subgroup $S$ of $\DSL_2(\Q_p)$ and 
a genuine character $\gamma$ of $S$, let
$H(S, \gamma)$ be the subalgebra of $C_c^{\infty}(\DSL_2(\Q_p))$ defined as 
follows: 
\[\{ f \in C_c^{\infty}(\DSL_2(\Q_p)) : 
f(\tilde{k}\tilde{g}\tilde{k'})=\overline{\gamma}(\tilde{k})\overline{\gamma}(\tilde{k'})f(\tilde{g}) 
\text{ for } \tilde{g} \in \DSL_2(\Q_p),\ \tilde{k},\ \tilde{k'} \in S\}.\]
%From the definition of $\gamma$ it is clear that $H(\gamma)$ consists of 
%genuine functions in $C_c^{\infty}(\DSL_2(\Q_p))$ that satisfy 
%\[f((k,\epsilon_1)(g,1)(k',\epsilon_2))=\epsilon_1\epsilon_2f((g,1)) 
%\text{ for all } g \in \SL_2(\Q_p),\ k,\ k' \in K_0.\]
Then $H(S, \gamma)$ is a $\C$-algebra under convolution which, for any
$f_1, f_2 \in H(S, \gamma)$, is defined by
\[ f_1*f_2(\tilde{h}) = 
\int_{\DSL_2(\Q_p)} f_1(\tilde{g})f_2(\tilde{g}^{-1}\tilde{h}) d\tilde{g} = 
\int_{\DSL_2(\Q_p)} f_1(\tilde{h}\tilde{g})f_2(\tilde{g}^{-1}) d\tilde{g}, \]
where $d\tilde{g}$ is the Haar measure on $\DSL_2(\Q_p)$ such that the measure 
of $S$ is one. We call $H(S, \gamma)$ the genuine Hecke algebra of 
$\DSL_2(\Q_p)$ with respect to $S$ and $\gamma$.

For certain $S$ and $\gamma$, we would like to describe 
the algebra $H(S, \gamma)$ using generators and relations. In order 
to do so we need to first compute the support of $H(S, \gamma)$. 
We say that $H(S, \gamma)$ is supported on $\tilde{g} \in \DSL_2(\Q_p)$ 
if there exists $f \in H(S, \gamma)$ such that $f(\tilde{g})\ne 0$. 
We shall use the following lemmas to compute the support. 
\begin{lem} \label{lem1:sup}
Let $S_{\tilde{g}} = S \cap 
\tilde{g} S \tilde{g}^{-1}$. Then
$H(S, \gamma)$ is supported on $\tilde{g}$ 
if and only if for every $\tilde{k} \in S_{\tilde{g}}$ we have 
$\gamma([\tilde{k}^{-1},\ \tilde{g}^{-1}]) = 1$.
\end{lem}
\begin{comment}
\begin{proof}
Suppose there exists $\tilde{k} \in K_{\tilde{g}}$ such that 
$\gamma(\tilde{k}) \ne \gamma(\tilde{g}^{-1}\tilde{k}\tilde{g})$.
Then for any $f \in H(\gamma)$, 
$\gamma(\tilde{k})f(\tilde{g}) = f(\tilde{k}\tilde{g}) =
 f(\tilde{g}\tilde{g}^{-1}\tilde{k}\tilde{g}) = 
 f(\tilde{g})\gamma(\tilde{g}^{-1}\tilde{k}\tilde{g})$ and so 
$f(\tilde{g})=0$. 

Next assume that for every $\tilde{k} \in K_{\tilde{g}}$ we have 
$\gamma(\tilde{k}) = \gamma(\tilde{g}^{-1}\tilde{k}\tilde{g})$. Define 
a function $S_g$ on $\DSL_2(\Q_p)$ by 
\[S_g(\tilde{h}) = \begin{cases}
            \gamma(\tilde{k_1})\gamma(\tilde{k_2}) & \text{if $\tilde{h} = \tilde{k_1}\tilde{g}\tilde{k_2}\in \ov{K}_0\tilde{g}\ov{K}_0$}\\
           0 & \text{else}.
           \end{cases}
\]
We show that $S_g$ is well defined and hence is an element of $H(\gamma)$ that 
is supported on $\tilde{g}$. Let 
$\tilde{k_1}\tilde{g}\tilde{k_2} = \tilde{k_3}\tilde{g}\tilde{k_4}$ 
in $\ov{K}_0\tilde{g}\ov{K}_0$. Then 
$\tilde{k_3}^{-1}\tilde{k_1} = 
\tilde{g}\tilde{k_4}\tilde{k_2}^{-1}\tilde{g}^{-1}$. 
So $\tilde{k_3}^{-1}\tilde{k_1} \in K_{\tilde{g}}$ and by assumption, 
$\gamma(\tilde{k_3}^{-1}\tilde{k_1}) = 
\gamma(\tilde{g}^{-1}\tilde{g}\tilde{k_4}\tilde{k_2}^{-1}\tilde{g}^{-1}\tilde{g})
=\gamma(\tilde{k_4}\tilde{k_2}^{-1})$. Hence 
$\gamma(\tilde{k_1})\gamma(\tilde{k_2}) = \gamma(\tilde{k_3})\gamma(\tilde{k_4})$ 
and $S_g$ is well defined.
\end{proof}
\end{comment}
\begin{lem} \label{lem2:sup}
The function $\alpha_{\tilde{g}} : S_{\tilde{g}} \longrightarrow \C$ 
defined by $\alpha_{\tilde{g}}(\tilde{k}) = 
\gamma([\tilde{k}^{-1},\ \tilde{g}^{-1}])$ is a character of $S_{\tilde{g}}$.
\end{lem}
\begin{comment}
\begin{proof}
Let $\tilde{k_1},\ \tilde{k_2} \in K_{\tilde{g}}$. Then
\[\alpha_{\tilde{g}}(\tilde{k_1}\tilde{k_2}) 
= \gamma((\tilde{k_1}\tilde{k_2})^{-1}\tilde{g}^{-1}
\tilde{k_1}\tilde{k_2}\tilde{g}) 
=  \gamma(\tilde{k_2}^{-1})
\gamma(\tilde{k_1}^{-1}\tilde{g}^{-1}\tilde{k_1}\tilde{g})
\gamma(\tilde{g}^{-1}\tilde{k_2}\tilde{g})\]
\[= \gamma(\tilde{k_2}^{-1}\tilde{g}^{-1}\tilde{k_2}\tilde{g})
\gamma(\tilde{k_1}^{-1}\tilde{g}^{-1}\tilde{k_1}\tilde{g})
=\alpha_{\tilde{g}}(\tilde{k_2})\alpha_{\tilde{g}}(\tilde{k_1}).\]
\end{proof}
\end{comment}
In order to compute the support using above lemmas we shall need certain 
results on cocycle multiplication. We note them in the appendix.

We also note the following well-known lemmas that will be useful in computing 
convolutions.
\begin{lem} \label{lem:rel1}
Let $f_1,\ f_2 \in H(S, \gamma)$ such that $f_1$ is supported on
$S\tilde{x}S=\bigcup_{i=1}^{m}\tilde{\alpha}_i S$ and 
$f_2$ is supported on 
$S\tilde{y}S=\bigcup_{j=1}^{n}\tilde{\beta}_i S$. Then
\[
f_1*f_2(\tilde{h})=
\sum_{i=1}^{m}f_1(\tilde{\alpha}_i)f_2(\tilde{\alpha}_{i}^{-1}\tilde{h})\]
where the nonzero summands are precisely for those $i$ for which there exist 
a $j$ such that $\tilde{h}\in \tilde{\alpha}_i \tilde{\beta}_j S$.
\end{lem}

For $\tilde{g} \in \DSL_2(\Q_p)$ let $\mu(\tilde{g})$ denotes the number 
of disjoint left (right) $S$ cosets in the decomposition of double 
coset $S\tilde{g}S$. 
\begin{lem}\label{lem:rel2}
Let $\tilde{g},\ \tilde{h} \in \DSL_2(\Q_p)$ be such that 
$\mu(\tilde{g})\mu(\tilde{h})= \mu(\tilde{g}\tilde{h})$. Let 
$f_1$ and $f_2 \in H(S, \gamma)$ are respectively supported on 
$S\tilde{g}S$ and $S\tilde{h}S$. 
Then $f_1*f_2$ is precisely supported on 
$S\tilde{g}\tilde{h}S$ and 
$f_1*f_2(\tilde{g}\tilde{h}) = f_1(\tilde{g})f_2(\tilde{h})$.
\end{lem}

%In the next two subsections we describe $H(S, \gamma)$ for certain choices 
%of $S$ and $\gamma$ separate the case of odd and even primes: 
%if $p=2$ w
%describe $H(S, \gamma)$ for 
%$S=\ov{K_0^2(4)}$ and 

\subsection{Local Hecke algebra of $\DSL_2(\Q_2)$ modulo $\ov{K_0^2(4)}$}
Let $S = \ov{K_0^2(4)}$ and $\gamma$ be a genuine character of $M_2$ determined 
by its value on $(-I,1)$. Since 
$\ov{K_0^2(4)}$ is the direct product $K_1^2(4) \times M$, 
%(see \cite{L-S}), 
we can extend $\gamma$ to a genuine character of $\overline{K_0^2(4)}$ by setting 
it trivial on $K_1^2(4)$. Loke and Savin described $H(S, \gamma)$ for the above 
choice of $S$ and $\gamma$ as follows.

Using relations in \eqref{eq:eq1}, extend $\gamma$ to the normalizer 
$N_{\DSL(\Q_2)}(T)$ by defining $\gamma((h(2^n),1)) = 1$ for all integers $n$ 
and $\gamma((w(1),1)) = (1 + \gamma((-I,1)))/\sqrt{2}$, a primitive 
$8$th root of unity. For $n\in \Z$, define the elements $\mathcal{T}_n$ and
$\mathcal{U}_n$ of $H(\ov{K_0^2(4)}, \gamma)$ supported respectively on 
the $\ov{K_0^2(4)}$ double cosets of $(h(2^n),1)$ and 
$(w(2^{-n}),1)$ such that
\[\mathcal{T}_n(\tilde{k}(h(2^n),1)\tilde{k'}) = 
\overline{\gamma}(\tilde{k})\overline{\gamma}((h(2^n),1))\overline{\gamma}(\tilde{k'}),\]
\[\mathcal{U}_n(\tilde{k}(w(2^{-n}),1)\tilde{k'}) = 
\overline{\gamma}(\tilde{k})\overline{\gamma}((w(2^{-n}),1))\overline{\gamma}(\tilde{k'})
\quad \text{for $\tilde{k}$, $\tilde{k'} \in \ov{K_0^2(4)}$}.\] 
\begin{thm}(Loke-Savin~\cite{L-S})\label{thm:LS} 
%We have following relations: 
For $m, n \in \Z$,
\begin{enumerate}
 \item If $mn \ge 0$ then $\mathcal{T}_m*\mathcal{T}_n =\mathcal{T}_{m+n}$.
 \item $\mathcal{U}_1 * \mathcal{T}_n = \mathcal{U}_{n+1}$ and 
 $\mathcal{T}_n * \mathcal{U}_1 = \mathcal{U}_{1-n}$.
 \item $\mathcal{U}_1 * \mathcal{U}_n = \mathcal{T}_{n-1}$ and
  $\mathcal{U}_n * \mathcal{U}_1 = \mathcal{T}_{1-n}$.
\end{enumerate}
The Hecke algebra $H(\ov{K_0^2(4)}, \gamma)$ is generated by $\mathcal{U}_0$ and $\mathcal{U}_1$ modulo 
relations 
$(\mathcal{U}_0-2\sqrt{2})(\mathcal{U}_0+\sqrt{2})=0$ 
and $\mathcal{U}_1^2 = 1$. 
\end{thm}
%In particular, note from the above that 
%$\mathcal{U}_2 = \mathcal{U}_1 * \mathcal{T}_1 = 
%\mathcal{U}_1 * \mathcal{U}_0 * \mathcal{U}_1$ and so 
%$(\mathcal{U}_2-2\sqrt{2})(\mathcal{U}_2+\sqrt{2})=0$.

\subsection{Iwahori Hecke Algebra of $\DSL_2(\Q_p)$ modulo $\ov{K_0^p(p)}$, $p$ odd.}
Fix an odd prime $p$. Let $S = \ov{K_0^p(p)}$. 
Let $\gamma$ be a character of $K_0^p(p)$ such that it is
trivial on $K_1^p(p)$.
Since $\frac{K_0^p(p)}{K_1^p(p)} \cong (\Z_p/p\Z_p)^\times$, 
%(by $\mat{a}{b}{c}{d}\mapsto d \pmod{p\Z_p}$), 
we can define  
$\gamma$ by a character of $(\Z/p\Z)^\times$. We shall use the same symbol 
$\gamma$ to denote a genuine character of $S$ by defining 
$\gamma(A, \epsilon) = \epsilon\gamma(A)$ for $A \in K_0^p(p)$.
We call $H(S, \gamma)$ with the above choice of 
$S$ and $\gamma$ to be the genuine Iwahori Hecke algebra 
of $\DSL_2(\Q_p)$ with central character $\gamma$.
Our main result in this subsection is to describe this Iwahori Hecke algebra 
using generators and relations when $\gamma$ is quadratic.
%either trivial or given by Kronecker symbol $\kro{\cdot}{p}$. 

In the rest of this subsection we shall denote $K_0^p(p)$ simply by $K_0$. 
We first note the following lemma.
\begin{lem}\label{lem:rep1}
A complete set of representatives for the double cosets of $\DSL_2(\Q_p)$ 
mod $\ov{K}_0$ are given by $(h(p^n),1),\ (w(p^{-n}),1)$ where $n$ varies 
over integers.
\end{lem}
We need to compute the support of $H(\ov{K}_0, \gamma)$. 
Fix an integer $n$. Let $A:=h(p^n)$ and $\tilde{A} := (A, \epsilon_1)$.
We shall show that $H(\ov{K}_0, \gamma)$ is supported on 
$\tilde{A}$. We have 
\begin{equation*}
\begin{split}
 K_{\tilde{A}}=\Big\{ \left(\mat{a}{b}{c}{d}, \pm1\right) \in 
 \ov{SL}_2(\Z_p) \ : \ & \ord_p(c) \ge \max\{-2n+1,1\}, \\
 & \ord_p(b) \ge \max\{2n,0\} \Big\}.
\end{split}
\end{equation*}
We check that $K_{\tilde{A}}$ has a triangular decomposition 
%(see Lemma~\ref{lem:char})
$K_{\tilde{A}} = N^{K_{\tilde{A}}} T^{K_{\tilde{A}}} \bar{N}^{K_{\tilde{A}}}$
where $T^{K_{\tilde{A}}} = T^{\ov{K}_0}$,
$N^{K_{\tilde{A}}} =\{ (x(s), \pm1) : \ord_p(s) \ge \max\{2n,0\}\}$ and
$\bar{N}^{K_{\tilde{A}}} =\{ (y(t), \pm1) : \ord_p(t) \ge \max\{-2n+1,1\}\}$.

By Lemma~\ref{lem1:sup} and~\ref{lem2:sup}, it is enough 
to check that the value of $\gamma$ on the commutator
$[(B,\ \epsilon_2)^{-1},\ (A,\ \epsilon_1)^{-1}]$ is $1$ for any
$(B,\ \epsilon_2)$ in $N^{K_{\tilde{A}}},\ T^{K_{\tilde{A}}}$ and 
$\bar{N}^{K_{\tilde{A}}}$ respectively. 

By Lemma~\ref{lem2:sigma}, for 
$B = (x(s),\epsilon_2) \in N^{K_{\tilde{A}}}$, we get
$[(B,\ \epsilon_2)^{-1},\ (A,\ \epsilon_1)^{-1}]=
(\mat{1}{sp^{-2n}-s}{0}{1},1)$; 
for $B = (h(u),\epsilon_2) \in T^{K_{\tilde{A}}}$, 
 $[(B,\ \epsilon_2)^{-1},\ (A,\ \epsilon_1)^{-1}]=
(I,1)$; and
for $B = (y(t),\epsilon_2) \in N^{K_{\tilde{A}}}$, 
we get that $[(B,\ \epsilon_2)^{-1},\ (A,\ \epsilon_1)^{-1}]=
(\mat{1}{0}{(p^{2n}-1)t}{1},1)$. Since each of them 
belong to $K_1^p(p)\times\{1\}$, we are done.
\vskip 2mm

Next let $A:=w(p^{-n})$. We show that $H(\ov{K}_0, \gamma)$ is supported on 
$\tilde{A}= (A, \epsilon_1)$ provided 
$\gamma(u^2)=1$ for all units $u$ in $\Z_p$. In this case we have 
\begin{equation*}
\begin{split}
 K_{\tilde{A}}=\Big\{ \left(\mat{a}{b}{c}{d}, \pm1\right) \in 
 \ov{SL}_2(\Z_p) \ : \ & \ord_p(c) \ge \max\{2n,1\}, \\
 & \ord_p(b) \ge \max\{-2n+1,0\} \Big\}
\end{split}
\end{equation*}
and $K_{\tilde{A}}$ has a triangular decomposition
$K_{\tilde{A}} = N^{K_{\tilde{A}}} T^{K_{\tilde{A}}} \bar{N}^{K_{\tilde{A}}}$
where $T^{K_{\tilde{A}}} = T^{\ov{K}_0}$,
$N^{K_{\tilde{A}}}=\{ (x(s), \pm1) : \ord_p(s) \ge \max\{-2n+1,0\}\}$ and
$\bar{N}^{K_{\tilde{A}}}=\{ (y(t), \pm1) : \ord_p(t) \ge \max\{2n,1\}\}$.

By Lemma~\ref{lem2:sigma}, for $B = (x(s),\epsilon_2) \in N^{K_{\tilde{A}}}$, we get 
$[(B,\ \epsilon_2)^{-1},\ (A,\ \epsilon_1)^{-1}]=
(\mat{1+s^2 p^{2n}}{-s}{-sp^{2n}}{1},1)$, so $\gamma$ takes value 
$1$ on this commutator. 
In the case $B = (y(t),\epsilon_2) \in N^{K_{\tilde{A}}}$,
we have
\[B^{-1}A^{-1}BA = \mat{1}{-p^{-2n}t}{-t}{1+p^{-2n}t^2}, \text{ where } 
\ord_p(t) \ge \max\{2n,1\},\]
so $s_p(B^{-1}A^{-1}BA)=1$ if either  $-t(1+p^{-2n}t^2)=0$ or 
$\ord_p(t)$ is even. Assume that $-t(1+p^{-2n}t^2) \ne 0$ and 
$\ord_p(t)$ is odd. Then $s_p(B^{-1}A^{-1}BA) = 
\hs{-t, 1+p^{-2n}t^2} = \hs{-p, 1+p^{-2n}t^2}$.
Let $u = 1+p^{-2n}t^2$. Since $\ord_p(t) \ge \max\{2n,1\}$, we have 
$u \equiv 1 \pmod{p \Z_p}$.
%Note that $\ord_p(p^{-2n}t^2) = \ord_p(p^{-2n}t) + \ord_p(t) \ge 1 $. Thus $u \equiv 1 \pmod{p \Z_p}$. 
Hence $s_p(B^{-1}A^{-1}BA) = \hs{-p, u} = \kro{u}{p} = 1$. 
So in this case also $\gamma$ takes value $1$. 

For $B = (h(u),\epsilon_2) \in T^{K_{\tilde{A}}}$, 
$[(B,\ \epsilon_2)^{-1},\ (A,\ \epsilon_1)^{-1}]=
(\mat{1/u^2}{0}{0}{u^2},1)$, so 
$\gamma([(B,\ \epsilon_2)^{-1},\ (A,\ \epsilon_1)^{-1}])=\gamma(u^2)$.

Thus if $\gamma(u^2)=1$ for all units $u$ in $\Z_p$ then  
$H(\ov{K}_0, \gamma)$ is supported on $(w(p^{-n}),\epsilon)$. In particular 
this holds if our choice of $\gamma$ is quadratic. Thus we have
\begin{prop}\label{prop1:sup} 
If $\gamma$ is a quadratic character then
$H(\ov{K}_0, \gamma)$ is supported on the double cosets of $\ov{K}_0$ represented 
by $(h(p^{n}), 1)$ and $(w(p^{-n}),1)$ as $n$ varies over integers.
\end{prop}
%\subsection{Decomposition of double cosets into union of single cosets and
%convolution formulae.}
%\subsection{Generators and relations.}
We now obtain the generators and relations in
$H(\ov{K}_0, \gamma)$ when $\gamma$ is quadratic.

We consider the character $\gamma$ of $\ov{K}_0$ to be 
the genuine character of the center $M_p$ and extend it to the 
normalizer group $N_{\DSL_2(\Q_p)}(T)$ as follows.

Let $\varepsilon_p = 1$ or $(-1)^{1/2}$ depending on whether 
$p \equiv 1\ \text{or}\ 3\pmod{4}$, i.e. $\varepsilon_p^2 = \kro{-1}{p}$.
%Note since $(h(p),1)(h(p^{-1}),1)=(I,\hs{p,p}) = (I, \kro{-1}{p})$, it is 
%reasonable to define $\gamma((h(p),1)) = \varepsilon_p$.
%Then for $\gamma$ to be a character we must have 
%\[\gamma((h(p^{-1}),1)) = \gamma((h(p),1)^{-1}(I, \kro{-1}{p})) = 
%\gamma((h(p),1))^{-1}\gamma((I, \kro{-1}{p})) \]
%\[= \varepsilon_p^{-1}\kro{-1}{p} = \varepsilon_p^{-1}\varepsilon_p^{2} 
%= \varepsilon_p.\]
%Similarly we must have
%\[\gamma((h(p^2),1))= \gamma((h(p^{-2}),1)) = \varepsilon_p^{2}\kro{-1}{p} = 1.\]
%We can now show by induction that we must have
%Define 
%\[\gamma((h(p^n),1)) = \begin{cases}
%                     1 & \text{ if $n$ is even}\\
%                    \varepsilon_p & \text{ if $n$ is odd}.
%                    \end{cases}\]
%use the relation $(h(t),1)=(h(p^n),1)(h(u),1)(I,\hs{p^n,u})$
Let $t=p^nu \in \Qs_p$ where $n \in \Z$ and 
$u$ is a unit in $\Z_p$. Define
\[\gamma((h(t),1))=
%\gamma((h(p^n),1))\gamma((h(u),1))\gamma((I,\hs{p^n,u}))=
\begin{cases}
1 & \text{ if $n$ is even}\\
\varepsilon_p\kro{u}{p} & \text{ if $n$ is odd}. 
\end{cases}\]
It is easy to see that $\gamma$ extends to a character of 
$T$. 

We now extend the character to the normalizer 
$N_{\DSL_2(\Q_p)}(T)$ by defining $\gamma((w(1),1)) = 1$ 
%$\gamma((w(1),1)) = \kro{-1}{p}$ 
and extend it using the relation
\[(w(t),1)=(h(t),1)(w(1),1)(I,\hs{-1,t^{-1}}).\] 
Thus for $t=p^nu$ as above, 
%we have $\hs{-1,t^{-1}} = 1$ for $n$ even and 
%$\hs{-1,t^{-1}} =  
%\hs{-1,pu}=\kro{-1}{p}$ if $n$ is odd, hence 
\[\gamma((w(t),1))=
\begin{cases}
1 & \text{ if $n$ is even}\\
\varepsilon_p\kro{-u}{p} & \text{ if $n$ is odd}. 
\end{cases}\]
%\begin{cases}
%\kro{-1}{p} & \text{ if $n$ is even}\\
%\varepsilon_p\kro{u}{p} & \text{ if $n$ is odd}. 
%\end{cases}
%{\bf Remark.}
%I initially defined $\gamma((w(1),1))$ to be $\kro{-1}{p}$ then 
%$\gamma((h(p^n),1))= \gamma((w(p^n),1)) = \varepsilon_p$ for $n$ odd and 
% we get $\mathcal{U}_0^2 = \kro{-1}{p}(p-1)\mathcal{U}_0 + p$.(see below)

We define the elements $\mathcal{T}_n$ and $\mathcal{U}_n$ of 
$H(\ov{K}_0, \gamma)$ supported 
respectively on the double cosets of $(h(p^n),1)$ and 
$(w(p^{-n}),1)$ as in the even prime case, i.e. 
\[\mathcal{T}_n(\tilde{k}(h(p^n),1)\tilde{k'}) = 
\overline{\gamma}(\tilde{k})\overline{\gamma}((h(p^n),1))\overline{\gamma}(\tilde{k'}),\]
\[\mathcal{U}_n(\tilde{k}(w(p^{-n}),1)\tilde{k'}) = 
\overline{\gamma}(\tilde{k})\overline{\gamma}((w(p^{-n}),1))\overline{\gamma}(\tilde{k'})
\quad \text{for $\tilde{k}$, $\tilde{k'} \in \ov{K}_0$}.\] 
Thus 
Proposition~\ref{prop1:sup} implies that 
$\mathcal{T}_n$ and $\mathcal{U}_n$ forms a $\C$-basis for $H(\ov{K}_0, \gamma)$ 
when $\gamma$ is quadratic.

In order to obtain relations amongst $\mathcal{T}_n$ and $\mathcal{U}_n$, 
we note the following lemma which can be obtained by using triangular 
decomposition of $K_0$.
\begin{lem}\label{lem:decomp}
1. For $n\geq 0$,
\[
K_0 h(p^n) K_0 = \bigcup_{s\in \Z_p / p^{2n}\Z_p}
x(s)h(p^n)K_0=\bigcup_{s\in \Z_p / p^{2n}\Z_p}K_{0}h(p^{n})y(ps).
\]
2. For $n \geq 1$,
\[
K_0 h(p^{-n}) K_0 = \bigcup_{s\in \Z_p / p^{2n}\Z_p}
y(ps)h(p^{-n})K_0=\bigcup_{s\in \Z_p / p^{2n}\Z_p}K_{0}h(p^{-n})x(s).
\]
3. For $n\geq 1$,
\[
K_0 w(p^{-n}) K_0 = \bigcup_{s\in \Z_p / p^{2n-1}\Z_p}
y(ps)w(p^{-n})K_0=\bigcup_{s\in \Z_p / p^{2n-1}\Z_p}K_{0}w(p^{-n})y(ps).
\]
4. For $n \geq 0$,
\[
K_0 w(p^{n}) K_0 = \bigcup_{s\in \Z_p / p^{2n+1}\Z_p}
x(s)w(p^{n})K_0=\bigcup_{s\in \Z_p / p^{2n+1}\Z_p}K_{0}w(p^{n})x(s).
\]
\end{lem}

\begin{prop}\label{prop2:rel1}
We have following relations:
\begin{enumerate}
 \item If $mn \ge 0$ then $\mathcal{T}_m*\mathcal{T}_n =\mathcal{T}_{m+n}$.
 \item For $n \ge 0$, $\mathcal{U}_1 * \mathcal{T}_n = \mathcal{U}_{n+1}$ and 
 $\mathcal{T}_{-n} * \mathcal{U}_1 = \mathcal{U}_{n+1}$.
 \item For $n \ge 0$, $\mathcal{U}_0 * \mathcal{T}_{-n} = \mathcal{U}_{-n}$ and 
 $\mathcal{T}_n * \mathcal{U}_0 = \mathcal{U}_{-n}$.
 \item For $n \ge 1$, $\mathcal{U}_0 * \mathcal{U}_n = \overline{\gamma}(-1) 
 \cdot \mathcal{T}_n$ and $\mathcal{U}_n * \mathcal{U}_0 = 
 \overline{\gamma}(-1) \cdot \mathcal{T}_{-n}$.
\end{enumerate}
\end{prop}
\begin{proof}
We prove $(1)$ and the second part of $(4)$. The rest are similar.

For $(1)$ let $mn \ge 0$. We may assume both $m,\ n \ge 0$.
It follows from Lemma~\ref{lem:decomp} and \ref{lem:rel2} that 
$\mathcal{T}_m*\mathcal{T}_n$ is precisely supported on the double coset 
$\ov{K}_0(h(p^{n+m}),1)\ov{K}_0$ and that 
\[\mathcal{T}_m*\mathcal{T}_n((h(p^{m}),1)(h(p^{n}),1)) = \mathcal{T}_m((h(p^{m}),1))\mathcal{T}_n((h(p^{n}),1)).\]
Let $m$ and $n$ both be even. 
Then $(h(p^{m}),1)(h(p^{n}),1) = (h(p^{n+m}),1)$ and so
\[\mathcal{T}_m*\mathcal{T}_n((h(p^{n+m}),1)) = \mathcal{T}_m((h(p^{m}),1))\mathcal{T}_n((h(p^{n}),1)) = \]
\[= \overline{\gamma}((h(p^{m}),1))\overline{\gamma}((h(p^{n}),1)) =1 = 
\mathcal{T}_{m+n}((h(p^{n+m}),1)),
\]
hence $\mathcal{T}_m*\mathcal{T}_n =\mathcal{T}_{m+n}$.
Next suppose both $m$ and $n$ are odd, so $m+n$ is even. 
Then $(h(p^{m}),1)(h(p^{n}),1) = (h(p^{n+m}),1)(I,\hs{p,p})$ and so 
\[\mathcal{T}_m*\mathcal{T}_n((h(p^{n+m}),1)) = \overline{\gamma}((I,\hs{p,p}))\mathcal{T}_m((h(p^{m}),1))\mathcal{T}_n((h(p^{n}),1)) = \]
\[= \kro{-1}{p}\overline{\gamma}((h(p^{m}),1))\overline{\gamma}((h(p^{n}),1)) 
=\kro{-1}{p}\overline{\varepsilon_p^2} = 1=\mathcal{T}_{m+n}((h(p^{n+m}),1)).\] 
Now suppose $m$ is odd and $n$ is even (or vice versa), so $m+n$ is odd. 
In this case $(h(p^{m}),1)(h(p^{n}),1) = (h(p^{n+m}),1)$ and so
\[\mathcal{T}_m*\mathcal{T}_n((h(p^{n+m}),1)) = 
%\overline{\gamma}((h(p^{m}),1))\overline{\gamma}((h(p^{n}),1)) = 
\overline{\varepsilon_p} =\mathcal{T}_{m+n}((h(p^{n+m}),1))\]
and we are done.

For $(4)$, let $n \ge 1$. As before using Lemma~\ref{lem:decomp} and 
\ref{lem:rel2} we know that  $\mathcal{U}_n * \mathcal{U}_0$ is 
supported on the double coset 
$\ov{K}_0(h(p^{-n}),1)\ov{K}_0$ and that 
\[\mathcal{U}_n * \mathcal{U}_0((w(p^{-n}),1)(w(1),1)) = \mathcal{U}_n((w(p^{-n}),1))\mathcal{U}_0((w(1),1)).\]
We have $(w(p^{-n}),1)(w(1),1) = (h(p^{-n}),1)(-I,\hs{p^{-n},-1})$ and so
\[\overline{\gamma}(-1)\ \mathcal{U}_n * \mathcal{U}_0((h(p^{-n}),1)) = \hs{p^{-n},-1}\mathcal{U}_n((w(p^{-n}),1))\mathcal{U}_0((w(1),1)) = \]
\[=\begin{cases}
    \kro{-1}{p}\overline{\varepsilon_p}\kro{-1}{p} = \overline{\varepsilon_p} & \text{ if $n$ is odd}\\
     1 & \text{ if $n$ is even}
   \end{cases}
=\mathcal{T}_{-n}((h(p^{-n}),1)),
\]
and thus $\mathcal{U}_n * \mathcal{U}_0 = \overline{\gamma}(-1) \cdot \mathcal{T}_{-n}$.
\end{proof}

We shall consider two choices for 
$\gamma$ as a character of $(\Z/p\Z)^*$, either $\gamma$ is trivial 
or $\gamma$ is given by the Kronecker symbol $\gamma = \kro{\cdot}{p}$. 
Then we have following proposition.
\begin{prop}\label{prop2:rel2}
We have following relations.
\begin{enumerate}
 \item $\mathcal{U}_0^2 = \begin{cases}
                           (p-1)\mathcal{U}_0 + p & \text{ if $\gamma$ is trivial}\\
                           \kro{-1}{p}p & \text{ if $\gamma = \kro{\cdot}{p}$}.
                          \end{cases}$
 \item $\mathcal{U}_1^2 = \begin{cases}
                           p & \text{ if $\gamma$ is trivial}\\
                           \varepsilon_p(p-1){U}_1 + \kro{-1}{p}p & \text{ if $\gamma = \kro{\cdot}{p}$}.
                          \end{cases}$ 
 \item If $\gamma$ is trivial, then 
 $\mathcal{T}_1*\mathcal{U}_1=p\ \mathcal{U}_0$ and 
 $\mathcal{T}_{-1} = 1/p \cdot \mathcal{U}_1 * \mathcal{T}_1 * \mathcal{U}_1$.
\end{enumerate}
\end{prop}
\begin{proof}
For $(1)$ we use Lemma~\ref{lem:rel1} to check that $\mathcal{U}_0*\mathcal{U}_0$ is at most
supported on the double cosets $\ov{K}_0$ and $\ov{K}_0(w(1),1)\ov{K}_0$.
Thus we need to only compute the values of $\mathcal{U}_0^2$ at $(I,1)$ and $(w(1),1)$. 
Using Lemma~\ref{lem:decomp} and \ref{lem:rel1} we have
\begin{equation*}
\begin{split}
\mathcal{U}_0*\mathcal{U}_0((I,1))&=
\sum_{s=0}^{p-1}\mathcal{U}_0((x(s),1)(w(1),1))\mathcal{U}_0((w(1),1)^{-1}(x(s),1)^{-1}) \\
&=\sum_{s=0}^{p-1}\mathcal{U}_0((w(1),1))\mathcal{U}_0((w(-1),1)(x(-s),1))\\
&=\sum_{s=0}^{p-1}\mathcal{U}_0((h(-1),1)(w(1),1)(x(-s),1))\\
&=\sum_{s=0}^{p-1} \overline{\gamma}(-1) = \begin{cases}
                                            p & \text{ if $\gamma$ is trivial}\\
                                           \kro{-1}{p}p & \text{ if $\gamma = \kro{\cdot}{p}$}.
                                           \end{cases}
\end{split}
\end{equation*}
Similarly, we get that $\mathcal{U}_0*\mathcal{U}_0((w(1),1))=$
\begin{equation*}
\begin{split}
&=\sum_{s=0}^{p-1}\mathcal{U}_0((x(s),1)(w(1),1))\mathcal{U}_0((w(1),1)^{-1}(x(s),1)^{-1}(w(1),1)) \\
&=\sum_{s=0}^{p-1}\mathcal{U}_0((\mat{0}{-1}{1}{-s},1)(w(1),1)))
=\sum_{s=0}^{p-1}\mathcal{U}_0((y(s),1))\\ 
&=\sum_{s=1}^{p-1}\mathcal{U}_0((y(s),1)) \ \ \text{since $\mathcal{U}_0((I,1))=0$.}\\
\end{split}
\end{equation*}
It is easy to check that for $1 \le s \le p-1$ 
\[(y(s),1) = (\mat{1}{1/s}{0}{1},1)(w(1),1)(\mat{-s}{-1}{0}{-1/s},1) \in 
\ov{K}_0(w(1),1)\ov{K}_0\]
and hence $\mathcal{U}_0*\mathcal{U}_0((w(1),1))=$
\[=
%\sum_{s=1}^{p-1}\mathcal{U}_0((w(1),1)) = 
\sum_{s=1}^{p-1}\overline{\gamma}(-1/s)
=\sum_{s=1}^{p-1}\gamma(s) = \begin{cases}
                                            p-1 & \text{ if $\gamma$ is trivial}\\
                                           \sum_{s=1}^{p-1}\kro{s}{p} =0 & \text{ if $\gamma = \kro{\cdot}{p}$}.
                                           \end{cases}\]
Thus if we write $\mathcal{U}_0^2 = c_1\mathcal{U}_0 + c_2$, we get that 
\[c_1 = \begin{cases}
         p-1 & \text{ if $\gamma$ trivial}\\
         0 & \text{ if $\gamma = \kro{\cdot}{p}$}
        \end{cases},
      \quad
 c_2 = \begin{cases}
         p & \text{ if $\gamma$ trivial}\\
         \kro{-1}{p}p & \text{ if $\gamma = \kro{\cdot}{p}$}.
        \end{cases}
      \]

Now we prove $(2)$. Again using Lemma~\ref{lem:rel1} we see that 
$\mathcal{U}_1*\mathcal{U}_1$ is at most supported on the double cosets 
$\ov{K}_0$ and $\ov{K}_0(w(p^{-1}),1)\ov{K}_0$.
So we need to find the values of $\mathcal{U}_1^2$ at $(I,1)$ and $(w(p^{-1}),1)$. 
Using Lemma~\ref{lem:decomp} and \ref{lem:rel1}, 
%we get that
\begin{equation*}
\begin{split}
\mathcal{U}_1*\mathcal{U}_1((I,1))&=
\sum_{s=0}^{p-1}\mathcal{U}_1((y(ps),1)(w(p^{-1}),1))\mathcal{U}_1((w(p^{-1}),1)^{-1}(y(ps),1)^{-1}) \\
&=\sum_{s=0}^{p-1}\mathcal{U}_1((w(p^{-1}),1))\mathcal{U}_1((w(-p^{-1}),1)(y(-ps),1))\\
&=\sum_{s=0}^{p-1}\overline{\varepsilon_p} \kro{-1}{p} \mathcal{U}_1((h(-1),\hs{-p,-1})(w(p^{-1}),1))\\
&=\sum_{s=0}^{p-1}\overline{\varepsilon_p}\kro{-1}{p}\gamma(-1)
\kro{-1}{p}\overline{\varepsilon_p}\kro{-1}{p} \\
&= \gamma(-1)p =\begin{cases}
p & \text{ if $\gamma$ trivial}\\
\kro{-1}{p}p & \text{ if $\gamma = \kro{\cdot}{p}$}.
\end{cases}\\
\end{split}
\end{equation*}
Finally, we have $\mathcal{U}_1*\mathcal{U}_1((w(p^{-1}),1))=$
\begin{equation*}
\begin{split}
&=\sum_{s=0}^{p-1}\mathcal{U}_1((y(ps),1)(w(p^{-1}),1))\mathcal{U}_1((w(-p^{-1}),1)
(y(-ps),1)(w(p^{-1}),1)) \\
&=\sum_{s=0}^{p-1}\overline{\varepsilon_p}\kro{-1}{p}\mathcal{U}_1((\mat{s}{-p^{-1}}{p}{0},\hs{p^2,-p^2s})
(w(p^{-1}),1)))\\
&=\sum_{s=0}^{p-1}\overline{\varepsilon_p}\kro{-1}{p}\mathcal{U}_1((x(s/p),\hs{p,-p})) = 
\sum_{s=1}^{p-1}\overline{\varepsilon_p}\kro{-1}{p}\mathcal{U}_1((x(s/p),1)).\\
\end{split}
\end{equation*}
Now we check that for $1 \le s \le p-1$ 
\[(x(s/p),1)(I,\kro{s}{p}) = 
(\mat{s}{0}{p}{1/s},1)(w(p^{-1}),1)(\mat{1}{0}{p/s}{1},1)\]
and so 
\begin{equation*}
\begin{split}
\mathcal{U}_1*\mathcal{U}_1((w(p^{-1}),1))&=
%\sum_{s=1}^{p-1}\overline{\varepsilon_p}\kro{-1}{p}
%\kro{s}{p})\gamma(1/s)\mathcal{U}_1((w(p^{-1}),1))\\ 
\sum_{s=1}^{p-1}\overline{\varepsilon_p}\kro{-1}{p}
\kro{s}{p}\overline{\gamma}(1/s)\overline{\varepsilon_p}\kro{-1}{p}\\
&= \sum_{s=1}^{p-1}\kro{-s}{p}\overline{\gamma}(1/s)\\
&=\begin{cases}
\sum_{s=1}^{p-1}\kro{-s}{p}= 0 & \text{ if $\gamma$ trivial}\\
\sum_{s=1}^{p-1}\kro{-s}{p}\kro{s^{-1}}{p} = \kro{-1}{p}(p-1) & 
\text{ if $\gamma = \kro{\cdot}{p}$}.
\end{cases}
\end{split}
\end{equation*}
Thus if we write $\mathcal{U}_1^2 = c_1\mathcal{U}_1 + c_2$, we get that 
\[c_1 = \begin{cases}
         0 & \text{ if $\gamma$ trivial}\\
         \varepsilon_p(p-1) & \text{ if $\gamma = \kro{\cdot}{p}$}
        \end{cases},
      \quad
 c_2 = \begin{cases}
         p & \text{ if $\gamma$ trivial}\\
         \kro{-1}{p}p & \text{ if $\gamma = \kro{\cdot}{p}$}.
        \end{cases}
      \] 

For $(3)$ let $\gamma$ be a trivial character.
From Proposition~\ref{prop2:rel1}$(4)$, we have 
$\mathcal{U}_0 * \mathcal{U}_1 = \mathcal{T}_1$. 
Right multiplication by $\mathcal{U}_1$ on both 
sides and using $(2)$ above gives  
$\mathcal{T}_1 * \mathcal{U}_1=p\ \mathcal{U}_0$. 
Further using the same proposition we get that 
 $\mathcal{T}_{-1} = \mathcal{U}_1 * \mathcal{U}_0 = 
 1/p \cdot \mathcal{U}_1 * \mathcal{T}_1 * \mathcal{U}_1$.
\end{proof}

\begin{remark}\label{rem:rem1}
We compare the p-adic operator $\mathcal{U}_1$ with Ueda's classical 
operator $Y_p$ \cite[Proposition 1.27]{Ueda} which satisfies a 
similar relation. 
In particular if we consider operator 
$\mathcal{U}'_1=\overline{\varepsilon_p}\mathcal{U}_1$, then 
in the case $\gamma$ is trivial we have 
\[(\mathcal{U}'_1)^2 = (\overline{\varepsilon_p}\mathcal{U}_1)^2 = 
{\varepsilon_p^2} p = \kro{-1}{p}p,\] 
while in the case $\gamma=\kro{\cdot}{p}$ we have
\[(\mathcal{U}'_1)^2 = (\overline{\varepsilon_p}\mathcal{U}_1)^2 = 
{\overline{\varepsilon_p}}^2\left(\varepsilon_p(p-1)\mathcal{U}_1 
+ p\kro{-1}{p}\right) 
= (p-1)\mathcal{U}'_1 + p.\]
Thus $\mathcal{U}'_1$ satisfies exactly the same relations as the  
operator $Y_p$. 
%Further when $\gamma$ is trivial Ueda defines an involution 
%$Z_p = \varepsilon_p p^{-1/2}Y_p$, so $Z_p$ should correspond to 
%$p^{-1/2}\mathcal{U}_1$. 
\end{remark}

\begin{thm}
The ``genuine'' Iwahori Hecke algebra $H(\ov{K_0^p(p)}, \gamma)$ for 
$\gamma$ trivial or $\kro{\cdot}{p}$ is generated  
as an $\C$-algebra by $\mathcal{U}_0$ and $\mathcal{U}_1$ with the 
defining relations given by above proposition.
\end{thm}
\begin{proof}
We let $\mathcal{A}$ be an abstract algebra generated by $\tilde{\mathcal{U}}_0$
and $\tilde{\mathcal{U}}_1$ with defining relations as $(1)$ and $(2)$ of 
Proposition~\ref{prop2:rel2}. We have a homomorphism
from $\tilde{A}$ to $H(\gamma)$ mapping $\tilde{\mathcal{U}}_0$ to $\mathcal{U}_0$ and 
$\tilde{\mathcal{U}}_1$ to $\mathcal{U}_1$. 
It follows from Proposition~\ref{prop2:rel1} that this homomorphism is onto. 
We let $M$ be the kernel of this homomorphism. 
Using relations $(1)$ and $(2)$
it follows that $M$ is a linear combination of words of the 
form $\tilde{\mathcal{U}}_0 \tilde{\mathcal{U}}_1\tilde{\mathcal{U}}_0....$ and 
$\tilde{\mathcal{U}}_1 \tilde{\mathcal{U}}_0\tilde{\mathcal{U}}_1....$. There are four possibilities
for the beginning and ending of such a word and each one is mapped by
the homomorphism to a different basis element (again using 
Proposition~\ref{prop2:rel1}). It follows that $M=0$.
\end{proof}

\section{Translation of adelic to classical.}\label{sec:trans}
In this section following Gelbart~\cite{Gelbart} and 
Waldspurger~\cite{Waldspurger} we review the connection between
automorphic forms on $\DSL_2(\A)$ and classical modular forms of 
half-integral weight. We use this connection to
translate certain elements in the p-adic Hecke algebra described in 
the previous section into classical operators
and thus obtain relations satisfied by these classical operators.
 
Let $\A=\A_\Q$ be the adele ring of $\Q$ and $\DSL_2(\A) = \SL_2(\A) \times \{\pm1\}$ with 
the group law: for $g=(g_\nu),\ h=(h_\nu) \in  \SL_2(\A)$ and 
$\epsilon_1,\ \epsilon_2 \in \{\pm1\}$
\[(g,\epsilon_1)(h,\epsilon_2)=(gh,\ \epsilon_1\epsilon_2\sigma(g,h) ),\ 
 \text{where}\ \sigma(g,h) = \prod_{\nu}\sigma_\nu(g_\nu,h_\nu).\]
The group $\DSL_2(\A)$ splits over $\SL_2(\Q)$ and the splitting is given by 
\[s_{\Q}:\SL_2(\Q) \longrightarrow \DSL_2(\A),\ g\mapsto(g, s_{\A}(g))\ 
\text{where}\ s_{\A}(g) = \prod_\nu s_{\nu}(g).\]
By \cite[Proposition 2.16]{Gelbart}, for $\alpha =\mat{a}{b}{c}{d} \in \Gamma_1(N)$, 
$s_{\A}(\alpha) = \kro{c}{d}_s$ unless $c=0$ in which case $s_{\A}(\alpha) = 1$. 
Here $\kro{c}{d}_s = \kro{c}{d}(c,d)_{\infty}$. 
\begin{lem}\label{lem:sA} 
Let $4 \mid N$.
For $\alpha =\mat{a}{b}{c}{d} \in \Gamma_0(N)$, we have 
\[s_{\A}(\alpha) = \begin{cases}
                   \kro{c}{d}_s\hst{c,d} & \text{if $c \ne 0$ and $\ord_2(c)$ is even}\\
                    \kro{c}{d}_s & \text{if $c \ne 0$ and $\ord_2(c)$ is odd}\\
                       1       &      \text{if $c = 0$}.
                  \end{cases}\]
\end{lem}
\begin{proof}
If $c=0$ then $s_{\nu}(\alpha) = 1 $ for all places $\nu$ and so $s_{\A}(\alpha)=1$. 

Suppose $c \ne 0$. Since $\alpha \in \Gamma_0(N)$ and $4 \mid N$, $d$ is odd and coprime 
to $c$. By definition, for any finite prime $q$, we have $s_q(\alpha) = \hsq{c,d}$ if 
$\ord_q(c)$ is odd and is $1$ else. Hence 
\[s_{\A}(\alpha) = \prod_{q\ \text{finite}}s_{q}(\alpha)= 
\prod_{\ord_q(c)\ \text{odd}}\hsq{c,d}.\]
It follows from the proof of \cite[Proposition 2.16]{Gelbart}  
(the proof only uses that $d$ is odd and coprime to $c$), that 
$\kro{c}{d}_s = \prod_{q\mid c}\hsq{c,d}$.
Now \[\prod_{\ord_q(c)\ \text{odd}}\hsq{c,d} = 
\prod_{q\mid c}\hsq{c,d} \prod_{\ord_q(c)\ \text{even} > 0}\hsq{c,d} 
= \kro{c}{d}_s \prod_{\ord_q(c)\ \text{even} > 0}\hsq{c,d}\]
So we just need to show that 
$\prod_{\ord_q(c)\ \text{even} > 0}\hsq{c,d}$ is $\hst{c,d}$ if 
$\ord_2(c)$ is even and is $1$ if $\ord_2(c)$ is odd 
(note that $\ord_2(c) \ge 2$). Let $p$ be any odd prime such that 
$\ord_p(c)$ is even and > $0$. Let $c = p^{2n}u$ where $u$ is 
unit in $\Z_p$. Then $\hs{c,d} = \hs{u,d} = 1$ as both 
$d,\ u$ are units in $\Z_p$. Hence we are done.
\end{proof}

For $\tilde{g} = (\mat{a}{b}{c}{d}, \epsilon) \in \DSL_2(\R)$ and $z\in \Hup$, define
\[\tilde{g}(z) = \frac{az+b}{cz+d} \qquad  \text{and} \qquad
J(\tilde{g}, z) = \epsilon(cz+d)^{1/2}.\]
By \cite[Lemma 3.3]{Gelbart}, $J(\tilde{g}, z)$ satisfies the automorphy condition i.e.,
\[J(\tilde{g}\tilde{h}, z)= J(\tilde{g}, \tilde{h}z)J(\tilde{h}, z).\]
Let $f \in S_{k+1/2}(\Gamma_0(N))$ and $\alpha \in \Gamma_0(N)$. Then
considering $\overline{\alpha} = (\alpha, s_{\A}(\alpha)) \in \DSL_2(\R)$, using
above lemma we have,
\begin{equation*}
\begin{split}
f(\overline{\alpha}z) &= \kro{c}{d} (\varepsilon_d^{-1})^{2k+1} (cz+d)^{k+1/2}f(z)\\
&=\kro{c}{d} (\varepsilon_d^{-1})^{2k+1} s_{\A}(\alpha)J(\overline{\alpha}, z)^{2k+1}f(z)\\
&=\begin{cases}
   (\varepsilon_d^{-1}J(\overline{\alpha}, z))^{2k+1}f(z) & 
   \text{if $c = 0$ or $c \ne 0$ and $\ord_2(c)$ is odd}\\
   \hst{c,d}(\varepsilon_d^{-1}J(\overline{\alpha}, z))^{2k+1}f(z) & 
   \text{if $c \ne 0$ and $\ord_2(c)$ is even}.
  \end{cases}
\end{split}
\end{equation*}
For $\theta \in \R$, let $k(\theta) =\mat{cos\theta}{sin\theta}{-sin\theta}{cos\theta}$.
Define $\tilde{K}_{\infty} := \{\tilde{k}(\theta) : \theta \in (-2\pi, 2\pi]\}$ where 
\[\tilde{k}(\theta) = \begin{cases}
                      (k(\theta), 1) & 
                      \text{if $-\pi<\theta\le \pi$},\\
                      (k(\theta), -1) & 
                      \text{if $-2\pi<\theta\le -\pi$ or $\pi<\theta\le 2\pi$}.
                     \end{cases}
\]
Then $\tilde{K}_{\infty}$ is a maximal compact subgroup of $\DSL_2(\R)$ and 
$\tilde{k}(\theta)\mapsto e^{i\frac{2k+1}{2}\theta}$ is a genuine character of 
$\tilde{K}_{\infty}$.
Let 
\[K_1(N) = \prod_{q<\infty}\{\mat{a}{b}{c}{d}\in \SL_2(\Z_q)\ :\  
c \equiv 0,\ \text{and}\ a,d\equiv 1 \pmod{N\Z_q}\}.\]
Recall the strong approximation theorem for $\DSL_2(\A)$: 
every element   
$\tilde{g} \in \DSL_2(\A)$ can be written as   
\[\tilde{g} = (\alpha, s_{\A}(\alpha))\tilde{g}_{\infty}(k_1,1),\]
where $(\alpha, s_{\A}(\alpha)) \in s_{\Q}(\SL_2(\Q))$, $k_1 \in K_1(N)$ and
$\tilde{g_{\infty}} \in \DSL_2(\R)$ determined up to left multiplication 
by elements in $s_{\Q}(\Gamma_1(N))$.

We follow the notation of Waldspurger~\cite{Waldspurger}. Let $\chi$ be an even Dirichlet character modulo $N$. 
Write $\chi_0 = \chi\kro{-1}{.}^k$. 
Define $\tilde{\gamma_2}$ on $\Z_2^\times$ as 
\[\tilde{\gamma_2}(t) = \begin{cases}
                        1 & \text{ if $t \equiv 1 \pmod{4\Z_2}$}\\
                        -i & \text{ if $t \equiv 3 \pmod{4\Z_2}$},
                       \end{cases}
                       \]
and for $k_0 = \mat{a}{b}{c}{d} \in K_0^2(4)$, define
\[\tilde{\epsilon}_2(k_0) = \begin{cases}
                             \tilde{\gamma_2}(d)^{-1}\hst{c,d}s_2(k_0) 
                             & \text{if $c \ne 0$} \\
                             \tilde{\gamma_2}(d) & \text{if $c = 0$.}
                            \end{cases}
\]
Let $\chi_0$ also denote the idelic character 
(of $\Qs \backslash \As$) corresponding to 
the Dirichlet character $\chi_0$ (it will be clear 
from the context when we consider $\chi_0$ to be idelic or 
Dirichlet character) and $\chi_{0,p}$ be
%(character on $\Q_p^\times$) 
the $p$-component of idelic character $\chi_0$.
Let $A_{k+1/2}(N, \chi_0)$ denote the set of functions 
$\Phi:\DSL_2(\A)\rightarrow \C$ satisfying the following properties:
\begin{enumerate}
 \item $\Phi(s_{\Q}(\alpha) \tilde{g} (k_1,1))= \Phi(\tilde{g})$ 
 for all $k_1 \in \prod_{q\nmid N}\SL_2(\Z_q), \alpha \in \SL_2(\Q),\tilde{g} \in \DSL_2(\A)$.
 \item $\Phi$ is genuine, i.e., $\Phi((I,\zeta)\tilde{g})= \zeta\Phi(\tilde{g})$ for 
 $\zeta\in \mu_2$. 
 \item For odd primes $p$ such that $p^n \| N$,  
 $\Phi(\tilde{g}(k_0,1)) = \chi_{0,p}(d)\Phi(\tilde{g})$ for all $k_0 = \mat{a}{b}{c}{d} \in K_0^p(p^n)$.
 \item If $2^n \|N$ ($n \ge 2$),  
$\Phi(\tilde{g}(k_0,1)) = \tilde{\epsilon}_2(k_0)\chi_{0,2}(d)\Phi(\tilde{g})$ 
for all $k_0\in K_0^2(2^n)$.
 \item $\Phi(\tilde{g} \tilde{k}(\theta)) = e^{i\frac{2k+1}{2}\theta}\Phi(\tilde{g})$ 
 for all $\tilde{k}(\theta) \in \tilde{K}_\infty$.
 \item $\Phi$ is smooth as a function of $\DSL_2(\R)$ and 
 satisfies the differential equation $\Delta \Phi = -[(2k+1)/4 \cdot (2k-3)/4]\Phi$ where $\Delta$
 is the Casimir operator.
 \item $\Phi$ is square integrable, that is 
 $\int_{s_{\Q}(\SL_2(\Q))\backslash\DSL_2(\A)/\mu_2} |\Phi(\tilde{g})|^2 d\tilde{g} < \infty$.
 \item $\Phi$ is cuspidal, that is 
 $\int_{N_{\Q} \backslash N_{\A}} \Phi\left(\mat{1}{a}{0}{1}\tilde{g}\right) da =0$ for all 
 $\tilde{g} \in \DSL_2(\A)$.
\end{enumerate}

By \cite[Proposition 3]{Waldspurger} there exists an isomorphism between
\[A_{k+1/2}(N, \chi_0) \rightarrow S_{k+1/2}(\Gamma_0(N),\chi)\] given by $\Phi \mapsto f_{\Phi}$
where for $z \in \Hup$, 
\[
f_{\Phi}(z)=\Phi(\tilde{g}_\infty)J(\tilde{g}_\infty,i)^{2k+1}
\]
where $\tilde{g}_\infty \in \DSL_2(\R)$ is such that $\tilde{g}_\infty(i)=z$. 
The inverse map is given by $f\mapsto \Phi_f$ where for $g \in \DSL_2(\A)$ if
$\tilde{g} = (\alpha, s_{\A}(\alpha))\tilde{g}_{\infty}(k_1,1)$,
\[
\Phi_f(\tilde{g})=f(\tilde{g}_{\infty}(i))J(\tilde{g}_{\infty},i)^{-2k-1}.
\]
This isomorphism induces a ring isomorphism of spaces of linear operators,
\[q:\mathrm{End}_{\C}(A_{k+1/2}(N, \chi_0))\rightarrow \mathrm{End}_{\C}(S_{k+1/2}(\Gamma_0(N),\chi))\] given by
\[
q(\mathcal{T})(f)=f_{\mathcal{T}(\Phi_f)}.\]

\subsection{$N = 4M$, $M$ odd and $\ p\|M$}
Let $p$ be an odd prime and let $N=4M$ with $M$ odd such that 
$p$ strictly divides $M$. 
In this subsection we translate the elements 
$\mathcal{T}_1$, $\mathcal{U}_1$, $\mathcal{U}_0$ and $\mathcal{T}_{-1}$ in the 
$p$-adic Hecke algebra to certain classical 
operators on $S_{k+1/2}(\Gamma_0(4M),\chi)$.
We restrict ourselves to $\chi$ being the trivial character modulo $4M$. 
In this case  $\chi_0 = \kro{-1}{\cdot}^k$ has conductor either $1$ or $4$ 
and so $\chi_{0,p}$ is trivial on $\Z_p^\times$ while $\chi_{0,2}$ 
acts by $\chi_0^{-1}=\chi_0$ on $\Z_2^\times$.
\begin{comment}
Following is in the particular case when $N=4p$ and we have character.
We shall first consider the case when level is $N=4p$, with $p$ an odd prime.
In this case we have two choices of even Dirichlet character 
$\chi$ modulo $4p$, namely 
$\chi$ trivial and $\chi = \kro{p}{\cdot}$. Recall $\kro{p}{\cdot}$ is character
of conductor $p$ if $p \equiv \pmod{4}$ and has conductor $4p$ if 
$p \equiv 3\pmod{4}$. If $\chi = \kro{p}{\cdot}$ then 
the $p$-primary part of $\chi$ is the Dirichlet character $\chi_p = \kro{\cdot}{p}$. 
For $\chi$ trivial, $\chi_0 = \kro{-1}{\cdot}^k$ of conductor $1$ or $4$ 
and so $\chi_{0,p}$ is trivial on $\Z_p^\times$ while $\chi_{0,2}$ 
acts by $\chi_0^{-1}=\chi_0$ on $\Z_2^\times$. 
For $\chi =  \kro{p}{\cdot}$, if $p\equiv 1\pmod{4}$ then
$\chi_0 = \kro{\cdot}{p}\kro{-1}{\cdot}^k$, thus 
$\chi_{0,p}$ acts by $\kro{\cdot}{p}$ on $\Z_p^\times$ and 
$\chi_{0,2}$ acts by $\kro{-1}{\cdot}^k$ on $\Z_2^\times$.
While if $p\equiv 3\pmod{4}$ then
$\chi_0 = \kro{\cdot}{p}\kro{-1}{\cdot}^{k+1}$, thus 
$\chi_{0,p}$ again acts by $\kro{\cdot}{p}$ on $\Z_p^\times$ but
$\chi_{0,2}$ acts by $\kro{-1}{\cdot}^{k+1}$ on $\Z_2^\times$.
\end{comment}

Let $\gamma$ be character on $(\Z_p/p\Z_p)^\times$ induced by 
$\chi_{0,p}|\Z_p^\times$ (so in the current case $\gamma$ is trivial). Then
Iwahori Hecke algebra $H(\ov{K_0^p(p)}, \gamma)$ is a subalgebra of 
$\mathrm{End}_{\C}(A_{k+1/2}(N, \chi_0))$ via the following
action: for $\mathcal{T} \in H(\ov{K_0^p(p)}, \gamma)$ and 
$\Phi \in A_{k+1/2}(N, \chi_0)$,
\[\mathcal{T}(\Phi)(\tilde{g}) = \int_{\DSL_2(\Q_p)}\mathcal{T}(\tilde{x}) 
\Phi(\tilde{g}\tilde{x})d\tilde{x}.\]

\begin{prop}\label{prop:rel3}
Let $\chi$ be the trivial character modulo $4M$ with $M$ as above
%or $\kro{p}{\cdot}$, 
and $\gamma$ be induced by $\chi_{0,p}$. 
Let $\mathcal{T}_1$, $\mathcal{U}_1$, $\mathcal{U}_0$, $\mathcal{T}_{-1} \in 
H(\ov{K_0^p(p)}, \gamma)$ and 
$f\in S_{k+1/2}(\Gamma_0(4M), \chi)$. Then,
\begin{enumerate}
 \item $\displaystyle{\kro{-1}{p}^k q(\mathcal{T}_1)(f)(z) 
 =  p^{-k-1/2}\sum_{s=0}^{p^2-1}f\left(\frac{z+s}{p^2}\right)
 = p^{(3-2k)/2}U_{p^2}(f)}$.
%\begin{cases}
%\kro{-1}{p}^k p^{-k-1/2}\sum_{s=0}^{p^2-1}f(\frac{z+s}{p^2}) & \text{ if $\chi$ trivial}\\
%\kro{-1}{p}^{k+1}p^{-k-1/2}\sum_{s=0}^{p^2-1}f(\frac{z+s}{p^2}) & \text{ if $\chi = \kro{p}{\cdot}$}.
%\end{cases}$
\item $\displaystyle{q(\mathcal{U}_1)(f)(z)= 
  \overline{\varepsilon_p}\kro{-1}{p}\kro{M/p}{p}
  \sum_{s=0}^{p-1}f|[(\alpha_s, \phi_{\alpha_s})]_{k+1/2}(z),}$ where \\
  $\alpha_s=\mat{p^2n-4Mms}{m}{4pM(1-s)}{p} \in \mathrm{M}_2(\Z)$ is of 
  determinant $p^2$ and $m,\ n\in\Z$ are such that $pn-(4M/p)m=1$, and 
$\phi_{\alpha_s}(z) = (4M(1-s)z + 1)^{1/2}$.
\item $\displaystyle{q(\mathcal{U}_0)(f)(z)= \sum_{s=0}^{p-1} 
f|[(\beta_s, \phi_{\beta_s})]_{k+1/2}(z),}$ where \\
$\beta_s=\mat{1}{m-s}{4M_1}{np-4M_1s} \in \Gamma_0(4M_1)$ with 
$M_1=M/p$ and $m,\ n\in\Z$ are chosen as above and
$\phi_{\beta_s}= (4M_1z+(np-4M_1s))^{1/2}$.
%f\left(\frac{z+(m-s)}{4M_1z+(np-4M_1s)}\right)(4M_1z+(np-4M_1s))^{-k-1/2}}$  
\item $\displaystyle{q(\mathcal{T}_{-1})(f)(z) = 
\kro{-1}{p}^k \sum_{s=0}^{p^2-1} f|[(\gamma_s, \phi_{\gamma_s}(z))]_{k+1/2}(z)}$, 
where \\
$\gamma_s = \mat{p^2}{0}{-4Ms}{1}$ and 
$\phi_{\gamma_s}(z) = (-4(M/p)sz + p^{-1})^{1/2}$.
\end{enumerate}
\end{prop}
 %\begin{cases}
 %\kro{-1}{p}^{k} p^{(3-2k)/2}U_{p^2}(f) & \text{ if $\chi$ trivial}\\
 %\kro{-1}{p}^{k+1} p^{(3-2k)/2}U_{p^2}(f) & \text{ if $\chi = \kro{p}{\cdot}$}.\\
 %\end{cases}
\begin{proof}
For $(1)$, let $\tilde{g}_\infty = (g_\infty, 1)\in \DSL_2(\R)$ such that 
$\tilde{g}_\infty i = z$. Then using decomposition in Lemma~\ref{lem:decomp} we have
\[\mathcal{T}_1(\Phi_f)(\tilde{g_\infty}) 
%= \int_{\ov{K}_0(h(p),1)\ov{K}_0}\mathcal{T}_1(\tilde{x}) 
%\Phi_f(\tilde{g}_\infty \tilde{x})d\tilde{x}
=\sum_{s=0}^{p^2-1}\overline{\gamma}(h(p),1) \Phi_f(\tilde{g}_\infty (x(s),1)(h(p),1))\]
\[=\overline{\varepsilon_p}\sum_{s=0}^{p^2-1}\Phi_f(\tilde{g}_\infty (x(s),1)(h(p),1)).\]
Take $A_s = h(p^{-1})x(-s) = \mat{p^{-1}}{-p^{-1}s}{0}{p} \in \SL_2(\Q)$, 
then $s_{\Q}(A_s) = (A_s, 1)$. 
Now the $\infty$-component of 
\[\underbrace{(A_s,1)}_{\text{diagonal emb.}} \cdot \underbrace{\tilde{g}_\infty}_{\infty\ \text{place}}
\underbrace{\cdot (x(s),1)(h(p),1)}_{p\ \text{place}}\]
%\[s_{\Q}(A)\tilde{g}_\infty (x(s),1)(h(p),1))\] 
is 
$(A_s,1)\tilde{g}_\infty$, for a prime $q$ such that $(q,2M)=1$ 
the $q$-component is 
$(A_s, 1) \in \SL_2(\Z_q)\times\{1\}$, 
for a prime $r$ such that $(r,2p)=1$ and $r^b\|M$, the $r$-component 
is $(A_s, 1) \in K_0^r(r^b)\times\{1\}$,
the $2$-component is 
$(A_s, 1)\in  K_0^2(4)\times\{1\}$ and the $p$-component is 
$(A_s, 1)(x(s),1)(h(p),1) = (I, \hs{p,p}) =(I, \kro{-1}{p})$. 

Since $\chi$ is trivial, $\chi_{0,2}=\kro{-1}{\cdot}^k$ while 
$\chi_{0,p}$ and $\chi_{0,r}$ are trivial. So the $2$-component acts by 
$\tilde{\epsilon}_2(A_2)\chi_{0,2}(p) = \tilde{\gamma_2}(p)\chi_{0,2}(p)= 
 \overline{\varepsilon_p}\kro{-1}{p}^k$.
Thus,
\begin{equation*}
\begin{split}
\mathcal{T}_1(\Phi_f)(\tilde{g}_\infty) 
 &=\overline{\varepsilon_p}\sum_{s=0}^{p^2-1}
 \Phi_f(s_{\Q}(A)\tilde{g}_\infty (x(s),1)(h(p),1))\\
 &= (\overline{\varepsilon_p})^2\kro{-1}{p}^k\kro{-1}{p}
 \sum_{s=0}^{p^2-1}\Phi_f(Ag_\infty,1)\\
 &= \kro{-1}{p}^k\sum_{s=0}^{p^2-1}
 f(Ag_\infty(i))J((Ag_\infty,1),i)^{-2k-1}.
\end{split}
\end{equation*}
Consequently, 
\[q(\mathcal{T}_1)(f)(z) 
=\mathcal{T}_1(\Phi_f)(\tilde{g}_\infty)J((g_\infty,1),i)^{2k+1}  
%= \kro{-1}{p}^{(k\mp1)/2}\sum_{s=0}^{p^2-1}f(Az)J((A,1),z)^{-k} 
=\kro{-1}{p}^{k}p^{-k-1/2}\sum_{s=0}^{p^2-1}f(\frac{z+s}{p^2}).\]

For $(2)$ we need the following decomposition (we use $(4,M)=1$)
%it is different from that in Lemma~\ref{lem:decomp} 
\[K_0 w(p^{-1}) K_0 = \bigcup_{s\in \Z_p / p\Z_p}
y(4Ms)w(p^{-1})K_0\]
Taking $\tilde{g}_\infty$ such that $\tilde{g}_\infty i=z$ we have
\[\mathcal{U}_1(\Phi_f)(\tilde{g}_\infty) 
%= \int_{\ov{K}_0(w(p^{-1}),1)\ov{K}_0}\mathcal{T}_1(\tilde{x}) 
%\Phi_f(\tilde{g}_\infty \tilde{x})d\tilde{x}
%=\sum_{s=0}^{p-1}\gamma(w(p^{-1}),1) \Phi_f(\tilde{g}_\infty (y(4ps),1)(w(p^{-1}),1))\]
=\overline{\varepsilon_p}\kro{-1}{p}
\sum_{s=0}^{p-1}\Phi_f(\tilde{g}_\infty (y(4Ms),1)(w(p^{-1}),1)).\]
Since $p$ is coprime to $4M/p$, we 
fix $m,\ n\in \Z$ such that $pn-(4M/p)m=1$. For $0\le s\le p-1$, take 
\[A_s = \mat{pn}{\frac{m}{p}}{4M}{1}\mat{1}{0}{-4Ms}{1}= 
\mat{pn-4ms\frac{M}{p}}{\frac{m}{p}}{4M(1-s)}{1} \in \SL_2(\Q).\] 
Since $s_{\nu}(A_s)=1$ for all primes $\nu$ we have $s_{\Q}(A_s) = (A_s, 1)$. 
As before the $\infty$-component of 
\[s_{\Q}(A_s)\ \tilde{g}_\infty\  (y(4Ms),1)(w(p^{-1}),1)\] is 
$(A_s,1)\tilde{g}_\infty$,  for a prime $q$ such that $(q,2M)=1$ 
the $q$-component is 
$(A_s, 1) \in \SL_2(\Z_q)\times\{1\}$, 
for a prime $r$ such that $(r,2p)=1$ and $r^b\|M$, the $r$-component 
is $(A_s, 1) \in K_0^r(r^b)\times\{1\}$,
the $2$-component is 
$(A_s, 1)\in  K_1^2(4)\times\{1\}$ (as $(2,2)$-th entry of $A_s$ is $1$). 
At the $p$-component we 
check that $(A_s, 1) = (\mat{pn}{m/p}{4M}{1},1)(y(-4Ms),1)$ and 
\[(A_s, 1)(y(4Ms),1)(w(p^{-1}),1) = 
%(\mat{pn}{m/p}{4M}{1},1)(w(p^{-1}),1) = 
(\mat{-m}{n}{-p}{4M/p}, \kro{M/p}{p}).\]
Since $\chi$ is trivial, the $p$ and $r$ components acts trivially and 
the $2$-component acts by $\tilde{\epsilon}_2(A_s)\chi_{0,2}(1) = 1$.
Hence
\begin{equation*}
\begin{split}
\mathcal{U}_1(\Phi_f)(\tilde{g}_\infty) 
& =\overline{\varepsilon_p}\kro{-1}{p}
\sum_{s=0}^{p-1}\Phi_f(s_{\Q}(A_s)\tilde{g}_\infty (y(4ps),1)(w(p^{-1}),1))\\
&= \overline{\varepsilon_p}\kro{-1}{p}\kro{M/p}{p}
 \sum_{s=0}^{p-1}\Phi_f((A_s,1)(g_\infty, 1))
\end{split}
\end{equation*}
\[=\overline{\varepsilon_p}\kro{-1}{p}\kro{M/p}{p}
 \sum_{s=0}^{p-1}f((A_s,1)z)J((A_s,1),z)^{-2k-1}J((g_\infty, 1),i)^{-2k-1}.\]
So we have 
\[q(\mathcal{U}_1)(f)(z)  
= \overline{\varepsilon_p}\kro{-1}{p}\kro{M/p}{p}
\sum_{s=0}^{p-1}f((A_s,1)z)J((A_s,1),z)^{-2k-1}.\] 
Let $\alpha_s = A_s \mat{p}{0}{0}{p}$ and 
$\phi_{\alpha_s}(z) = (4M(1-s)z + 1)^{1/2}$. Then $q(\mathcal{U}_1)(f)(z)=$
\[  
= \overline{\varepsilon_p}\kro{-1}{p}\kro{M/p}{p}
\sum_{s=0}^{p-1}f\left(\frac{(p^2n-4mMs)z +m}{4pM(1-s)z+p}\right)(4M(1-s)z + 1)^{-k-1/2}\]
\[=\overline{\varepsilon_p}\kro{-1}{p}\kro{M/p}{p}
\sum_{s=0}^{p-1}f|[(\alpha_s, \phi_{\alpha_s})]_{k+1/2}(z).\]

For $(3)$, using Lemma~\ref{lem:decomp} we have
\[\mathcal{U}_0(\Phi_f)(\tilde{g}_\infty) 
=\sum_{s=0}^{p-1}\Phi_f(\tilde{g}_\infty (x(s),1)(w(1),1)).\]
Let $m,\ n\in \Z$ such that $pn-(4M/p)m=1$ and 
let $M_1 = M/p$. For $0\le s\le p-1$, take 
\[A_s = \mat{1}{-s+m}{4M_1}{-4M_1s+np}\in \Gamma_0(4M_1).\]
%Consider the Kronecker symbol $\kro{4M_1}{-4M_1s+np}$. 
%If $M_1 = 2^{2k}u$ with $u$ odd then  
%$\kro{4M_1}{-4M_1s+np} = \kro{u}{-2^{2k+2}us+np} = 
%\kro{-2^{2k+2}us+np}{u} = \kro{np}{u} = \kro{1+4M_1m}{u}=1$.
%If $M_1 = 2^{2k+1}u$, then 
%$\kro{4M_1}{-4M_1s+np} = \kro{2u}{-4M_1s+np} = 
%\kro{u}{-4M_1s+np} = 1$ as $\kro{2}{-4M_1s+np}=1$.
Since the Kronecker symbol
$\kro{4M_1}{-4M_1s+np} = \kro{M_1}{-4M_1s+np} = 
\kro{-4M_1s+np}{M_1} = \kro{np}{M_1} =1$, by Lemma~\ref{lem:sA} we get that
%(as $4M_1>0$) 
\[s_{\A}(A_s)=\begin{cases}
              (4M_1,-4M_1s+np)_2 & \text{if $\ord_2(M_1)$ even}\\
              1 & \text{if $\ord_2(M_1)$ odd}.
              \end{cases}\]

The $\infty$-component of 
$s_{\Q}(A_s)\ \tilde{g}_\infty\ (x(s),1)(w(1),1)$ is
$(A_s,s_{\A}(A_s))(g_\infty, 1)$,  for a prime $q$ such that $(q,2M)=1$ 
the $q$-component is 
$(A_s, 1) \in \SL_2(\Z_q)\times\{1\}$, 
if $r$ is an odd prime such that $r^b\|M$ then the $r$-component is 
$(A_s, 1) \in K_0^r(r^b)\times\{1\}$ and 
the $2$-component is 
$(A_s, 1)\in  K_1^2(4)\times\{1\}$. For $p$-component,  
since $\ord_p(4M_1)=0$, we have 
$(\mat{1}{m}{4M_1}{np},1)(x(-s),1) =(A_s,1)$  and 
$(\mat{1}{m}{4M_1}{np},1)(w(1),1) =(\mat{-m}{1}{-np}{4M_1},\beta)$
where $\beta$ is either $(4M_1,-1)_p$ or $(4M_1,np)_p$ 
depending on whether $\ord_p(np)$ is odd or even. In either 
case it is clear that $\beta$ is $1$.
Thus the $p$-component is $(\mat{-m}{1}{-np}{4M_1},1)\in K_0\times\{1\}$

Since $\chi$ is trivial, the $p$ and $r$ components acts trivially, 
and the $2$-component acts by 
$\tilde{\epsilon}_2(A_s)\chi_{0,2}(-4M_1s+np) = (4M_1,-4M_1s+np)_2s_2(A_s)$
which is precisely equal to $s_{\A}(A_s)$.

Thus $\mathcal{U}_0(\Phi_f)(\tilde{g}_\infty)=$
\[=\sum_{s=0}^{p-1}\Phi_f((A_s,1)\tilde{g}_\infty)
=\sum_{s=0}^{p-1}f(A_sz)J((A_s,1),z)^{-2k-1}J((g_\infty, 1),i)^{-2k-1}\]
and consequently
\[q(\mathcal{U}_0)(f)(z)= \sum_{s=0}^{p-1} 
f\left(\frac{z+(m-s)}{4M_1z+(np-4M_1s)}\right)(4M_1z+(np-4M_1s))^{-k-1/2}.\]

For $(4)$, using 
%the decomposition 
$K_0 h(p^{-1}) K_0 = \bigcup_{s\in \Z_p / p^{2}\Z_p}
y(4Ms)h(p^{-1})K_0$ we have
\[\mathcal{T}_{-1}(\Phi_f)(\tilde{g_\infty}) 
%= \int_{\ov{K}_0(h(p),1)\ov{K}_0}\mathcal{T}_1(\tilde{x}) 
%\Phi_f(\tilde{g}_\infty \tilde{x})d\tilde{x}
%=\sum_{s=0}^{p^2-1}\overline{\gamma}(h(p^{-1}),1) 
%\Phi_f(\tilde{g}_\infty (y(4Ms),1)(h(p^{-1}),1))\]
=\overline{\varepsilon_p}\sum_{s=0}^{p^2-1}
\Phi_f(\tilde{g}_\infty (y(4Ms),1)(h(p^{-1}),1)).\]
Take $A_s = h(p)y(-4Ms) = \mat{p}{0}{-4(M/p)s}{p^{-1}}$, 
%\in \SL_2(\Q)$, 
then $s_{\Q}(A_s) = (A_s, \xi_s)$ 
where $\xi_s := \begin{cases}
                1 & \text{ if $s=0$}\\
                1 & \text{ if $\ord_p(s)=1$ and $\ord_2(s)$ odd}\\
                \kro{-1}{p}(\frac{Ms}{p},p)_2 & \text{ if $\ord_p(s)=1$ and $\ord_2(s)$ even}\\
                \kro{-1}{p}(\frac{Ms}{p},p)_p & \text{ if $(s,p)=1$ and $\ord_2(s)$ odd}\\
                (\frac{Ms}{p},p)_2 \ (\frac{Ms}{p},p)_p & \text{ if $(s,p)=1$ and $\ord_2(s)$ even}.
               \end{cases}$

Thus 
\[\mathcal{T}_{-1}(\Phi_f)(\tilde{g_\infty})
 =\overline{\varepsilon_p}\sum_{s=0}^{p^2-1}\xi_s
\Phi_f((A_s,1)\tilde{g}_\infty (y(4Ms),1)(h(p^{-1}),1)).\]
Now the $\infty$-component of 
$(A_s,1)\ \tilde{g}_\infty\ (y(4Ms),1)(h(p^{-1}),1)$ is 
%\[\underbrace{(A_s,1)}_{\text{diagonal emb.}} \cdot \underbrace{\tilde{g}_\infty}_{\infty\ \text{place}}
%\underbrace{\cdot (y(4Ms),1)(h(p^{-1}),1)}_{p\ \text{place}}\] is 
$(A_s,1)\tilde{g}_\infty$, for a prime $q$ such that $(q,2M)=1$ the 
$q$-component is 
$(A_s, 1) \in \SL_2(\Z_q)\times\{1\}$, 
if $r$ is an odd prime coprime to $p$ such that $r^b \| M$ then the
$r$-component belongs to $K_0^r(r^b)\times\{1\}$, 
the $2$-component is 
$(\mat{p}{0}{-4(M/p)s}{p^{-1}}, 1)\in  K_0^2(4)\times\{1\}$ and 
the $p$-component is 
$(A_s, 1)(y(4Ms),1)(h(p^{-1}),1)$ which is precisely equal to $(I, \eta_s)$ 
where $\eta_s := \begin{cases}
                 \kro{-1}{p} & \text{ if $s=0$}\\
                 1 & \text{ if $\ord_p(s)=1$}\\
                 \kro{-1}{p}(\frac{Ms}{p},p)_p & \text{ if $(s,p)=1$}.
                 \end{cases}$ 

Since $\chi$ is trivial, $\chi_{0,p}$ is trivial and so 
the $p$-component acts on $\Phi_f$ simply by multiplication by $\eta_s$.                
Next we look at how the $2$-component acts on $\Phi_f$. 
Since $\chi_{0,2}=\kro{-1}{\cdot}^k$ we get that
\[\tilde{\epsilon}_2(A_s)\chi_{0,2}(p^{-1}) = 
\begin{cases}
\tilde{\gamma_2}(p^{-1})\chi_{0,2}(p^{-1}) & \text{ if $s=0$}\\
\tilde{\gamma_2}(p^{-1})^{-1} (-4\frac{M}{p}s, p^{-1})_2 s_2(A_s)\chi_{0,2}(p^{-1}) 
& \text{ if $s\ne 0$}
\end{cases}\]
\[=
\begin{cases}
\overline{\varepsilon_p}\kro{-1}{p}^k & \text{ if $s=0$}\\
\varepsilon_p\kro{-1}{p}^{k} & \text{ if $s\ne 0$ and $\ord_2(s)$ odd}\\
\varepsilon_p\kro{-1}{p}^{k+1}(\frac{Ms}{p},p)_2 & \text{ if $s\ne 0$ and $\ord_2(s)$ even},
\end{cases} =:\vartheta_s\]
One can check that 
\[\vartheta_s\cdot \eta_s = \varepsilon_p\kro{-1}{p}^{k} \xi_s,\]
and so
\[\mathcal{T}_{-1}(\Phi_f)(\tilde{g}_\infty) 
 =\overline{\varepsilon_p}\sum_{s=0}^{p^2-1} \xi_s \cdot
 \vartheta_s\cdot \eta_s
 \Phi_f((A_s,1)\tilde{g}_\infty) = 
 \kro{-1}{p}^{k} \sum_{s=0}^{p^2-1} \Phi_f((A_s,1)\tilde{g}_\infty).
 \] 
Thus
\[q(\mathcal{T}_{-1})(f)(z) 
= \kro{-1}{p}^{k} \sum_{s=0}^{p^2-1} f\left(\frac{p^2z}{-4Msz + 1}\right)
\left(\frac{-4Msz +1}{p}\right)^{-k-1/2}\]
\[= \kro{-1}{p}^{k} \sum_{s=0}^{p^2-1} f|[(\gamma_s, \phi_{\gamma_s}(z))]_{k+1/2}(z) \]
where $\gamma_s = \mat{p^2}{0}{-4Ms}{1}$ and 
$\phi_{\gamma_s}(z) = (-4(M/p)sz + p^{-1})^{1/2}$.
\end{proof}
%It follows from $(1)$ of the above proposition that
%\[q(\mathcal{T}_1)(f)(z)= \kro{-1}{p}^{k} p^{(3-2k)/2}U_{p^2}(f).\]

Let $\widetilde{Q}_p:=q(\mathcal{U}_0)$ and 
$\widetilde{W}_{p^2}:=q(p^{-1/2}\mathcal{U}_1)$. 
Then we have the following
\begin{cor}\label{cor:rel4}
%If $\chi$ is the trivial character then 
On $S_{k+1/2}(\Gamma_0(4M))$ we have
\begin{enumerate}
 \item $\widetilde{W}_{p^2}$ is an involution.
 \item $(\widetilde{Q}_p-p)(\widetilde{Q}_p+1)=0$.
 \item $\widetilde{Q}_p = \kro{-1}{p}^k p^{1-k}U_{p^2}\widetilde{W}_{p^2}$.
 %\item If $f \in S_{k+1/2}(\Gamma_0(4))$ then $\widetilde{Q}_p(f)=pf$.
 \item  If $f \in S_{k+1/2}(\Gamma_0(4M/p))$ then $\widetilde{Q}_p(f)=pf$.
\end{enumerate}
\end{cor}
\begin{proof}
The proof of $(1)$ to $(3)$ follows by using Proposition~\ref{prop2:rel2} and 
\ref{prop:rel3}. For $(4)$ we use Proposition
\ref{prop:rel3}$(3)$.
\end{proof}
 
We further define an operator 
$\widetilde{Q}'_p$ on $S_{k+1/2}(\Gamma_0(4M))$ to be 
the conjugate of $\widetilde{Q}_p$ by $\widetilde{W}_{p^2}$, i.e., 
$\widetilde{Q}'_p = \widetilde{W}_{p^2}\widetilde{Q}_p\widetilde{W}_{p^2}$. 
Thus $\widetilde{Q}'_p$ satisfies the same quadratic as $\widetilde{Q}_p$ and we 
have $\widetilde{Q}'_p = \kro{-1}{p}^k p^{1-k}\widetilde{W}_{p^2}U_{p^2}$.

\begin{remark}
We note that for a prime $q$ such that $(q, 2M)=1$, one can similarly 
obtain the usual Hecke operator $T_{q^2}$ on $S_{k+1/2}(\Gamma_0(4M))$. 
In particular, if we take $\mathcal{T}_1 := X_{(h(q),1)} \in 
H(\ov{\SL_2(\Z_q)},\gamma_q)$ then $q(\mathcal{T}_1) = 
\kro{-1}{p}^k p^{(3-2k)/2}T_{q^2}$.

Moreover if $p$ and $q$ are distinct primes such that $p^n,\ q^m$ strictly 
divides $N$ then the operators 
$\mathcal{S} \in H(\ov{K_0^p(p^n)},\gamma_p)$ and 
$\mathcal{T} \in H(\ov{K_0^q(q^m)},\gamma_q)$ in 
$\mathrm{End}_{\C}(S_{k+1/2}(\Gamma_0(N))$ commute. 

In particular the operators $\widetilde{Q}_p$, $\widetilde{W}_{p^2}$ on 
$S_{k+1/2}(\Gamma_0(4M))$ that we defined above commute with 
$T_{q^2}$ for primes $q$ coprime to $2M$. 
\end{remark}

\subsection{Self-adjointness}
Let $M$ be odd such that $p\|M$. In this subsection we 
check that the operators $\widetilde{W}_{p^2}$, 
$\widetilde{Q}_p$ and $\widetilde{Q}'_p$ are self-adjoint operators on 
$S_{k+1/2}(\Gamma_0(4M))$. The property of 
self-adjointness will be used to give a description of our 
minus space in terms of common eigenspaces. 
\begin{prop}
 The operator $\widetilde{W}_{p^2}$ is self-adjoint 
 with respect to Petersson inner product.
\end{prop}
\begin{proof}
We write
\[\widetilde{W}_{p^2}(f) = 
\frac{\overline{\varepsilon_p}}{\sqrt{p}}\kro{-1}{p}\kro{M/p}{p} 
\mathcal{S}_p(f),\ \   
\mathcal{S}_p(f):=\sum_{s \in \Z/p\Z}f|[(\alpha_s, \phi_{\alpha_s}(z))]_{k+1/2}\]
where
$(\alpha_s, \phi_{\alpha_s}(z)) = 
\left(\mat{p^2n-4mMs}{m}{4pM(1-s)}{p}, (4M(1-s)z+1)^{1/2}\right) 
\in \mathcal{G}$ and $n,\ m$ are integers such that $pn-(4M/p)m=1$. 

We will show that 
$\langle \mathcal{S}_p(f), g \rangle = 
\kro{-1}{p}\langle f, \mathcal{S}_p(g) \rangle$.
We write $\mathcal{S}_p = S_{1,p} + S_{2,p}$ where $S_{1,p}$ consists 
of $s=0$ term and $S_{2,p}$ consists of rest of the terms. 
Also let $M_1 = M/p$.

We first consider $S_{2,p}$. For $s\ne 0$, as $pn-4M_1ms = 1+ 4M_1m(1-s)$ 
it is clear that $pn-4M_1ms$ and
$4M(1-s)$ are relatively coprime, hence there exists integers $u$, $v$ such that 
$u(-pn + 4M_1ms) - v 4M(1-s)=1$. Let 
$X = \mat{u}{v}{4M(1-s)}{-pn + 4M_1ms} \in \Gamma_0(4M)$, then 
$X^* = (X, j(X,z))$ where $j(X,z) = 
(-1)^{-1/2} \kro{4M(1-s)}{-pn+4M_1ms} (4M(1-s)z + (-pn + 4M_1ms))^{1/2}$
as
$-pn + 4M_1ms \equiv -1 \pmod{4}$. Since $f$ has level $4M$ we have 
$f|[(\alpha_s, \phi_{\alpha_s}(z))]_{k+1/2} = 
 f|[X^*]_{k+1/2} |[(\alpha_s, \phi_{\alpha_s}(z))]_{k+1/2}$. 
 We claim that in $\mathcal{G}$,
 \[X^* \cdot (\alpha_s, \phi_{\alpha_s}(z)) = 
 \left(\mat{-p}{um+vp}{0}{-p}, -\kro{um}{p}\kro{M_1}{p}\right).\]
 It is easy to see equality in the matrix component, so just need to check 
 that $j(X, \alpha_s z)\cdot \phi_{\alpha_s}(z) = -\kro{um}{p}\kro{M_1}{p}$. 
 We see that $j(X, \alpha_s z)$ simplifies to 
 $(-1)^{-1/2}\kro{4M(1-s)}{-pn+4M_1ms}\left(\frac{-1}{4M(1-s)z+1}\right)^{1/2}$ 
 and so $j(X, \alpha_s z)\cdot \phi_{\alpha_s}(z) = \kro{4M(1-s)}{-pn+4M_1ms}$. 
 Thus we are left to show equality of the Kronecker symbols 
 $\kro{4M(1-s)}{-pn+4M_1ms} = -\kro{um}{p}\kro{M_1}{p}$. 
Note that 
$\kro{m}{p}\kro{M_1}{p} = \kro{-1}{p}$ and 
$\kro{u}{p} = \kro{-1 - 4M_1m(1-s)}{p}$, so we have 
$\kro{um}{p}\kro{M_1}{p} 
%= \kro{-1 - 4m(1-s)}{p}\kro{-1}{p}
%\kro{p}{-1 - 4m(1-s)}\kro{-1}{p}\kro{-1}{p} 
= \kro{p}{-1 - 4M_1m(1-s)}$.
Further 
$\kro{4M(1-s)}{-pn+4M_1ms} =
\kro{p}{-1 - 4M_1m(1-s)}\kro{-4M_1m(1-s)}{-1 - 4M_1m(1-s)}\kro{-m}{-1 - 4M_1m(1-s)} = \\
%\kro{p}{-1 - 4M_1m(1-s)}\kro{1-1-4M_1m(1-s)}{-1 - 4M_1m(1-s)}\kro{-m}{-1 - 4M_1m(1-s)} 
\kro{p}{-1 - 4M_1m(1-s)}\kro{-m}{-1 - 4M_1m(1-s)}$. 
We now check that
$\kro{-m}{-1 - 4M_1m(1-s)} = -1$. Suppose $m$ is odd, then 
$\kro{-m}{-1 - 4mM_1(1-s)} = -\kro{m}{-1 - 4M_1m(1-s)} = 
-\kro{-1 - 4M_1m(1-s)}{m}\kro{-1}{m} = -1$. Suppose $m = 2^km'$ where $k>1$ and $m'$ 
odd, then \\
$\kro{-m}{-1 - 4M_1m(1-s)} =
\begin{cases}
\kro{-m'}{-1 - 4M_12^km'(1-s)} & \text{if $k$  even}\\
\kro{-2m'}{-1 - 4M_12^km'(1-s)} & \text{if $k$  odd}
\end{cases} 
= \kro{-m'}{-1 - 4M_12^km'(1-s)} 
= -1$ 
(as $\kro{2}{r} = (-1)^{(r^2-1)/8}$). Thus our claim is proved. 

Hence 
$f|[(\alpha_s, \phi_{\alpha_s}(z))]_{k+1/2} = 
f|[(\mat{-p}{um}{0}{-p}, -\kro{um}{p}\kro{M_1}{p})]_{k+1/2}$ and consequently
\[S_{2,p}(f)= -\kro{M_1}{p}\sum_{u\in(\Z/p\Z)^\times}
f|[(\mat{-p}{u}{0}{-p}, \kro{u}{p})]_{k+1/2}.\] Since the adjoint of 
$|[(\mat{-p}{u}{0}{-p}, \kro{u}{p})]_{k+1/2}$ is 
$|[(\mat{-p}{-u}{0}{-p}, \kro{u}{p})]_{k+1/2}$, so the adjoint of 
$S_{2,p}(f)$ is $\kro{-1}{p}S_{2,p}(f)$, i.e., 
$\langle S_{2,p}(f), g \rangle = 
\kro{-1}{p}\langle f, S_{2,p}(g) \rangle$.

Next we consider the term $S_{1,p}(f) = 
f|[(\mat{p^2n}{m}{4pM}{p}, (4Mz+1)^{1/2})]_{k+1/2}$.
Let $\gamma_p := \mat{a}{b}{c}{d} \in \SL_2(\Z)$ such that 
$\gamma_p \equiv \mat{0}{-1}{1}{0} \pmod{p}$, 
$\equiv \mat{1}{0}{0}{1} \pmod{8}$. We claim that 
\[S_{1,p}(f) =
%f|[(\mat{p^2n}{m}{4pM}{p}, (4Mz+1)^{1/2})]_{k+1/2}=
f|[(\mat{pa}{b}{p^2c}{pd}, 
\kro{M_1}{p}\kro{c}{d}(cpz+d)^{1/2})]_{k+1/2}
\] 
Choose $Y=\mat{a-4bM_1}{\frac{-ma+bpn}{p}}{pc-4Md}{-mc+dpn} 
\in \Gamma_0(4M)$. To prove the claim we need to check 
that 
\[Y^* \cdot (\mat{p^2n}{m}{4pM}{p}, (4Mz+1)^{1/2}) = 
 (\mat{pa}{b}{p^2c}{pd}, 
\kro{M_1}{p}\kro{c}{d}(cpz+d)^{1/2}).\]
As before, matrix equality is easy to check and the
automorphy factor of the left hand side equals 
kronecker symbol $\kro{pc-4Md}{-cm + dpn}$ 
times $(pcz +d)^{1/2}$. So we need to show that 
$\kro{pc-4Md}{-cm + dpn} = \kro{M_1}{p}\kro{c}{d}$.
Since $d-m(c-4M_1d) = -cm + dpn \equiv 1\pmod{4}$ we have  
$\kro{pc-4Md}{-cm + dpn} 
%= \kro{p(c-4d)}{d-m(c-4d)} 
= \kro{-cm + dpn}{p}\kro{c-4M_1d}{-cm + dpn}
= \kro{-mc}{p}\kro{-m(c-4M_1d)}{d-m(c-4M_1d)}\kro{-m}{-cm + dpn}=
%= \kro{-m}{p}\kro{-d + d-m(c-4M_1d)}{d-m(c-4M_1d)}\kro{-m}{d-m(c-4M_1d)}
 \kro{-m}{p}\kro{-d}{d-m(c-4M_1d)}\kro{m}{d-m(c-4M_1d)}
= \kro{M_1}{p}\kro{d-m(c-4M_1d)}{d}\kro{m}{d-m(c-4M_1d)}
= \kro{M_1}{p}\kro{c}{d} \kro{m}{d}\kro{m}{d-m(c-4d)}$. 
We can check $\kro{m}{d} = \kro{m}{d-m(c-4d)}$ as before
by considering $m$ odd and even case (when $m$ is even 
we use that $\kro{2}{d} =1$). Thus our claim is proved.

Now note that 
\[(\mat{pa}{b}{p^2c}{pd}, 
\kro{c}{d}(cpz+d)^{1/2}) = (\mat{1}{0}{0}{p},p^{1/4})\cdot \gamma_p^*
\cdot (\mat{p}{0}{0}{1},p^{-1/4}) =: \varsigma_p,\]
and so 
$S_{1,p}(f) = \kro{M_1}{p} f|[\varsigma_p]_{k+1/2}$. 

We check similarly that 
\[f|[(\mat{1}{0}{0}{p},p^{1/4})\cdot (\gamma_p^*)^2
\cdot (\mat{p}{0}{0}{1},p^{-1/4})]_{k+1/2} = \kro{-1}{p}f.\]
By the 
above action on $f$ we get that 
$f|[\varsigma_p^{-1}]_{k+1/2} = \kro{-1}{p}f|[\varsigma_p]_{k+1/2}$.
Since the adjoint of $\varsigma_p$ is $\varsigma_p^{-1}$ we get
$\langle S_{1,p}(f), g \rangle = 
\kro{-1}{p}\langle f, S_{1,p}(g) \rangle$.

Thus $\langle \mathcal{S}_p(f), g \rangle = 
\kro{-1}{p}\langle f, \mathcal{S}_p(g) \rangle$. So
$\langle \widetilde{W}_{p^2}(f), g  \rangle = 
\frac{\overline{\varepsilon_p}}{\sqrt{p}}\kro{-M_1}{p} 
\langle \mathcal{S}_p(f), g \rangle \\
= \frac{\overline{\varepsilon_p}}{\sqrt{p}}\kro{M_1}{p}\langle f, \mathcal{S}_p(g) \rangle
= \langle f, \frac{\varepsilon_p}{\sqrt{p}}\kro{M_1}{p}\mathcal{S}_p(g) \rangle 
=\langle f, \widetilde{W}_{p^2}(g)  \rangle$.
Hence we are done.
\end{proof}

\begin{remark}
In Remark \ref{rem:rem1} we checked that the $p$-adic operator 
$\mathcal{U}'_1 = \overline{\varepsilon_p}\mathcal{U}_1$ 
and Ueda's classical operator 
$Y_p$ satisfy same relations. 
Since $q(\mathcal{U}'_1) = \mathcal{S}_p$, it is natural 
to compare $\mathcal{S}_p$ and $Y_p$. While 
$\mathcal{S}_p = S_{1,p} + S_{2,p}$, it turns out from the 
above proof that Ueda's $Y_p$ is equal to 
$\kro{M/p}{p}(S_{1,p} - S_{2,p})$.
\end{remark}

Next we want to show that $\widetilde{Q}_p = q(\mathcal{U}_0)$ is self-adjoint. 
We use the relations  
$\mathcal{U}_1\mathcal{T}_1\mathcal{U}_1 = p\mathcal{T}_{-1}$ and 
$\mathcal{T}_1\mathcal{U}_1 = p\ \mathcal{U}_0$ 
(Proposition~\ref{prop2:rel2}$(3)$). Thus we have  
\[\langle q(\mathcal{U}_0) f, g\rangle
= \frac{1}{p}\ \langle  q(\mathcal{T}_1)q(\mathcal{U}_1) f , g\rangle.\] 
Since by the above theorem $q(\mathcal{U}_1)$ is self-adjoint we get that
\[\langle f, q(\mathcal{U}_0)g \rangle
=\frac{1}{p}\ \langle f, p\ q(\mathcal{U}_0)g \rangle
=\frac{1}{p}\ \langle f, q(\mathcal{T}_1)q(\mathcal{U}_1)g \rangle\]
\[
=\frac{1}{p}\ \langle f, \frac{1}{p}\ q(\mathcal{U}_1)^2 
q(\mathcal{T}_1)q(\mathcal{U}_1)g \rangle
=\frac{1}{p}\ \langle q(\mathcal{U}_1)f, 
\frac{1}{p}\ q(\mathcal{U}_1)q(\mathcal{T}_1)q(\mathcal{U}_1)g \rangle\]
\[
=\frac{1}{p}\ \langle q(\mathcal{U}_1)f, q(\mathcal{T}_{-1})g \rangle.\]
Since $q(\mathcal{U}_1)$ is surjective 
it follows that 
$q(\mathcal{U}_0)$ is self-adjoint {\it iff}
%$\langle q(\mathcal{T}_1)q(\mathcal{U}_1) f , g\rangle = 
%\langle q(\mathcal{U}_1)f, q(\mathcal{T}_{-1})g \rangle$, i.e., {\it iff} 
the adjoint of $q(\mathcal{T}_{-1})$ is 
$q(\mathcal{T}_{1})$. 
We now show that the adjoint of $q(\mathcal{T}_{-1})$ is 
$q(\mathcal{T}_{1})$. 
%Recall $q(\mathcal{T}_1)(f)(z) = 
%\kro{-1}{p}^{k} \sum_{s=0}^{p^2-1} f|[(\mat{1}{s}{0}{p^2}, p^{1/2})]_{k+1/2}$ 
%We know that the adjoint of the slash operator 
%\[|[(\mat{1}{s}{0}{p^2}, p^{1/2})]_{k+1/2}\ \ \text{is precisely}\ \ 
%|[(\mat{p^2}{-s}{0}{1}, p^{-1/2})]_{k+1/2}.\]
%So we would like to show that 
%\[\sum_{s=0}^{p^2-1} f|[(\alpha_s, \phi_{\alpha_s}(z))]_{k+1/2} 
% = \sum_{s=0}^{p^2-1} f|[(\mat{p^2}{-s}{0}{1}, p^{-1/2})]_{k+1/2}.
%\]

Consider elements $\xi = (\mat{1}{0}{0}{p^2},p^{1/2})$ and 
$\eta = (\mat{p^2}{0}{0}{1},p^{-1/2})$ in $\mathcal{G}$. 
%For $f\in S_{k+1/2}(\Delta_0(4M))$ we have \cite[Proposition 1.5]{Shimura}
%\[f|[\Delta_0(4M)\xi\Delta_0(4M)]_{k+1/2} = U_{p^2}(f) 
%=\kro{-1}{p}^{k} p^{(2k-3)/2} q(\mathcal{T}_1)(f).\]
We can choose $\beta_s$ such that 
$\Gamma_0(4M)\mat{1}{0}{0}{p^2}\Gamma_0(4M) = 
\bigsqcup \Gamma_0(4M)\beta_s = \bigsqcup \beta_s \Gamma_0(4M)$. 
So by \cite[Propositions 1.1, 1.2]{Shimura} we have 
$\Delta_0(4M)\xi\Delta_0(4M) = \bigsqcup \Delta_0(4M) \xi_s = 
\bigsqcup \xi_s \Delta_0(4M) $ where $P(\xi_s) = \beta_s$.

Since $\Delta_0(4M)\eta\Delta_0(4M) = 
\Delta_0(4M)\xi^{-1}\Delta_0(4M)(\mat{p^2}{0}{0}{p^2},1)$,
it follows that $\Delta_0(4M)\eta\Delta_0(4M) = 
\bigsqcup \Delta_0(4M)\xi_s^{-1} (\mat{p^2}{0}{0}{p^2},1)$.

Thus for $f$, $g \in S_{k+1/2}(\Gamma_0(4M))$, we have
\begin{equation}\label{eq:eq2}
\begin{split}
\langle f|[\Delta_0(4M)\xi\Delta_0(4M)]_{k+1/2},\ g \rangle &= 
\langle p^{(2k-3)/2}\sum_s f|[\xi_s]_{k+1/2},\ g \rangle\\
= \langle f,\ p^{(2k-3)/2}\sum_s g|[\xi_s^{-1}]_{k+1/2}\rangle 
%= \langle f,\ g|[\Delta_0(4p)\xi^{-1}\Delta_0(4p)]_{k+1/2}\rangle
&=\langle f,\ g|[\Delta_0(4M)\eta\Delta_0(4M)]_{k+1/2}\rangle
\end{split}
\end{equation}
as elements of type $(aI, 1)$ belongs to the center of 
$\mathcal{G}$ and acts trivially via slash operator.

Using triangular decomposition we check that 
$\Gamma_0(4M)\mat{p^2}{0}{0}{1}\Gamma_0(4M) = 
\bigsqcup_{s=0}^{p^2-1} \Gamma_0(4M)\mat{p^2}{0}{0}{1}\mat{1}{0}{-4Ms}{1}$ 
and so 
\[\Delta_0(4M)\eta\Delta_0(4M) = \bigsqcup_{s=0}^{p^2-1} 
\Delta_0(4M)\ \eta\ (\mat{1}{0}{-4Ms}{1}, (-4Msz+1)^{1/2})\]
\[= \bigsqcup_{s=0}^{p^2-1} \Delta_0(4M) 
(\mat{p^2}{0}{-4Ms}{1}, (-4(M/p)sz +p^{-1})^{1/2}).\]
Thus it follows from Proposition~\ref{prop:rel3} $(4)$ that  
$g|[\Delta_0(4M)\eta\Delta_0(4M)]_{k+1/2} 
=\kro{-1}{p}^{k} p^{(2k-3)/2} q(\mathcal{T}_{-1})(g)$. Also we 
noted that
$f|[\Delta_0(4M)\xi\Delta_0(4M)]_{k+1/2} = 
\kro{-1}{p}^{k} p^{(2k-3)/2} q(\mathcal{T}_{1})(f)$. 
Thus by equation~\eqref{eq:eq2} we obtain the following
\begin{prop}
The operator $q(\mathcal{T}_{-1})$ is adjoint of $q(\mathcal{T}_{1})$ and 
consequently $\widetilde{Q}_p$ is self-adjoint with respect to 
Petersson inner product.
\end{prop}

\subsection{Translating elements of $2$-adic Hecke algebra and Kohnen's plus space}
Following Niwa and Kohnen's work, Loke and Savin gave 
an interpretation of Kohnen's plus space at level $4$ in terms 
of certain elements in the $2$-adic Hecke algebra described previously.
In this subsection we shall describe Kohnen's plus space at level $4M$ for 
$M$ odd in a similar way. 

Let $\chi$ be the trivial character modulo $4$, 
thus $\chi_0 = \kro{-1}{\cdot}^k$. Let $\gamma$ be a character of $M_2$
such that
$\gamma((-I,1))= - i^{2k+1}$. Let $\zeta_8 := \gamma((w(1),1))$.
Then, for any $k_0 = \mat{a}{b}{c}{d}\in K_0^2(4)$ 
we have $\tilde{\epsilon}_2(k_0)\chi_{0,2}(d) = \gamma((k_0,1))$. 
\begin{prop}(Loke-Savin~\cite{L-S})
For $\mathcal{T}_1$, $\mathcal{U}_1 \in H(\overline{K_0^2(4)}, \gamma)$ and
$f \in S_{k+1/2}(\Gamma_0(4),\chi)$,
\begin{enumerate}
 \item $\displaystyle{q(\mathcal{T}_1)(f)(z)= 2^{(3-2k)/2} U_4(f)(z)}$. 
 \item $\displaystyle{q(\mathcal{U}_1)(f)(z)= \kro{2}{2k+1} W_4(f)(z)}$ where 
 the operator $W_4$ is given by $W_4(f)(z) = (-2iz)^{-k-1/2}f(-1/4z)$.
 \end{enumerate}
\end{prop}
Niwa~\cite{Niwa} considered operator $R= W_4U_4$ on $S_{k+1/2}(\Gamma_0(4),\chi)$, 
proved that it is self-adjoint and
that $(R-\alpha_1)(R-\alpha_2)=0$ where $\alpha_1 = \kro{2}{2k+1}2^k$, 
$\alpha_2=-\frac{\alpha_1}{2}$.
Kohnen~\cite{Kohnen1} defined his plus space $S^{+}(4)$ at level $4$  to be the 
$\alpha_1$-eigenspace of $R$ in $S_{k+1/2}(\Gamma_0(4))$. It 
follows from the above proposition that $S^{+}(4)$ is the $2$-eigenspace 
of $q(\mathcal{U}_1)q(\mathcal{T}_1)/\sqrt{2}$ and hence 
that of $q(\mathcal{U}_2)/\sqrt{2}$.

In the case of level $4M$ with $M$ odd and $\chi$ a trivial 
character modulo $4M$, Kohnen~\cite{Kohnen2} defines a classical 
operator $Q$ on $S_{k+1/2}(\Gamma_0(4M))$ in order to obtain 
his plus space. The operator $Q$ is defined by
\[Q := [\Delta_0(4M, \chi) \rho \Delta_0(4M, \chi)]\ \ \text{where}\  
\rho = (\mat{4}{1}{0}{4}, e^{\pi i/4}).\]
By \cite[Proposition 1]{Kohnen2} $Q$ is self-adjoint and satisfies 
$(Q-\alpha)(Q-\beta)=0$ where 
$\alpha = (-1)^{[(k+1)/2]}2\sqrt{2}$, $\beta = -\alpha/2$, and the 
plus space 
$S_{k+1/2}^{+}(4M)$ is precisely
the $\alpha$-eigenspace of $Q$. 

In our work we need that Kohnen's plus space at level $4M$ is the 
$2$-eigenspace of the product $q(\mathcal{U}_1)q(\mathcal{T}_1)/\sqrt{2}$. 
%(just as in the case $M=1$).
We prove this below. 

\begin{prop} Let $f \in S_{k+1/2}(\Gamma_0(4M))$ with $M$ odd. Then we have 
\[Q(f) = \kro{2}{2k+1}q(\mathcal{U}_2)(f) = 
\kro{2}{2k+1}q(\mathcal{U}_1)q(\mathcal{T}_1)(f).\]
Consequently $S_{k+1/2}^{+}(4M)$ is the 
$2$-eigenspace of $q(\mathcal{U}_1)q(\mathcal{T}_1)/\sqrt{2}$.
\end{prop}
\begin{proof}
Following \cite[Proposition 1]{Kohnen2} we can write
\begin{equation*}
\begin{split}
Q(f) &= \sum_{s=0}^{4} f|[\rho]_{k+1/2}
|[\mat{1}{0}{4Ms}{1}, (4Msz+1)^{1/2}]_{k+1/2}\\
&=e^{-(2k+1)\pi i/4} \sum_{s=0}^{4}
f|[\mat{4+4Ms}{1}{16Ms}{4}, (4Msz+1)^{1/2}]_{k+1/2}.
\end{split}
\end{equation*} and it's adjoint 
\begin{equation*}
\begin{split}
\tilde{Q}(f) &= \sum_{s=0}^{4} f|[\mat{4}{-1}{0}{4},e^{-\pi i/4} ]_{k+1/2}
|[\mat{1}{0}{4Ms}{1}, (4Msz+1)^{1/2}]_{k+1/2}\\
&=e^{(2k+1)\pi i/4} \sum_{s=0}^{4}
f|[\mat{4-4Ms}{-1}{16Ms}{4}, (4Msz+1)^{1/2}]_{k+1/2}.
\end{split}
\end{equation*}
Since $Q$ is self adjoint, $Q=\tilde{Q}$.

We now compute $q(\mathcal{U}_2)(f)$.
Let $\tilde{g}_\infty \in \DSL_2(\R)$ be such that $\tilde{g}_\infty i=z$. 
Using $K_0^2(4) w(2^{-2}) K_0^2(4) = \bigcup_{s\in \Z/ 4\Z}
y(4M(1-s))w(2^{-2})K_0^2(4)$, we get 
\[\mathcal{U}_2(\Phi_f)(\tilde{g}_\infty) 
=\overline{\zeta_8}
\sum_{s=0}^{3}\Phi_f(\tilde{g}_\infty (y(4M(1-s)),1)(w(2^{-2}),1)).\]
Take $A_s = \mat{1-\kro{-1}{M}Ms}{-\kro{-1}{M}/4}{4Ms}{1} \in \SL_2(\Q)$, 
so $s_{\Q}(A_s) = (A_s, 1)$. 
The $\infty$-component of 
\[s_{\Q}(A_s)\ \tilde{g}_\infty\  (y(4M(1-s)),1)(w(2^{-2}),1))\] 
is $(A_s,1)\tilde{g}_\infty$,  for a prime $q$ such that $(q,2M)=1$ 
the $q$-component is 
$(A_s, 1) \in \SL_2(\Z_q)\times\{1\}$, 
for an odd prime $p$ such that $p^b\|M$, the $p$-component 
is $(A_s, 1) \in K_0^p(p^b)\times\{1\}$ while
the $2$-component is 
\[(A_s,1)
(y(4M(1-s)),1)(w(2^{-2}),1)) = 
(\mat{\kro{-1}{M}}{\frac{1-M\kro{-1}{M}}{4}}{-4}{M},1).\]
Since $M$ is odd, it is clear that $\frac{1-M\kro{-1}{M}}{4}\in \Z_2$ and 
so the $2$-component is in $K_0^2(4)\times\{1\}$.
The $p$-component acts trivially while  
the $2$-component acts by 
$(\tilde{\gamma_2}(M))^{-1}(-1,M)_2\chi_{0,2}(M)=:\omega_M$. 
Hence
\begin{equation*}
\begin{split}
q(\mathcal{U}_2)(f)(z) 
&=\overline{\zeta_8}\ \omega_M
\sum_{s=0}^{3}f(A_sz)J(A_s,z)^{-2k-1} \\
&=\overline{\zeta_8}\ \omega_M
\sum_{s=0}^{3}f\left(\frac{(4-4M\kro{-1}{M}sz)-\kro{-1}{M}}{16Msz+4}\right)(4Msz+1)^{-k-1/2}
\end{split}
\end{equation*}
We note that $e^{(2k+1)\pi i/4} = \kro{2}{2k+1}\overline{\zeta_8}$.
Thus when $M\equiv 1 \pmod{4}$ since $\omega_M =1$, comparing the 
expression of $\tilde{Q}$ and $q(\mathcal{U}_2)$ we see that 
$\tilde{Q}(f)=\kro{2}{2k+1}q(\mathcal{U}_2)(f)$. In the case 
$M\equiv 3 \pmod{4}$ we get that
$\omega_M = -i(-1)^k$, so 
$\kro{2}{2k+1}\overline{\zeta_8}\omega_M = e^{-(2k+1)\pi i/4}$ and 
consequently $Q(f)=\kro{2}{2k+1}q(\mathcal{U}_2)(f)$. Since by
Theorem~\ref{thm:LS},
$\mathcal{U}_2 = \mathcal{U}_1 * \mathcal{T}_1$ we get that 
$Q(f) = \kro{2}{2k+1}q(\mathcal{U}_1)q(\mathcal{T}_1)f$.
Hence we are done.

The last statement follows since $(-1)^{[(k+1)/2]}= \kro{2}{2k+1}$.
\end{proof}

As before we can translate 
$\mathcal{T}_1,\ \mathcal{U}_1,\ \mathcal{U}_0 \in H(\ov{K_0^2(4)}, \gamma)$
to classical operators on $S_{k+1/2}(\Gamma_0(4M))$ .
\begin{prop}For $f \in S_{k+1/2}(\Gamma_0(4M))$, 
\begin{enumerate}
 \item $\displaystyle{q(\mathcal{T}_1)(f)(z)= 2^{(3-2k)/2} U_4(f)(z)}$. 
 \item $\displaystyle{q(\mathcal{U}_1)(f)(z)= 
 \overline{\zeta_8}\kro{2}{M}\kro{-1}{M}^{k+3/2}f|[W,\phi_W(z)]_{k+1/2}(z)}$
 where\\ $W=\mat{4n}{m}{4M}{2}$ with $a,b\in \Z$ such that $2n-mM=1$ and 
 $\phi_W(z)=(2Mz+1)^{1/2}$.
 \item $\displaystyle{q(\mathcal{U}_0)(f)(z)= 
 \overline{\zeta_8}\kro{-1}{M}^{k+3/2}\sum_{s=0}^{3}f|[A_s,\phi_{A_s}(z)]_{k+1/2}(z)}$
 where\\ $A_s=\mat{n}{-ns+m}{M}{-Ms+4}$ with $m,n\in\Z$ such that  $4n-mM=1$ and 
 $\phi_W(z)=(Mz+4-Ms)^{1/2}$.
 \end{enumerate}
\end{prop}  

Define $\widetilde{Q}_2:= q(\mathcal{U}_0)/\sqrt{2} = 
q(\mathcal{T}_1)q(\mathcal{U}_1)/\sqrt{2}$ and 
$\widetilde{W}_4:= q(\mathcal{U}_1)$. Further define
$\widetilde{Q}'_2$ to be 
the conjugate of $\widetilde{Q}_2$ by $\widetilde{W}_4$, thus 
$\widetilde{Q}'_2 = q(\mathcal{U}_2)/\sqrt{2} = q(\mathcal{U}_1)q(\mathcal{T}_1)/\sqrt{2}$.
Hence the Kohnen's plus space at level $4M$ is the $2$-eigenspace of 
$\widetilde{Q}'_2$. Further $\widetilde{Q}_2$ and $\widetilde{Q}'_2$ 
are self-adjoint with respect to Petersson inner product.

\section{Eigenvalues of $U_p$} 
%Let $N$ be a positive integer and $p$ be a prime strictly 
%dividing $N$. We compute the eigenvalues of the 
%operator $U_p$ on the space of old forms of weight 
%$2k$ and level $\Gamma_0(N)$.
For every positive integer $n$ and a modular form $F$, let 
$F_n(z):= V_nF(z) = F(nz)$.
Let $M$ be a positive integer such that $p\nmid M$.
If $F\in S_{2k}(\Gamma_0(M))$, then by well-known 
action of $T_p$ and $U_p$ we have 
\begin{equation} \label{E:U2}
U_p(F)(z)=T_p(F)(z)-p^{2k-1}F_p(z) .
\end{equation}
Assume that $F\in S_{2k}(\Gamma_0(M))$ is a primitive Hecke eigenform 
and $a_p$ is the $p$-th Fourier coefficient of $F$. 
Then $T_p(F)=a_p F$. It is known that $a_p$ is
real and by the Ramanujan conjecture proved by Deligne 
we have that $|a_p|\leq 2\cdot p^{(2k-1)/2}$. 
\begin{lem} \label{L:action}
\begin{flushleft}
(a) If $(p,n)=1$ then $U_p(F_n)=a_p F_n -p^{2k+1}F_{np}$. \\
(b) If $p\mid n$ then $U_p(F_n)=F_{n/p}$.
\end{flushleft}
\end{lem}
\begin{proof}
It is well known that 
if $(p,n)=1$ then $V(n)T_p(F)=T_pV(n)F$. 
Hence using (\ref{E:U2}) and that $F$ is a primitive 
Hecke eigenform we get that
\begin{equation*}
\begin{split}
U_p(F_n)&=T_p(F_n)-p^{2k-1}F_{np}
%=T_pV(n)F-p^{2k-1}F_{np}
=V(n)T_p(F)-p^{2k-1}F_{np}\\
&=V(n)a_pF-p^{2k-1}F_{np}
=a_pF_n-p^{2k-1}F_{np}.
\end{split}
\end{equation*}
For $(b)$ write $n=m p$. Then
\[U_p(F_n)(z)=\frac{1}{p}\sum_{k=0}^{p-1}F_{mp}\left( \frac{z+k}{p}\right)
=\frac{1}{p}\sum_{k=0}^{p-1}F_{m}(z+k)=F_{n/p}(z).\]
\end{proof}

Thus $U_p$ stabilizes the two dimensional subspace spanned by $F_n$ and $F_{np}$ for $(p,n)=1$. We will compute the eigenvalues of $U_p$ on this space.
If $G=\lambda F_n+\beta F_{np}$ is an eigenfunction of $U_p$ then it follows from (2) that $\lambda \not=0$.
Hence we can assume that $\lambda =1$. We have
\[U_p(F_n+\beta F_{np})=(a_p+\beta) F_n-p^{2k-1}F_{np}\]
It is clear from above that $\beta$ cannot be zero and
that $G$ is an eigenfunction if and only if 
$a_p+\beta=-p^{2k-1}/\beta$ with eigenvalue
$a_p+\beta$. Hence $\beta^{2}+a_p\beta +p^{2k-1}=0$ and we have
\[\beta=\frac{-a_p\pm \sqrt{a_p^2-4p^{2k-1}}}{2}.\]
The eigenvalues of $U_p$ on the subspace $\langle F_n, F_{np} \rangle$ are
\[a_p+\beta= \frac{a_p\pm \sqrt{a_p^2-4p^{2k-1}}}{2}.\]
\begin{prop}\label{T:eigenvalue}
If an eigenvalue $\lambda$ of $(U_p)^2$ on the 
two dimensional subspace spanned by $F_n$ and $F_{np}$ 
is real then $\lambda = \pm p^{2k-1}$.
\end{prop}
\begin{proof}
Using the Ramanujan conjecture
we can see that the eigenvalues of $U_p$ 
are real or purely imaginary if and only if
$a_p=\pm 2p^{k-1/2}$ or $a_p=0$. 
In those cases the eigenvalue of $(U_p)^2$ 
are precisely $\pm p^{2k-1}$.
\end{proof}

\section{The minus space of half-integral weight forms}
Let $M$ be odd and square-free. In this section we shall 
define the minus space $S_{k+1/2}^{-}(4M)$ of 
weight $k+1/2$ and level $4M$ and show that there is an 
Hecke algebra isomorphism between$S_{k+1/2}^{-}(4M)$ 
and $S_{2k}^{\text{new}}(\Gamma_0(2M))$.

We shall first start with defining the minus space at level $4$ 
and then at level $4p$ for $p$ an odd prime. We can then 
extend our definition to level $4M$.
%where $M$ is as above.

\subsection{Minus space for $\Gamma_0(4)$}
We recall the following theorem of Niwa which was obtained 
by proving equality of traces of Hecke operators.
\begin{thm}(Niwa~\cite{Niwa})\label{T:Niwa}
%\begin{flushleft} 
Let $M$ be odd and square-free.
There exists an isomorphism of vector spaces
$\psi:S_{k+1/2}(\Gamma_0(4M))\to S_{2k}(\Gamma_0(2M))$ 
satisfying 
\[T_p (\psi(f))=\psi (T_{p^2}(f))\ \ \text{for all primes $p$ coprime to $2M$}.\] 
Moreover if 
$f\in S_{k+1/2}(\Gamma_0(4))$ then we further have $U_2(\psi(f))=\psi(U_4(f))$.
%, \quad T_p(\psi(f))=\psi(T_{p^2}(f))\text{ for all odd primes $p$}.\]
%\end{flushleft}
\end{thm}
\begin{comment}
We next consider an involution $\widetilde{W}_4$ on 
$S_{k+1/2}(\Gamma_0(4))$ and an operator $R$,
\[%U_4(f)(z) = \frac{1}{4} \sum_{s=0}^{3}f\left(\frac{z+s}{4}\right),\quad 
\widetilde{W}_4(f)(z) = 
(-2iz)^{-k-\frac{1}{2}}f\left(\frac{-1}{4z}\right), 
\quad R:= \widetilde{W}_4U_4.\] 
Niwa proved that 
$R$ is self-adjoint with respect to Petersson inner product and
$(R-\alpha_1)(R-\alpha_2)=0$ 
where $\alpha_1 = \kro{2}{2k+1}2^k$, $\alpha_2=-\frac{\alpha_1}{2}$.
Kohnen defined his plus space $S^{+}(4)$ to be the 
$\alpha_1$ eigenspace of $R$ in $S_{k+1/2}(\Gamma_0(4))$ and 
\end{comment}

We also recall Shimura lift~\cite{Shimura}: For $t$ a positive 
square-free integer, there is a linear map $\Sh_t: S_{k+1/2}(\Gamma_0(4M))\to S_{2k}(\Gamma_0(2M))$ given 
by \[\Sh_t\left(\sum_{n=1}^{\infty}a_nq^n\right) = 
\sum_{n=1}^{\infty}\left(\sum_{d \mid n} 
\kro{-1}{d}^k\kro{t}{d}d^{k-1}a\left(t\frac{n^2}{d^2}\right)\right)q^n.\]
We note the following observations \cite{Purkait}:\\
$(a)$ $\Sh_t$ need not be $1-1$ but $\Sh_t(f)=0$ for all square-free $t$ 
implies $f=0$.\\
$(b)$ $\Sh_t$ commutes with all Hecke operators, i.e.,  
$T_p (\Sh_t(f))=\Sh_t (T_{p^2}(f))$ for all primes $p$ coprime to $2M$
and $U_p (\Sh_t(f))=\Sh_t (U_{p^2}(f))$ for all primes $p$ dividing $2M$.

We need the following theorem of Kohnen.
\begin{thm}(Kohnen~\cite{Kohnen1})
\begin{enumerate}
 \item $\mathrm{dim}(S^{+}(4))=\mathrm{dim}(S_{2k}(\Gamma_0(1))$.
 \item $S^{+}(4)$ has a basis of eigenforms for all the operators 
 $T_{p^2}$, $p$ odd.
 \item If $f$ is such an eigenform then $\psi(f)$ is an old form and
$\psi(f)=\lambda F+\beta F_2$ where $F\in S_{2k}(\Gamma_0(1))$ is a 
primitive eigenform determined by the eigenvalues of $f$.
\end{enumerate}   
\end{thm}

Define $A^{+}(4) = \widetilde{W}_4S^{+}(4)$. 
We know that $S^{+}(4)$ is the $2$-eigenspace of $\widetilde{Q}'_2$, 
hence $A^{+}(4)$ is the $2$-eigenspace of $\widetilde{Q}_2$. 
Following the above theorem of Kohnen we have 
$\mathrm{dim}(A^{+}(4))=\mathrm{dim}(S_{2k}(\Gamma_0(1))$
and 
\begin{cor}\label{cor:A4}
%\begin{flushleft}
\begin{enumerate}
 \item $A^{+}(4)$ has a basis of eigenforms under $T_{p^2}$ for all
 $p$ odd.
 \item $\psi$ maps $A^{+}(4)$ into the space of old forms in 
 $S_{2k}(\Gamma_0(2))$.
\end{enumerate}
%\end{flushleft}
\end{cor}
\begin{proof}
Let $f\in S^{+}(4)$ be an eigenform under $T_{p^2}$ for all $p$ odd 
satisfying $T_{p^2}(f)=\lambda_p f$.
Since $\widetilde{W}_4$ commutes with all such $T_{p^2}$, we get that
$g=\widetilde{W}_4f \in A^{+}(4)$ is also an eigenform under all $T_{p^2}$
with eigenvalues $\lambda_p$. By Theorem~\ref{T:Niwa}, 
$\psi(f)$ and $\psi(g)$ are eigenforms in $S_{2k}(\Gamma_0(2))$ under 
all $T_p$ with the same set of eigenvalues $\lambda_p$. Since
$\psi(f)$ is an old form it follows from Atkin-Lehner~\cite{A-L} that 
$\psi(g)$ is also an old form (belonging to the same two 
dimensional subspace $\langle F, F_2 \rangle$). 
\end{proof}

\begin{prop}
$S^{+}(4)\bigcap A^{+}(4)=\{0\}$.
\end{prop}
\begin{proof}
Suppose there is a nonzero $f \in S^{+}(4)\bigcap A^{+}(4)$. 
We can assume that $f$ is an eigenform under $T_{p^2}$ for all $p$ odd 
(since $T_{p^2}$ stabilizes the intersection $S^{+}(4)\bigcap A^{+}(4)$).
Then $\widetilde{Q}_2(f) = 2f = \widetilde{Q}'_2(f)$. This implies that
$U_4\widetilde{W}_4(f) = 2^k f = \widetilde{W}_4U_4(f)$. 
Since $\widetilde{W}_4^2 =1$ we get
\[(U_4)^2(f) = 2^kU_4\widetilde{W}_4(f) 
%= 2^{2k}\kro{2}{2k+1}^2f 
= 2^{2k}f.\]
Applying $\psi$ to the above equation we get 
that $(U_2)^2(\psi(f))= 2^{2k}\psi(f)$. Now 
$\psi(f) \in \langle F, F_2 \rangle$ for some 
primitive form $F\in S_{2k}(\Gamma_0(1))$ 
and by Proposition~\ref{T:eigenvalue}, the eigenvalues of 
$(U_2)^2$ on this subspace are either non real
or $2^{2k-1}$. This is a contradiction.
\end{proof}

Define $S_{k+1/2}^{-}(4)$ to be the orthogonal 
complement of $S^{+}(4)\oplus A^{+}(4)$. Since 
$\widetilde{Q}_2$ and $\widetilde{Q}'_2$ are Hermitian 
it follows that
$S_{k+1/2}^{-}(4)$ is the common eigenspace with the 
eigenvalue $-1$ of the operators $\widetilde{Q}_2$ and $\widetilde{Q}'_2$. 
We shall write $S_{k+1/2}^{-}(4)$ simply by 
$S^{-}(4)$. So we have
\begin{equation} \label{E:directsum}
S_{k+1/2}(\Gamma_0(4))=S^{+}(4)\oplus A^{+}(4)\oplus S^{-}(4)
\end{equation}
\begin{thm}\label{thm:minus4}
$S^{-}(4)$ has a basis of eigenforms for all the operators $T_{p^2}$, $p$ odd; 
these eigenforms are also eigenfunctions under $U_{4}$. 
If two eigenforms in $S^{-}(4)$ share the same eigenvalues 
for all $T_{p^2}$ then they are a scalar multiple of each other. 
%$\psi$ maps $S^{-}(4)$ onto the new space in $S_{2k}(\Gamma_0(2))$. 
%That is, if we take a basis of primitive eigenforms for the new space 
%$S_{2k}(\Gamma_0(2))$ then there is a corresponding basis of
%$S^{-}(4)$ under $\psi^{-1}$ that shares all the eigenvalues. That is, 
$\psi$ induces a Hecke algebra isomorphism:
\[S_{k+1/2}^{-}(4) \cong S_{2k}^{\mathrm{new}}(\Gamma_0(2)).\]
\end{thm}
\begin{proof}
Since $\psi$ maps $S^{+}(4) \oplus A^{+}(4)$ into 
$S_{2k}^{\mathrm{old}}(\Gamma_0(2))$ and 
$\mathrm{dim} (S^{+}(4) \oplus A^{+}(4))\\
= 2 \mathrm{dim} (S_{2k}(\Gamma_0(1))) 
= \mathrm{dim} (S_{2k}^{\mathrm{old}}(\Gamma_0(2)))$,
we get that $\psi$ maps this direct sum onto 
$S_{2k}^{\mathrm{old}}(\Gamma_0(2))$.

Now $T_{p^2}$ commutes with $\widetilde{Q}_2$ and $\widetilde{Q}'_2$ 
for every odd prime $p$ so we get that $T_{p^2}$ stabilizes $S^{-}(4)$,
%and forms a commuting family of self-adjoint operator on $S^{-}(4)$, 
hence it has a basis of eigenforms for all $T_{p^2}$ with $p$ odd. 

If $f$ is such an eigenform then $F:=\psi(f)$ is an eigenform in 
$S_{2k}(\Gamma_0(2))$ under all $T_p$, $p$ odd. By Atkin-Lehner~\cite{A-L} $F$ 
is either an old form or a newform. 
Since $\psi$ is injective, it follows that 
%$\psi(f)$ can not be an old form, so 
$F$ must be a newform. So
$\psi$ maps the space $S^{-}(4)$ into the space
$S_{2k}^{\mathrm{new}}(\Gamma_0(2))$.
By equality of dimensions, we get that $\psi$ is an isomorphism of 
$S^{-}(4)$ onto $S_{2k}^{\mathrm{new}}(\Gamma_0(2))$. 
Consequently by \cite{A-L} an eigenform in 
$S^{-}(4)$ under all $T_{p^2}$ for $p$ odd is 
uniquely determined up to scalar multiplication. 

Further for such an eigenform $f$, by \cite[Theorem 3]{A-L},
$U_2 (F) = -2^{k-1}\lambda(2)F$ where 
$\lambda(2) = \pm 1$.
Thus $\psi(U_4(f)) = U_2 (F) \in 
S_{2k}^{\mathrm{new}}(\Gamma_0(2))$, so
$U_4(f)$ belongs to 
$S^{-}(4)$. Since $U_4$ commutes with $T_{p^2}$ for all 
$p$ odd, we get that $U_4(f)$ is an eigenform under all 
$T_{p^2}$ with the same eigenvalues as $f$ and hence is a 
scalar multiple of $f$. 
\end{proof}

\subsection{Minus space for $\Gamma_0(4p)$ for $p$ an odd prime}
%Let $p$ be an odd prime. 
In this subsection we need the involution 
$\widetilde{W}_{p^2}$, the operators $U_{p^2}$,  
$\widetilde{Q}_p$ and 
$\widetilde{Q}'_p = \widetilde{W}_{p^2}\widetilde{Q}_p\widetilde{W}_{p^2}$ 
on $S_{k+1/2}(\Gamma_0(4p))$ that we defined in Section~\ref{sec:trans}.

Consider the subspace $\mathcal{V}(1)$ of $S_{2k}(\Gamma_0(2p))$ 
coming from the old forms at level $1$, that is,  
\[\mathcal{V}(1)=S_{2k}(\Gamma_0(1))\oplus
V(2) S_{2k}(\Gamma_0(1)) \oplus
V(p) S_{2k}(\Gamma_0(1)) \oplus
V(2p) S_{2k}(\Gamma_0(1)).\]
We consider the eigenvalues of $(U_p)^2$ on $\mathcal{V}(1)$.
\begin{lem} \label{T:eigenvalues}
The operator $U_p$ stabilizes $\mathcal{V}(1)$. If an eigenvalue $\lambda$
of $(U_p)^2$ on this space is real then  $\lambda =\pm p^{2k-1}$.
\end{lem}
\begin{proof}
For a a primitive Hecke eigenform $F$ in $S_{2k}(\Gamma_0(1))$ 
consider the four dimensional subspace spanned by  $F,F_2,F_p,F_{2p}$. 
Then $\mathcal{V}(1)$ is a direct sum of such four dimensional subspaces.
By Lemma~\ref{L:action}, $U_p$ preserves the two dimensional subspaces
spanned by $F$ and $F_p$ and the two dimensional subspace spanned by
$F_2$ and $F_{2p}$. It follows by Proposition~\ref{T:eigenvalue}, 
that the eigenvalues of
$(U_p)^2$ on these two dimensional subspaces are either 
non real or $\pm p^{2k-1}$.
\end{proof}

Let $R:=S^{+}(4)\oplus A^{+}(4)$. Then we have
\begin{prop} \label{T:intersect}
$R\bigcap \widetilde{W}_{p^2}R=\{0\}$
\end{prop}
\begin{proof}
Let $f\ne0$ belongs to the intersection. 
We can again assume that $f$ is an eigenform under $T_{q^2}$ for all 
primes $q$ coprime to $2p$. 
Since by Corollary~\ref{cor:rel4}$(4)$,  
$S_{k+1/2}(\Gamma_0(4))$ is contained 
in the $p$ eigenspace of $\widetilde{Q}_p$ and so 
$\widetilde{W}_{p^2}S_{k+1/2}(\Gamma_0(4))$ is contained 
in the $p$-eigenspace of $\widetilde{Q}'_p$ we have
$\widetilde{Q}_p(f) = pf = \widetilde{Q}'_p(f)$.
Using $\widetilde{Q}_p = \kro{-1}{p}^k p^{1-k}U_{p^2}\widetilde{W}_{p^2}$ 
in the above equality we obtain  
\[(U_{p^2})^2(f) = p^{2k} f.\] 
Since $f\ne 0$, there 
exists a $t$ square-free such that the Shimura lift 
$\mathrm{Sh}_t(f)\ne0$. Applying this $\mathrm{Sh}_t$ to the 
above equation we get that 
$(U_p)^2(\mathrm{Sh}_t(f)) = p^{2k}\mathrm{Sh}_t(f)$. 
Since $\mathrm{Sh}_t$ commutes with all the Hecke operators we get that 
$\mathrm{Sh}_t(f) \in \mathcal{V}(1)$. But by 
Lemma~\ref{T:eigenvalues}, the eigenvalues of $(U_p)^2$ on 
$\mathcal{V}(1)$ are either non real or $p^{2k-1}$
leading to a contradiction.
\end{proof}

\begin{cor} \label{C:level1}
Niwa's map $\psi$ maps $R\oplus \widetilde{W}_{p^2}R$ isomorphically 
onto $\mathcal{V}(1)$.
\end{cor}
\begin{proof}
As before (see Corollary \ref{cor:A4}$(2)$) $\psi$ maps 
$R\oplus \widetilde{W}_{p^2}R$ into $\mathcal{V}(1)$. 
It follows from the equality of dimensions that the map 
is onto.
\end{proof}

Next we consider the following subspace of 
$S_{2k}(\Gamma_0(2p))$ coming from the old forms at level $2$,
\[
\mathcal{V}(2)=S^{\text{new}}_{2k}(\Gamma_0(2)) \oplus 
V(p)S^{\text{new}}_{2k}(\Gamma_0(2)).
\]
This space is a direct sum of two dimensional subspaces 
spanned by $F$ and $F_p$ where $F$ is a primitive 
Hecke eigenform in $S^{\text{new}}_{2k}(\Gamma_0(2))$. 
Using Proposition~\ref{T:eigenvalue} we have the following 
lemma.
\begin{lem}
If an eigenvalue $\lambda$
of $(U_p)^2$ on $\mathcal{V}(2)$ is real then  $\lambda =\pm p^{2k-1}$.
\end{lem}

Since (by Theorem~\ref{thm:minus4})
$\psi$ maps $S^{-}_{k+1/2}(4)$ isomorphically onto 
$S^{\text{new}}_{2k}(\Gamma_0(2))$,
it follows that $\psi$ maps $\widetilde{W}_{p^2}S^{-}_{k+1/2}(4)$ into the space
$\mathcal{V}(2)$.  The proof of the following is 
identical to that of Proposition~\ref{T:intersect}.
\begin{prop}
$S^{-}_{k+1/2}(4)\bigcap \widetilde{W}_{p^2}S^{-}_{k+1/2}(4)=\{0\}$.
\end{prop}

\begin{cor} \label{C:level2}
$\psi$ maps $S^{-}_{k+1/2}(4)\oplus \widetilde{W}_{p^2}S^{-}_{k+1/2}(4)$ 
isomorphically onto $\mathcal{V}(2)$.
\end{cor}

Finally, we consider the following subspace of 
$S_{2k}(\Gamma_0(2p))$ coming from the old forms at level $p$,
\[\mathcal{V}(p)=S^{\mathrm{new}}_{2k}(\Gamma_0(p)) 
\oplus V(2)S^{\mathrm{new}}_{2k}(\Gamma_0(p)). 
\]
This space is a direct sum of two dimensional subspaces 
spanned by $F$ and $F_2$ where $F$ is a primitive 
Hecke eigenform in $S^{\text{new}}_{2k}(\Gamma_0(p))$. 
We have
\begin{lem}
If an eigenvalue $\lambda$
of $(U_2)^2$ on $\mathcal{V}(p)$ is real then  $\lambda =\pm 2^{2k-1}$.
\end{lem}

Let $S^{+,\mathrm{new}}_{k+1/2}(4p)$ be the new space 
inside the plus space in $S_{k+1/2}(\Gamma_0(4p))$.
Kohnen \cite[Theorem 2]{Kohnen2} proved that $\psi$ maps
$S^{+,\mathrm{new}}_{k+1/2}(4p)$ into $\mathcal{V}(p)$ and the dimension of
$S^{+,\mathrm{new}}_{k+1/2}(4p)$ equals the dimension of 
$S^{\mathrm{new}}_{2k}(\Gamma_0(p))$. Then as before 
$\psi$ maps $\widetilde{W}_4S^{+,\mathrm{new}}_{k+1/2}(4)$ 
into the space
$\mathcal{V}(p)$ and we have the following proposition and corollary.
\begin{prop}
$S^{+,\mathrm{new}}_{k+1/2}(4p)\bigcap \widetilde{W}_4
S^{+,\mathrm{new}}_{k+1/2}(4p)=\{0\}$.
\end{prop}
\begin{cor}\label{C:levelp}
$\psi$ maps $S^{+,\mathrm{new}}_{k+1/2}(4p)\oplus 
\widetilde{W}_4S^{+,\mathrm{new}}_{k+1/2}(4p)$ isomorphically 
onto $\mathcal{V}(p)$.
\end{cor} 

We define the following subspace of $S_{k+1/2}(\Gamma_0(4p))$,
\[E=R\oplus \widetilde{W}_{p^2}R
\oplus S^{-}_{k+1/2}(4)\oplus \widetilde{W}_{p^2}S^{-}_{k+1/2}(4)
\oplus
S^{+,\mathrm{new}}_{k+1/2}(4p)\oplus \widetilde{W}_4S^{+,\mathrm{new}}_{k+1/2}(4p).
\]
By Corollary~\ref{C:level1}, \ref{C:level2}
and \ref{C:levelp}, we get that $\psi$ maps the space $E$ isomorphically
onto the old space $S_{2k}^{\mathrm{old}}(\Gamma_0(2p))$.
We define the minus space to be the orthogonal complement of $E$ under the Petersson
inner product. That is,
\[S^{-}_{k+1/2}(4p):=E^{\perp}.\]
\begin{thm}
$S^{-}_{k+1/2}(4p)$ has a basis of eigenforms for all the operators 
$T_{q^2}$, $q$ an odd prime different than $p$, uniquely determined up to 
a non-zero scalar multiplication. 
%These eigenforms are also eigenfunctions
%under $U_{p^2}$ and $U_{4}$ (see cor below).
$\psi$ maps the space  $S^{-}_{k+1/2}(4p)$ isomorphically to the 
space $S^{\mathrm{new}}_{2k}(\Gamma_0(2p))$.
\end{thm}
\begin{proof}
Since the operators $T_{q^2}$ with $(q,2p)=1$ stabilize the space 
$E$ and since they are self adjoint with respect to the Petersson 
inner product, it follows that they stabilize the space 
$S^{-}_{k+1/2}(4p)$, hence it has a basis of eigenforms for all such  
operators $T_{q^2}$. If $f$ is such an
eigenform then $\psi(f)\in S_{2k}(\Gamma_0(2p))$ is also an eigenform for all the
operators $T_q$, $(q,2p)=1$ and thus (by \cite{A-L})
$\psi(f)$ is either an old form or a newform. 
Since $\psi$ is injective and maps $E$ onto 
$S_{2k}^{\mathrm{old}}(\Gamma_0(2p))$, it follows that 
$\psi(f)$ is a newform. 
Thus $\psi$ maps the space $S^{-}_{k+1/2}(4p)$ into the space
$S_{2k}^{\mathrm{new}}(\Gamma_0(2p))$. 
By equality of dimensions, we get that $\psi$ maps the space 
$S^{-}_{k+1/2}(4p)$ isomorphically onto $S_{2k}^{\mathrm{new}}(\Gamma_0(2p))$. 
Consequently an eigenform in $S^{-}_{k+1/2}(4p)$ is uniquely determined up to 
a scalar multiplication. 
%Since $U_{4},\ U_{p^2}$ commutes with $T_{q^2}$ 
%it follows that an eigenform under all the $T_{q^2}$ 
%must also be an eigenfunction under $U_{4}$ and $U_{p^2}$.
\end{proof}

\begin{cor} \label{C:Weigenvalue}
Let $f\in S^{-}_{k+1/2}(4p)$ be a Hecke eigenform for 
all the operators $T_{q^2}$, $q$ prime and $(q,2p)=1$. Then 
$\widetilde{W}_{p^2}f=\beta(p) f,\ \widetilde{W}_4f= \beta(2)f$ where 
$\beta(p)=\pm 1,\ \beta(2)=\pm 1$.
\end{cor}
\begin{proof}
Let $g=\widetilde{W}_{p^2}f$. Since $\widetilde{W}_{p^2}$ commutes with 
all the operators $T_{q^2}$ for $(q,2p)=1$ 
we get that $g$ is an eigenform for all the operators 
$T_{q^2}$ with the same eigenvalues as $f$. 
Since $\psi(f)$ is a newform, it follows \cite{A-L} that 
$\psi(g)$ is a scalar multiple of $\psi(f)$.
Since $\psi$ is an isomorphism we get that 
$g$ is a scalar multiple of $f$. 
Since $\widetilde{W}_{p^2}$ is an involution we get
that the scalar is $\pm 1$. The same proof applies to $\widetilde{W}_4$.
\end{proof}

Let $f\in S^{-}_{k+1/2}(4p)$ be a Hecke eigenform for all the operators 
$T_{q^2}$ as above. It follows that 
$F:=\psi(f)$ is a Hecke eigenform in $S_{2k}^{\mathrm{new}}(\Gamma_0(2p))$ 
for all the operators $T_q$, $(q,2p)=1$. Since the Shimura lift $\Sh_t(f)$
is also an eigenform for all the operators $T_q$ with the same 
eigenvalues as $F$,
it follows from \cite{A-L}
that $\mathrm{Sh}_t(f)$ is a scalar multiple of $F$ (which could be zero).
Also, $U_p(F)=-p^{k-1}\lambda(p)F$ where $\lambda(p)=\pm 1$ and
$U_2(F)=-2^{k-1}\lambda(2)F$ where $\lambda(2)=\pm 1$.
\begin{prop} \label{T:Ueigenvalue}
Let $f\in S^{-}_{k+1/2}(4p)$ be a Hecke eigenform for all 
the operators $T_{q^2}$, $q$ prime and $(q,2p)=1$. Then
\[U_{p^2}(f)=-p^{k-1}\lambda(p)f,\quad U_{4}(f)=-2^{k-1}\lambda(2)f\]
where $\lambda(p)=\pm 1$ and $\lambda(2)=\pm 1$ are defined as above.
\end{prop}
\begin{proof}
Let $g=U_{p^2}f$. Then $\mathrm{Sh}_t(g)=U_p\mathrm{Sh}_t(f)=
-p^{k-1}\lambda(p)\mathrm{Sh}_t(f)$ for every positive square-free 
integer $t$.
It follows that $\mathrm{Sh}_t(g-p^{k-1}\lambda(p)f)=0$ 
for all such $t$ implying $g-p^{k-1}\lambda(p)f=0$
which is what we need. For the prime $2$, the proof is the same.
\end{proof}

\begin{prop} \label{T:Qeigenvalues}
Let $f\in S^{-}_{k+1/2}(4p)$. Then 
$\widetilde{Q}_p(f) = -f = \widetilde{Q}'_p(f)$ 
and 
$\widetilde{Q}_2(f) = -f = \widetilde{Q}'_2(f)$.
\end{prop}
\begin{proof}
Let $f\in S^{-}_{k+1/2}(4p)$ be a Hecke eigenform for all the operators 
$T_{q^2}$, $(q,2p)=1$. 
Since
$\widetilde{Q}_p = \kro{-1}{p}^k p^{1-k}U_{p^2}\widetilde{W}_{p^2}$ 
and 
$\widetilde{Q}_2 = 2^{1-k}U_{4}\widetilde{W}_{4}$
it follows from
Corollary~\ref{C:Weigenvalue} and 
Proposition~\ref{T:Ueigenvalue} that $f$ is an 
eigenform for the operators 
$\widetilde{Q}_p$, $\widetilde{Q}'_p$, 
$\widetilde{Q}_2$ and $\widetilde{Q}'_2$ with eigenvalues $\pm 1$. 
However, the eigenvalues of $\widetilde{Q}_p$, $\widetilde{Q}'_p$ are 
$p$ and $-1$ and the eigenvalues of $\widetilde{Q}_2$ and 
$\widetilde{Q}'_2$ are $2$ and $-1$ hence the eigenvalues have 
to be $-1$. Since $S^{-}_{k+1/2}(4p)$
has a basis of such eigenforms we get the result.
\end{proof}

\begin{thm}
Let $f\in S_{k+1/2}(4p)$. Then $f\in S^{-}_{k+1/2}(4p)$ if and only if
$\widetilde{Q}_p(f) = -f = \widetilde{Q}'_p(f)$ 
and 
$\widetilde{Q}_2(f) = -f = \widetilde{Q}'_2(f)$.
\end{thm}
\begin{proof}
If $f\in S^{-}_{k+1/2}(4p)$ then by Proposition~\ref{T:Qeigenvalues} 
the conditions hold. Now assume that $f\in S_{k+1/2}(4p)$ is in the 
intersection of $-1$ eigenspaces of $\widetilde{Q}_p$, 
$\widetilde{Q}'_p$, $\widetilde{Q}_2$ and $\widetilde{Q}'_2$.
For every $g \in S_{k+1/2}(\Gamma_0(4))$
we have $\widetilde{Q}_p(g)=pg$. Since $\widetilde{Q}_p$ is 
self-adjoint,
\[-\langle f, g \rangle = \langle \widetilde{Q}_pf, g \rangle
= \langle f, \widetilde{Q}_pg \rangle = p\langle f, g\rangle\]
implying $\langle f, g \rangle=0$. Thus $f$ is orthogonal to 
$R \oplus S^{-}_{k+1/2}(4)$. For every 
$g \in \widetilde{W}_{p^2}S_{k+1/2}(4)$ we have 
$\widetilde{Q}'_p(g)=pg$ and the same argument shows that 
$\langle f, g \rangle=0$ implying $f$ is orthogonal to 
$\widetilde{W}_{p^2}(R \oplus S^{-}_{k+1/2}(4))$. 
Since Kohnen's plus space is the $2$-eigenspace of 
$\widetilde{Q}'_2$, for $g \in S^{+,\mathrm{new}}_{k+1/2}(4p)$
we have $\widetilde{Q}'_2(g)=2g$, consequently for 
$g \in \widetilde{W}_{4}S^{+,\mathrm{new}}_{k+1/2}(4p)$ we have
$\widetilde{Q}_2(g)=2g$. Hence $\langle f, g \rangle=0$ for such $g$,
that is, $f$ is orthogonal to $S^{+,\mathrm{new}}_{k+1/2}(4p)
\oplus \widetilde{W}_{4}S^{+,\mathrm{new}}_{k+1/2}(4p)$.
It follows that $f \in S^{-}_{k+1/2}(4p)$.
\end{proof}

\subsection{Minus space for $\Gamma_0(4M)$ for $M$ odd and square-free}
Let $M\ne 1$ be an odd and square-free natural number. 
Write $M=p_1p_2\cdots  p_k$. 
For each $i=1, \ldots k$ define $M_{i}=M/p_i$. 
Since $S_{k+1/2}(\Gamma_0(4M_i))$ is contained in the $p_i$-eigenspace 
of $\widetilde{Q}_{p_i}$ (Corollary~\ref{cor:rel4}$(4)$),
following the proof of Proposition~\ref{T:intersect} 
we obtain that
\begin{prop}
$S_{k+1/2}(\Gamma_0(4M_i)) \bigcap 
\widetilde{W}_{p_i^2}S_{k+1/2}(\Gamma_0(4M_i))=\{0\}$.
\end{prop}

\begin{cor} \label{C:imap}
The Niwa map $\psi : S_{k+1/2}(\Gamma_0(4M)) \rightarrow 
S_{2k}(\Gamma_0(2M))$ maps
$S_{k+1/2}(\Gamma_0(4M_i))\oplus 
\widetilde{W}_{p_i^2}S_{k+1/2}(\Gamma_0(4M_i))$
isomorphically onto $S_{2k}(\Gamma_0(2M_i))\oplus 
V(p_i) S_{2k}(\Gamma_0(2M_i))$.
\end{cor}

Let $S^{+,\mathrm{new}}_{k+1/2}(4M)$ be the new space 
inside the Kohnen plus subspace of $S_{k+1/2}(4M)$. Then 
similarly we have
\begin{prop}
$S^{+,\mathrm{new}}_{k+1/2}(4M) \bigcap 
\widetilde{W}_4S^{+,\mathrm{new}}_{k+1/2}(4M)=\{0\}$.
\end{prop}

\begin{cor} \label{C:2map}
$\psi$ maps
$S^{+,\mathrm{new}}_{k+1/2}(4M) \oplus 
\widetilde{W}_4S^{+,\mathrm{new}}_{k+1/2}(4M)$ isomorphically onto 
$S^{\mathrm{new}}_{2k}(\Gamma_0(M)) \oplus 
V(2) S^{\mathrm{new}}_{2k}(\Gamma_0(M))$.
\end{cor}

We let $B_i=S_{k+1/2}(\Gamma_0(4M_i)) \oplus 
\widetilde{W}_{p_i^2}S_{k+1/2}(\Gamma_0(4M_i))$, 
$i=1,\ldots k$. Define
\[E=\sum_{i=1}^k B_i \oplus S^{+,\text{new}}_{k+1/2}(4M) \oplus 
\widetilde{W}_4S^{+,\mathrm{new}}_{k+1/2}(4M).\]
\begin{prop}
Under $\psi$ the space $E$ maps isomorphically onto the old space 
$S_{2k}^{\mathrm{old}}(\Gamma_0(2M))$.
\end{prop}
\begin{proof}
This follows from Corollary~\ref{C:imap} and~\ref{C:2map} and from the decomposition
\[
S_{2k}^{\mathrm{old}}(\Gamma_0(2M))=\left( 
\sum_{i=1}^{k}S_{2k}(\Gamma_0(2M_i))\oplus 
V(p_i)S_{2k}(\Gamma_0(2M_i)) \right) \oplus\]
\[\left(S_{2k}^{\mathrm{new}}(\Gamma_0(M))\oplus 
V(2)S_{2k}^{\mathrm{new}}(\Gamma_0(M))\right).\]
\end{proof}

We now define the minus space to be the orthogonal complement of $E$, 
\[S^{-}_{k+1/2}(4M):=E^{\perp}\]
Let $f\in S^{-}_{k+1/2}(4M)$ be a Hecke eigenform for all the 
operators $T_{q^2}$ where $q$ is an odd prime
satisfying $(q,M)=1$. Let $\psi(f)=F$. 
The proof of the following results is identical to the proofs 
in the previous subsections.
\begin{prop}
$F$ is up to a scalar a primitive Hecke eigenform in $S_{2k}^{new}(2M)$.
\end{prop}

\begin{thm}
The space $S^{-}_{k+1/2}(4M)$ has a basis of eigenforms for 
all the operators $T_{q^2}$ where $q$ is an odd prime
satisfying $(q,M)=1$. Under
$\psi$, the space  $S^{-}_{k+1/2}(4M)$ maps isomorphically onto the 
space $S^{\mathrm{new}}_{2k}(\Gamma_0(2M))$. If 
two forms in $S^{-}_{k+1/2}(4M)$ have the same eigenvalues for
all the operators $T_{q^2}$, $(q,2M)=1$, then they are same up to a scalar 
factor.
\end{thm}

Let $f\in S^{-}_{k+1/2}(4M)$ be a Hecke eigenform for all the 
operators $T_{q^2}$, $(q,2M)=1$. Then $\psi(f)=F$ is a 
Hecke eigenform in $S^{\mathrm{new}}_{2k}(\Gamma_0(2M))$ for 
all operators $T_q$, $(q,2M)=1$. By \cite{A-L}, for all primes $p$ 
such that $p|M$, $U_p(F)=-p^{k-1}\lambda(p)F$ where 
$\lambda(p)=\pm 1$ and $U_2(F)=-2^{k-1}\lambda(2)F$ 
where $\lambda(2)=\pm 1$.
\begin{prop}
Let $f\in S^{-}_{k+1/2}(4M)$ be a Hecke eigenform for all 
the operators $T_{q^2}$, $q$ prime, $(q,2M)=1$. Then for all 
primes $p$ such that $p|M$
\[U_{p^2}(f)=-p^{k-1}\lambda(p)f\ \text{ and }\ U_{4}(f)=-2^{k-1}\lambda(2)f\]
where $\lambda(p)=\pm 1$ and $\lambda(2)=\pm 1$ are defined as above.
\end{prop}

Following \cite[Theorem 1.9]{Shimura} we have 
\begin{cor}
Let $f = \sum_{n=0}^{\infty} a_nq^n \in S^{-}_{k+1/2}(4M)$ be a Hecke 
eigenform for all Hecke operators, i.e, $T_{q^2}(f) = \omega_q f$ for 
all primes $(q, 2M)=1$ and $U_{p^2}(f)=\omega_pf$ for all primes $p\mid 2M$. 
Let $F = \sum_{n=0}^{\infty} A_nq^n \in S^{\mathrm{new}}_{2k}(\Gamma_0(2M))$ 
be the unique normalized primitive form determined by $f$, i.e. $A_p = \omega_p$ 
for all primes $p$. Then for a fundamental discriminant $D$ 
such that $(-1)^k D>0$,
\[\mathrm{L}\left(s-k+1,\kro{D}{\cdot}\right)
\sum_{n=1}^{\infty} a_{|D|n^2}\ n^{-s} 
=a(|D|)\sum_{n=1}^{\infty} A_n n^{-s}.\]
\end{cor}

We finally give the characterization of our minus space. The proofs 
of the following proposition and theorem are as before. 
\begin{prop}
Let $f\in S^{-}_{k+1/2}(4M)$. Then for every prime $p$ dividing $M$ we have
$\widetilde{Q}_p(f) = -f = \widetilde{Q}'_p(f)$ 
and 
$\widetilde{Q}_2(f) = -f = \widetilde{Q}'_2(f)$.
\end{prop}

\begin{thm}
Let $f\in S_{k+1/2}(4M)$. Then $f\in S^{-}_{k+1/2}(4M)$ if and only if
$\widetilde{Q}_p(f) = -f = \widetilde{Q}'_p(f)$ for every prime 
$p$ dividing $M$ and 
$\widetilde{Q}_2(f) = -f = \widetilde{Q}'_2(f)$.
\end{thm}

\begin{remark}
We note that although Kohnen has a strong multiplicity one
property on his plus new space $S^{+,\text{new}}_{k+1/2}(4M)$, 
this no longer holds for the full space.   
That is, for a primitive Hecke eigenform 
$F \in S^{\text{new}}_{2k}(\Gamma_0(M))$ there are two 
different Hecke eigenforms in $S_{k+1/2}(\Gamma_0(4M))$ that 
correspond to it; one in the Kohnen plus new space and one 
in the "conjugate" space $\widetilde{W}_4S^{+,\mathrm{new}}_{k+1/2}(4M)$.
On the other hand the minus space $S^{-}_{k+1/2}(4M)$ has 
strong multiplicity one property in the full space, in particular, 
for a primitive Hecke eigenform 
$F \in S^{\text{new}}_{2k}(\Gamma_0(2M))$ there is precisely 
one Hecke eigenform in $S_{k+1/2}(\Gamma_0(4M))$ that corresponds 
to it and this form is in the minus space.
\end{remark}

We summarize the above remark as follows.
\begin{thm}
Let $f_1$ and $f_2$ be Hecke eigenforms in $S_{k+1/2}(\Gamma_0(4M))$ 
with the same eigenvalues for all $T_{q^2}$, $(q,2M)=1$. If $f_1$ 
is a nonzero element of the minus space $S^{-}_{k+1/2}(4M)$ then 
$f_2$ is a scalar multiple of $f_1$.
\end{thm}

\subsection{Some examples}
We complete this section by giving two examples. 
\begin{ex}
The space $S_{3/2}(\Gamma_0(28))$ is one dimensional and is
spanned by 
\[f= q - q^2 - q^4 + q^7 + q^8 - q^9 + q^{14} - 2q^{15} + q^{16} + 
3q^{18} - 2q^{21} + \ldots\]
By Shimura decomposition (for notation see \cite{Purkait2})
\[S_{3/2}(\Gamma_0(28)) = 
\bigoplus_{\substack{F \in S^{\mathrm{new}}_{2}(\Gamma_0(M))\\ \text{prim.,}\ M \mid 14}} 
S_{3/2}(28,F) = S_{3/2}(28,F_{14})\]
as there are no primitive Hecke eigenforms of weight $2$ at level $1$, $2$, $7$ and 
$F_{14}
%= q - q^2 - 2q^3 + q^4 + 2q^6 + q^7 - q^8 + q^9 + \ldots 
\in S^{\mathrm{new}}_{2}(\Gamma_0(14))$ is the only primitive Hecke eigenform at 
level $14$. 
In particular, we have $S_{3/2}^{+}(28)=\{0\}$ and 
$S_{3/2}^{-}(28)=S_{3/2}(\Gamma_0(28))= \langle f \rangle$.
\end{ex}

\begin{comment}
\begin{ex}
The space $S_{7/2}(\Gamma_0(12))$ is three dimensional and is
spanned by
\begin{equation*}
\begin{split}
g_1 &= q - 4q^4 - 6q^6 + 8q^7 + 9q^9 + 4q^{10} + 12q^{12} - 20q^{13} 
- 24q^{15} + \ldots \\
g_2 &= q^3 - 2q^4 + 2q^7 - 2q^{12} - 6q^{15} + 12q^{16} - 10q^{19} + 12q^{24} + 
\ldots\\
g_3 &= q^2 - q^3 - 4q^5 + 3q^6 + 4q^8 + 2q^{11} - 4q^{12} - 8q^{14} + 16q^{17} 
+ \ldots\\
\end{split}
\end{equation*}
We have two primitive forms of weight $6$ and level dividing $6$, namely 
$F_3$ of level $3$ and $F_6$ of level $6$. Using algorithm in \cite{Purkait2} 
we have 
\[S_{7/2}(\Gamma_0(12)) = S_{7/2}(12,F_3) \oplus S_{7/2}(12,F_6) = 
\langle g_1,\ g_2 \rangle \oplus \langle g_3 \rangle.\]
Note that $g_2$ is in the plus space, so we have 
$S_{7/2}^{+}(12)= \langle g_2 \rangle$. The minus space  
$S_{7/2}^{-}(12)= S_{7/2}(12,F_6) = \langle g_3 \rangle$.
\end{ex}
\end{comment}

\begin{ex}
The space $S_{17/2}(\Gamma_0(12))$ is $13$-dimensional. 
We first give Shimura decomposition of this 
space \cite{Purkait2}. We have seven primitive Hecke eigenforms of weight $16$ 
and level dividing $6$, namely, $F_1$ of level $1$, $G_2$ of level $2$, 
$H_3$, $K_3$ of level $3$ each and $L_6$, $M_6$, $N_6$ each of level $6$. 
We have 
\begin{equation*}
\begin{split}
S_{17/2}(\Gamma_0(12)) = S_{17/2} & (12, F_1) \oplus S_{17/2}(12, G_2)
\oplus S_{17/2}(12, H_3)\oplus S_{17/2}(12, K_3)\\
&\oplus S_{17/2}(12, L_6)
\oplus S_{17/2}(12, M_6)\oplus S_{17/2}(12, N_6),  
\end{split}
\end{equation*}
where $S_{17/2}(12, F_1)$ is $4$-dimensional space spanned by
\begin{equation*}
\begin{split}
f_1 &= q + 88q^4 + 513q^9 + 3024q^{12} - 4368q^{13} - 13760q^{16} +
    33264q^{21} + \cdots\\
%+ 63504q^{24} - 26015q^{25} - 77952q^{28}\\
f_2 &= 11q^2 + 64q^4 + 232q^7 - 1408q^8 + 4608q^9 + 190q^{10} 
-6578q^{11} + \cdots\\
f_3 &= 9q^3 - 64q^4 + 189q^6 - 232q^7 - 190q^{10} + 1152q^{12} 
 -3328q^{13} +\cdots\\
%   + 1380q^{15} - 1536q^{16} + 3300q^{18} - 31594/9q^{19} +
 %   2816q^{21}\\    
f_4 &= q^5 - 11q^8 + 18q^9 - 9q^{12} - 116q^{17} + 344q^{20} - 99q^{21} -
    189q^{24} +\cdots; 
\end{split}
\end{equation*}
the space $S_{17/2}(12, G_2)$ is $2$-dimensional and is spanned by
\begin{equation*}
\begin{split}
g_1 &= q + 21q^3 - 128q^4 - 609q^6 + 3192q^7 + 5313q^9 - 12810q^{10}
     + \cdots \\
   % - 2688q^{12} + 2580q^{15} + 16384q^{16} + 67500q^{18} - 148386q^{19} \\
g_2 &= 3q^2 + 7q^3 - 203q^6 - 384q^8 - 416q^9 + 2706q^{11} -
    896q^{12} + \cdots;
    %980q^{14} + 860/3q^{15} - 3136q^{17} + 939q^{18} - 
\end{split}
\end{equation*}
the space $S_{17/2}(12, H_3)$ is $2$-dimensional spanned by
\begin{equation*}
\begin{split}
h_1 &= q^5 + 7q^8 - 27q^{12} - 80q^{17} + 56q^{20} + 189q^{21} + 
81q^{24} + 231q^{29} + \cdots\\
h_2 &= 7q^2 - 27q^3 + 81q^6 - 896q^8 + 854q^{11} + 3456q^{12} -
    1876q^{14} +\cdots;
\end{split}
\end{equation*}
the space $S_{17/2}(12, K_3)$ is $2$-dimensional spanned by
\begin{equation*}
\begin{split}
k_1 &= q - 362q^4 - 2187q^9 - 11826q^{12} + 19032q^{13} + 51940q^{16} 
+\cdots \\
k_2 &= 1971q^3 + 13184q^4 + 31266q^6 - 20158q^7 + 271340q^{10}
    + \cdots;
\end{split}
\end{equation*}
the last three summands are $1$-dimensional each with 
$S_{17/2}(12, L_6)$ spanned by
\[l_1 =
13q^2 + 129q^3 + 736q^5 + 1323q^6 + 1664q^8 + 5918q^{11} +16512q^{12}
+ \cdots; \]
the space $S_{17/2}(12, M_6)$ spanned by
\[m_1 = q^3 - 18q^6 - 42q^7 - 12q^{10} + 128q^{12} + 384q^{13} - 126q^{15} -
    1074q^{19} + 896q^{21} + \cdots;\]
and the space $S_{17/2}(12, N_6)$ spanned by
\[n_1 = 16q - 1539q^3 - 2048q^4 - 5994q^6 - 50178q^7 - 34992q^9 -
    2460q^{10} + \cdots.\]
We can also check (using bound in \cite{Kumar}) 
that the Kohnen's plus space $S_{17/2}^{+}(12)$ is $4$-dimensional.  
Indeed 
\[S_{17/2}^{+}(12) = \langle f_1,\ f_4,\ h_1,\ k_1\rangle = 
 S_{17/2}^{+}(4) \oplus \widetilde{W}_9S_{17/2}^{+}(4) \oplus 
 S_{17/2}^{+,\mathrm{new}}(12)
\]
with 
$S_{17/2}^{+}(4)=\langle f_1- 336f_4 \rangle$ and 
$S_{17/2}^{+,\mathrm{new}}(12)=\langle h_1,\ k_1 \rangle$. 
Further from the Shimura decomposition of $S_{17/2}(\Gamma_0(4))$ 
we get $A_{17/2}^{+}(4) = \langle f_2 + 2f_3 -256f_4 \rangle$ and 
$S_{17/2}^{-}(4) = \langle g_1 + 3g_2 \rangle$. Thus we have  
\[S_{17/2}(12, F_1) = R \oplus \widetilde{W}_9R\ \text{where}\ 
R = S_{17/2}^{+}(4) \oplus A_{17/2}^{+}(4),\]
\[S_{17/2}(12, G_2) = S_{17/2}^{-}(4) \oplus \widetilde{W}_9S_{17/2}^{-}(4),\]
\[S_{17/2}(12, H_3) \oplus S_{17/2}(12, K_3) = 
S_{17/2}^{+,\mathrm{new}}(12) \oplus \widetilde{W}_4S_{17/2}^{+,\mathrm{new}}(12)\]
and 
\[S_{17/2}(12, L_6) \oplus 
S_{17/2}(12, M_6) \oplus S_{17/2}(12, N_6) = \langle l_1,\ m_1,\ n_1\rangle 
= S_{17/2}^{-}(12).\]
\end{ex}

\begin{remark}
$(i)$ In general,
$S_{k+1/2}^{-}(4M) = \bigoplus_{F} S_{k+1/2}(4M,F)$
where $F$ runs through all primitive Hecke eigenforms of weight $2k$ and level $2M$.\\
$(ii)$ The Kohnen plus space is given by well-known Fourier coefficient 
condition. But we do not expect any such Fourier 
coefficient condition for forms in our minus space. 
This is also evident from the above examples. 
\end{remark}

\appendix
\section{Some observations on cocycle multiplication}

Let $p$ denotes any prime.
In this appendix we note down some useful observations on multiplication 
in $\DSL_2(Q_p)$ by cocycle $\sigma_p$.

Recall the Hilbert symbol $\hs{\cdot,\cdot}$ 
defined on $\Qs_p \times \Qs_p$. For an odd prime $p$ it can be given by the 
formula: For $a$, $b$ coprime to $p$,
\[ \hs{p^s a, p^t b} = \kro{-1}{p}^{st} \kro{a}{p}^t \kro{b}{p}^s.\]
Thus $\hs{p,p} = \kro{-1}{p}$ and $\hs{-p,u} = \hs{p,u} = \kro{u}{p}$ where 
$u$ is a unit in $\Z_p$. For an even prime, if $a$, $b$ are odd 
\[ \hst{2^s a, 2^t b} = (-1)^{\frac{(a-1)(b-1)}{4}} \kro{2}{|a|}^t \kro{2}{|b|}^s.\]

Let $A = \left(\begin{matrix} a & b\\ c & d \end{matrix}\right) \in \SL_2(\Q_p)$.
For $(A,\epsilon_1) \in \DSL_2(\Q_p)$, $(A,\ \epsilon_1)^{-1} 
= (A^{-1},\ \epsilon_1 \sigma_p(A,A^{-1}))$ where
\begin{itemize}
 \item[(i)] If $c = 0$ then $\sigma_p(A, A^{-1}) = \hs{a,a} = \hs{d,d}$.
 \item[(ii)] If $c \ne 0$ and $\ord_p(c)$ is even then $\sigma_p(A, A^{-1})=1$.
 \item[(iii)] If $c \ne 0$ and $\ord_p(c)$ is odd then 
 \[\sigma_p(A, A^{-1}) = 
  \begin{cases}
   \hs{c,d} \hs{-c, a} & \text{if $d \ne 0,\ a \ne 0$}\\
   \hs{c,d}  & \text{if $d \ne 0,\ a = 0$}\\
   \hs{-c, a} & \text{if $d = 0,\ a \ne 0$}\\
   1 & \text{if $d = 0,\ a = 0$}.\\
  \end{cases}
  \]
\end{itemize}
In particular if $A \in \{x(p^n),\ y(p^n),\ w(p^n)\}_{n \in \Z}$ 
then $\sigma_p(A, A^{-1})=1$. For $A = h(p^n)$ with $n\in \Z$, 
if $p=2$ then $\sigma_p(A, A^{-1})=1$, however if $p$ is 
an odd prime then
\[\sigma_p(A, A^{-1})=\begin{cases}
                       1 & \text{if $n$ even},\\
                       \kro{-1}{p} & \text{else}.
                      \end{cases}
\]

Let $(A,\epsilon_1),\  (B,\ \epsilon_2) \in \DSL_2(\Q_p)$.
The following lemmas can be easily obtained using cocycle formula. 
\begin{lem} We have 
$[(B,\ \epsilon_2)^{-1},\ (A,\ \epsilon_1)^{-1}] = (B^{-1}A^{-1}BA, \xi)$
where
$\xi = \sigma_p(A, A^{-1})\sigma_p(B, B^{-1}) \sigma_p(B,A) 
\sigma_p(A^{-1},BA) \sigma_p(B^{-1},A^{-1}BA)$.
%\begin{equation*}
%\begin{split}
%& \qquad [(B,\ \epsilon_2)^{-1},\ (A,\ \epsilon_1)^{-1}] 
%=\\&(B^{-1}A^{-1}BA,\ \sigma_p(A, A^{-1})\sigma_p(B, B^{-1}) \sigma_p(B,A) 
%\sigma_p(A^{-1},BA) \sigma_p(B^{-1},A^{-1}BA)).
%\end{split}
%\end{equation*}
\end{lem}
\begin{comment}
\begin{proof}
Let $\beta_1 = \epsilon_1 \sigma_p(A,A^{-1})$ and $\beta_2 = \epsilon_2 \sigma_p(B,B^{-1})$. Then we have
\begin{equation*}
\begin{split}
&[(B,\ \epsilon_2)^{-1},\ (A,\ \epsilon_1)^{-1}]\\
&= [(B^{-1},\ \epsilon_2 \sigma_p( B, B^{-1})),\ (A^{-1},\ \epsilon_1 \sigma_p( A, A^{-1}))] \\
&= (B^{-1},\ \beta_2)(A^{-1},\ \beta_1)(B,\ \epsilon_2)(A,\ \epsilon_1) \\
&= (B^{-1},\ \beta_2)(A^{-1},\ \beta_1)(BA,\ \epsilon_2 \epsilon_1\sigma_p(B,A)) \\
&= (B^{-1},\ \beta_2)(A^{-1}BA,\ \beta_1 \epsilon_2 \epsilon_1\sigma_p(B,A) \sigma_p(A^{-1},BA)) \\
&= (B^{-1}A^{-1}BA,\ \beta_2 \beta_1 \epsilon_2 \epsilon_1\sigma_p(B,A) \sigma_p(A^{-1},BA) \sigma_p(B^{-1},A^{-1}BA))\\
&= (B^{-1}A^{-1}BA,\ \sigma_p(A, A^{-1})\sigma_p(B, B^{-1}) \sigma_p(B,A) \sigma_p(A^{-1},BA) \sigma_p(B^{-1},A^{-1}BA)).
\end{split}
\end{equation*}
\end{proof}
\end{comment}

\begin{lem}\label{lem1:sigma}
The $\sigma_p$-factor ($\xi$ factor above) of 
$[(B,\ \epsilon_2)^{-1},\ (A,\ \epsilon_1)^{-1}]$ equals
\begin{equation*}
\begin{split}
&\hsb{\tau(B),\ \tau(B^{-1})} \hsb{\tau(A),\ \tau(A^{-1})} \\ 
& \hsb{\tau(BA)\tau(B),\ \tau(BA)\tau(A)} \hsb{\tau(A^{-1}BA)\tau(A^{-1}),\ \tau(A^{-1}BA)\tau(BA)} \\
& \hsb{\tau(B^{-1}A^{-1}BA)\tau(B^{-1}),\ \tau(B^{-1}A^{-1}BA)\tau(A^{-1}BA)}s_p(B^{-1}A^{-1}BA).
\end{split}
\end{equation*}
\end{lem}
\begin{comment}
\begin{proof}
Following above lemma and cocycle formula we get that 
the $\sigma_p$-factor of $[(B,\ \epsilon_2)^{-1},\ (A,\ \epsilon_1)^{-1}]$ equals
\begin{equation*}
\begin{split}
\sigma_p(B,\ & B^{-1})\sigma_p(A,\ A^{-1}) \sigma_p(B,\ A) \sigma_p(A^{-1},\ BA) \sigma_p(B^{-1},\ A^{-1}BA) \\
= \ &\sigma_p(B,\ B^{-1}) \sigma_p(A,\ A^{-1}) \hsb{\tau(BA)\tau(B),\ \tau(BA)\tau(A)}s(B)s(A)s(BA) \\
& \hsb{\tau(A^{-1}BA)\tau(A^{-1}),\ \tau(A^{-1}BA)\tau(BA)}s(A^{-1})s(BA)s(A^{-1}BA) \\
& \hsb{\tau(B^{-1}A^{-1}BA)\tau(B^{-1}),\ \tau(B^{-1}A^{-1}BA)\tau(A^{-1}BA)}\\
& s(B^{-1})s(A^{-1}BA)s(B^{-1}A^{-1}BA) \\ 
= \ &\hsb{\tau(B),\ \tau(B^{-1})} \hsb{\tau(A),\ \tau(A^{-1})} \\ 
& \hsb{\tau(BA)\tau(B),\ \tau(BA)\tau(A)} \hsb{\tau(A^{-1}BA)\tau(A^{-1}),\ \tau(A^{-1}BA)\tau(BA)} \\
& \hsb{\tau(B^{-1}A^{-1}BA)\tau(B^{-1}),\ \tau(B^{-1}A^{-1}BA)\tau(A^{-1}BA)}s(B^{-1}A^{-1}BA). 
\end{split}
\end{equation*}
\end{proof}
\end{comment}

In the proofs for checking support of our local Hecke algebra (section 3)
we required the following lemma.
\begin{lem}\label{lem2:sigma}
Let $A = \mat{a}{b}{c}{d} \in \SL_2(\Q_p)$. Then 
\begin{enumerate}
\item[(a)] If $B = x(s)$ where $s \ne 0$, then $\sigma_p$-factor is 
 \[\begin{cases}
  \hs{-sc^2,\ 1-cds} & \text{if $sc^2(1-cds) \ne 0$ and 
  $\ord_p(s)$ is odd,}\\
   1 & \text{ else}.
    \end{cases}\]
 \item[(b)] If $B = h(u)$ where $u \ne \pm 1$, then $\sigma_p$-factor is 
 \[\begin{cases}
  \hs{ac(1-u^2),\ 1+(1-u^2)bc} &\parbox[t]{.4\textwidth}{if $ac(1-u^2)(1+(1-u^2)bc) \ne 0$ 
  and $\ord_p(ac(1-u^2))$ is odd,}\\
   1 & \text{ else}.
 \end{cases}\]
 \item[(c)] If $B = y(t)$ where $t \ne 0$, then $\sigma_p$-factor is 
 \[\begin{cases}
  \hs{(a^2-1)t+abt^2, 1+ abt +b^2t^2} &\parbox[t]{.49\textwidth}
  {if $((a^2-1)t+abt^2)(1+ abt +b^2t^2) \ne 0$ 
  and $\ord_p((a^2-1)t+abt^2)$ is odd,}\\
   1 & \text{ else}.
   \end{cases}\]
\end{enumerate}
In each of the above cases the $\sigma_p$-factor is simply $s_p(B^{-1}A^{-1}BA)$.
\end{lem}
\begin{proof}
For $(a)$ let $B = x(s)$ where $s \ne 0$. Then we have
 \[BA= \mat{ a+sc }{ b+sd }{c}{d}, \quad A^{-1}BA= \mat{ 1 + cds }{ sd^2 }{ -sc^2 }{ 1-cds },\]
 \[B^{-1}A^{-1}BA= \mat{ 1 + cds + s^2c^2 }{ sd^2-s+cds^2 }{ -sc^2 }{ 1-cds }. \]
It is easy to see that $\hsb{\tau(B),\ \tau(B^{-1})} = 1$ and 
that $\hsb{\tau(A),\ \tau(A^{-1})}=$ 
\[ =  
 \hsb{\tau(A^{-1}BA)\tau(A^{-1}),\ \tau(A^{-1}BA)\tau(BA)} 
 = \begin{cases}
    1 & \text{ if $c \ne 0$}\\
    \hs{d,a} & \text{ else}.
    \end{cases}
\]
Further one can check that $\hsb{\tau(BA)\tau(B), \tau(BA)\tau(A)}= 1$ 
and also
\[\hsb{\tau(B^{-1}A^{-1}BA)\tau(B^{-1}),\ \tau(B^{-1}A^{-1}BA)\tau(A^{-1}BA)}= 1.\]
Finally we have $s_p(B^{-1}A^{-1}BA)$
\[= \begin{cases}
\hs{-sc^2,\ 1-cds} &\parbox[t]{.5\textwidth}{if $sc^2(1-cds) \ne 0$ and $ord_p(s)$ is odd,}\\
1 & \text{ else}.
\end{cases}\]
By using Lemma~\ref{lem1:sigma}, multiplying all the above terms 
we get the required $\sigma_p$-factor.

For $(b)$ we proceed similarly.  Let $B = h(u)$ where $u \ne \pm 1$. Then
\[BA= \mat{ua}{ub}{u^{-1}c}{u^{-1}d}, \quad A^{-1}BA= \mat{ uad - u^{-1}bc }{ bd(u-u^{-1}) }{ ac(u^{-1}-u) }{ u^{-1}ad - ubc },\]
\[B^{-1}A^{-1}BA= \mat{ 1 + (1-u^{-2})bc }{ bd(1-u^{-2}) }{ ac(1-u^2) }{ 1+ (1-u^2)bc }. \]
 
We have $\hsb{\tau(B),\ \tau(B^{-1})} = \hs{u,u^{-1}}$. Also $\hsb{\tau(A),\ \tau(A^{-1})}= 1$ if $c \ne 0$ and $\hs{d,a}$ else.
We check that
\[\hsb{\tau(BA)\tau(B),\ \tau(BA)\tau(A)} = \begin{cases}
                                             \hs{c,u^{-1}} & \text{ if $c \ne 0$}\\
                                             \hs{d,u^{-1}} & \text{ else},
                                            \end{cases}
\]
$\hsb{\tau(A^{-1}BA)\tau(A^{-1}),\ \tau(A^{-1}BA)\tau(BA)}$
\[=\begin{cases}
                                             \hs{-a(u^{-1}-u),u^{-1}} & \text{ if $ac \ne 0$}\\
                                             \hs{bu,-b} & \text{ if $a=0$ and $c \ne 0$}\\
                                             \hs{du^{-1},a}& \text{ if $a \ne 0$ and $c = 0$},
                                            \end{cases}
\]
$\hsb{\tau(B^{-1}A^{-1}BA)\tau(B^{-1}),\ \tau(B^{-1}A^{-1}BA)\tau(A^{-1}BA)}$
\[=\begin{cases}
                                             \hs{ac(u^{-1}-u),u^{-1}} & \text{ if $ac \ne 0$}\\
                                             \hs{bc,u} = \hs{1,u} & \text{ if $a=0$ and $c \ne 0$}\\
                                             \hs{-ad,u}& \text{ if $a \ne 0$ and $c = 0$},
                                             \end{cases} 
\]
and $s_p(B^{-1}A^{-1}BA)=$ 
\[\begin{cases}
  \hs{ac(1-u^2),\ 1+(1-u^2)bc} &\parbox[t]{.4\textwidth}{if $ac(1-u^2)(1+(1-u^2)bc) \ne 0$ and $\ord_p(ac(1-u^2))$ is odd,}\\
   1 & \text{ else}.
\end{cases}
\]
Again by multiplying all the above terms we get the required $\sigma_p$-factor. 

For $(c)$, let $B = y(t)$ where $t \ne 0$. Then 
\[BA= \mat{ a }{ b }{at+c}{bt+d}, \quad A^{-1}BA= \mat{ 1 -abt }{ -b^2 t }{ a^2 t }{ 1+abt },\]
\[B^{-1}A^{-1}BA= \mat{ 1 -abt }{ -b^2 t }{(a^2-1)t+abt^2} {1+ abt +b^2t^2}. \]
 
 %Let $z=(a^2-1)t+abt^2$ and $v=1+ abt +b^2t^2$. 
As before,  
$\hsb{\tau(B),\ \tau(B^{-1})} = \hs{t,-t} =1$, and $\hsb{\tau(A),\ \tau(A^{-1})}= 1$ if $c \ne 0$ and $\hs{d,a}$ else.
One can compute (using $ad-bc =1$ in the Hilbert symbol calculations) that
\[\hsb{\tau(BA)\tau(B),\ \tau(BA)\tau(A)} = \begin{cases}
                                             \hs{t(at+c),-ct} & \text{ if $a \ne -c/t$ and $c \ne 0$}\\
                                             \hs{-c,a} & \text{ if $a = -c/t$ and $c \ne 0$}\\
                                             \hs{a,-dt} & \text{ if $c=0$},
                                            \end{cases}
\]
$\hsb{\tau(A^{-1}BA)\tau(A^{-1}),\ \tau(A^{-1}BA)\tau(BA)} = $
\[=\begin{cases}
                                            \hs{t(at+c),-ct} & \text{ if $a \ne -c/t$ and $c \ne 0$ and $a \ne 0$}\\
                                              1 & \text{ if $a \ne -c/t$ and $c \ne 0$ and $a = 0$}\\
                                             \hs{-c, a} & \text{ if $a = -c/t$ and $c \ne 0$}\\
                                            \hs{a,at}& \text{ if $c = 0$}.
                                            \end{cases}
\]
All the above factors clearly multiplies to $1$. Also it turns out that
\[\hsb{\tau(B^{-1}A^{-1}BA)\tau(B^{-1}),\ \tau(B^{-1}A^{-1}BA)\tau(A^{-1}BA)}=1,\] 
so we get the required $\sigma_p$-factor.
\end{proof}

We also note the triangular decomposition of $\ov{K_0^p(p^n)}$.
\begin{lem}\label{lem:char}
%Let $S$ denotes $\ov{K_0^p(p^n)}$. Then 
We have a triangular decomposition 
\[\ov{K_0^p(p^n)}= N^{\ov{K_0^p(p^n)}} T^{\ov{K_0^p(p^n)}} 
\bar{N}^{\ov{K_0^p(p^n)}}.\]
More precisely for $(A, \epsilon)=(\mat{a}{b}{c}{d}, \epsilon) \in \ov{K_0^p(p^n)}$, 
\[(A, \epsilon) = (x(s),1)(h(u),1)(y(t),1)(I, \epsilon\delta) \]
where 
\[u=d^{-1},\ s={d^{-1}b},\ t={d^{-1}c},\]
and  
\[\delta = \begin{cases}
            1 & c=0\\
            \hs{d,-1} & c \ne 0,\ \ord_p(c) \text{ is odd }\\
            \hs{-c,d} & c \ne 0,\ \ord_p(c) \text{ is even}.
           \end{cases}
\] 
\end{lem}
\begin{proof}
Clearly 
\[\mat{a}{b}{c}{d} = \mat{1}{bd^{-1}}{0}{1} 
\mat{d^{-1}}{0}{0}{d} \mat{1}{0}{cd^{-1}}{1}.\]
Let $u=d^{-1},\ s={bd^{-1}},\ t={cd^{-1}}$.
Since 
\[x(s)h(u)y(t)= \mat{u}{su^{-1}}{0}{u^{-1}}\mat{1}{0}{t}{1} = \mat{u+su^{-1}t}{su^{-1}}{tu^{-1}}{u^{-1}},\]
we get that 
\[(x(s),1)(h(u),1)(y(t),1)= (x(s)h(u)y(t), \delta) = (A, \delta)\]
where 
\[\delta = \sigma(x(s),h(u))\sigma(x(s)h(u),y(t))= \begin{cases}
                                                    1 & t=0\\
                                                    \hs{u,-1} & t \ne 0,\ \ord_p(t) \text{ is odd }\\
                                                    \hs{t,u}  & t \ne 0,\ \ord_p(t) \text{ is even}.
                                                   \end{cases}\]
Substituting $u$, $s$, $t$ in terms of $b$, $c$, $d$ we get 
$\delta$ as in the statement. 
%Clearly $x(bd^{-1}),\ h(d^{-1}),\ y(cd^{-1}) \in K_0^p(p^n)$. 
%Thus $(A, \epsilon) = (x(s),1)(h(u),1)(y(t),1)(I, \epsilon\delta) \in 
%N^{S} T^{S} \bar{N}^{S}$.
\end{proof}


\begin{thebibliography}{}

\bibitem{A-L} A.\ O.\ L.\ Atkin and J.\ Lehner,
{\em Hecke operators on $\Gamma_0(m)$},
Math.\ Ann.\ {\bf 185} (1970), 134--160.

\bibitem{B-P} E.\ M.\ Baruch and S.\ Purkait,
{\em Hecke algebras, new vectors and newforms on 
$\Gamma_0(m)$}, on arXiv. https://arxiv.org/abs/1503.02767

\bibitem{Gelbart} S.\ Gelbart,
{\em Weil's representation and the spectrum of the metaplectic group},
Lecture Notes in Math.\ {\bf 530}, Springer, Berlin, 1976.

\bibitem{Kohnen1} W.\ Kohnen,
{\em Modular forms of half-integral weight on $\Gamma_0(4)$}, 
Math.\ Ann.\ {\bf 248} (1980), 249--266. 

\bibitem{Kohnen2} W.\ Kohnen,
{\em Newforms of half-integral weight},
J.\ Reine Angew.\ Math.\ {\bf 333} (1982), 32--72.

\bibitem{Kumar}N.\ Kumar, S.\ Purkait,
{\em A note on the Fourier coefficients of half-integral weight modular forms}, 
Arch.\ Math.\ (Basel) {\bf 102} (2014), no. 04, 369--378.

\bibitem{L-S} H.\ Y.\ Loke and G.\ Savin,
{\em Representations of the two-fold central
extension of $\SL_2(\Q_2)$},
Pacific J.\ Math.\ {\bf 247} (2010), 435--454.

\bibitem{Niwa} S.\ Niwa,
{\em On Shimura's trace formula}, 
Nagoya Math.\ J.\ {\bf 66} (1977), 183--202.

\bibitem{Purkait} S.\ Purkait, 
{\em On Shimura's decomposition}, 
Int.\ J.\ Number Theory {\bf 9} (2013), 1431-1445.

\bibitem{Purkait2} S.\ Purkait, 
{\em Hecke operators in half-integral weight}, 
J.\ Th\'{e}or.\ Nombres Bordeaux {\bf 26} (2014), 233--251.

\bibitem{Shimura} G.\ Shimura,
{\em On modular forms of half integral weight}, 
Ann.\ of Math.\ {\bf 97} (1973), 440--481.

\bibitem{Ueda} M.\ Ueda, 
{\em On twisting operators and newforms of half-integral weight},
Nagoya Math J.\ {\bf 131} (1993), 135--205.

\bibitem{Waldspurger} J.-L.\ Waldspurger,
{\em Sur les coefficients de Fourier des formes modulaires de 
poids demi-entier}, J.\ Math.\ Pures Appl.\ (9) {\bf 60} (1981), 375--484. 
\end{thebibliography}
\end{document}